\documentclass[12pt,CJK]{article}
\usepackage{amsfonts,amsmath,amssymb,amsthm,mathrsfs}
\usepackage{latexsym}
\usepackage{graphicx}
\usepackage{indentfirst}
\usepackage{color}
\usepackage{fancyhdr}
\usepackage[colorlinks,linktocpage,linkcolor=blue]{hyperref}
\usepackage{titlesec}

\usepackage{tikz}
\usetikzlibrary{chains} 

\usetikzlibrary{shapes.geometric}
\usetikzlibrary{arrows.meta,arrows}

\titleformat*{\subsection}{\centering}

\theoremstyle{plain}                       
\newtheorem{lemma}{Lemma}[section]
\newtheorem{theorem}[lemma]{Theorem}
\newtheorem{corollary}[lemma]{Corollary}
\newtheorem{remark}[lemma]{Remark}
\newtheorem{definition}[lemma]{Definition}

\theoremstyle{remark}

\numberwithin{equation}{section}

\def\Xint#1{\mathchoice
  {\XXint\displaystyle\textstyle{#1}}%
  {\XXint\textstyle\scriptstyle{#1}}%
  {\XXint\scriptstyle\scriptscriptstyle{#1}}%
  {\XXint\scriptscriptstyle\scriptscriptstyle{#1}}%
  \!\int}
\def\XXint#1#2#3{{\setbox0=\hbox{$#1{#2#3}{\int}$}
  \vcenter{\hbox{$#2#3$}}\kern-.5\wd0}}

\def\dashint{\Xint-}

\begin{document}
\allowdisplaybreaks
\pagestyle{myheadings}
\markboth{$~$ \hfill {\rm Q. Xu,} \hfill $~$}
{$~$ \hfill {\rm  } \hfill$~$}

\author{
Li Wang,\quad Qiang Xu,\quad
Peihao Zhao\\
}

%


\title{\textbf{Suboptimal Error Estimates for Linear Elasticity
Systems on Perforated Domains}}
\maketitle
\begin{abstract}
In the present work, we established almost-sharp error
estimates for linear elasticity systems in periodically
perforated domains. The first result was
$L^{\frac{2d}{d-1-\tau}}$-error estimates
$O\big(\varepsilon^{1-\frac{\tau}{2}}\big)$ with
$0<\tau<1$ for a bounded smooth
domain. It followed from
weighted Hardy-Sobolev's inequalities
and a suboptimal error estimate
for the square function of
the first-order approximating corrector (which was
earliest investigated by C. Kenig, F. Lin, Z. Shen \cite{KLS}
under additional regularity assumption on coefficients).
The new approach  relied on the weighted quenched
Calder\'on-Zygmund estimate (initially appeared in
A. Gloria, S. Neukamm, F. Otto's work
\cite{Gloria_Neukamm_Otto_2015} for a quantitative
stochastic homogenization theory).
The second effort was $L^2$-error estimates
$O\big(\varepsilon^{\frac{5}{6}}
\ln^{\frac{2}{3}}(1/\varepsilon)\big)$ for a Lipschitz domain,
followed from a new duality scheme
coupled with interpolation inequalities.
Also, we developed a new weighted extension theorem
for perforated domains, and
a real method imposed by Z. Shen
\cite{S3} played a fundamental role in the whole project.
\end{abstract}

\tableofcontents

\section{Introduction}

\subsection{Hypothesises and Main Results}
In this paper, we are aimed to establish sharp convergence rates for linear elasticity systems in periodically perforated domains.
Consider the operator
\begin{equation}
\mathcal{L_{\varepsilon}}:=-\nabla\cdot A_{\varepsilon}(x)\nabla =
-\frac{\partial}{\partial x_{i}}\bigg\{a_{ij}^{\alpha\beta}
\big(\frac{x}{\varepsilon}\big)
\frac{\partial}{\partial x_{j}}\bigg\},
\quad x\in\varepsilon\omega,~\varepsilon > 0,
\end{equation}
where $A_{\varepsilon}(x):=A(x/\varepsilon)$
and $A(y)=\{a^{\alpha\beta}_{ij}(y)
\}_{1\leq i,j,\alpha,\beta\leq d}$ for $y\in\omega$
with $d\geq 2$,
and $\omega\subset  \mathbb{R}^d$ is an unbounded Lipschitz
domain with 1-periodic structure.
We
denote the $\varepsilon$-homothetic set $\{x\in\mathbb{R}^d:x/\varepsilon\in\omega\}$ by
$\varepsilon\omega$.
For a bounded domain $\Omega\subset\mathbb{R}^d$,
we study the following mixed boundary value problem:
\begin{eqnarray}\label{pde:1.1}
\left\{\begin{aligned}
\mathcal{L}_{\varepsilon} (u_\varepsilon)
&= F &\qquad &\text{in}~~\Omega_{\epsilon}, \\
\sigma_{\varepsilon}( u_{\varepsilon})&=0
&\qquad &\text{on}~~S_{\epsilon},\\
 u_\varepsilon &= g &\qquad & \text{on}~~\Gamma_{\varepsilon},
\end{aligned}\right.
\end{eqnarray}
in which $\sigma_{\varepsilon}:=n\cdot A(\cdot/\varepsilon)\nabla$
and $n$ denotes the outward unit normal to $\Omega_{\varepsilon}$,
and $\Omega_{\varepsilon}:=\Omega\cap\varepsilon\omega$;
$S_{\varepsilon}:=\partial\Omega_{\varepsilon}\cap\Omega$;
$\Gamma_{\varepsilon}:=\partial\Omega_{\varepsilon}
\cap\partial\Omega$. Throughout the paper,
we have the following hypothesises.
\begin{enumerate}
  \item[(H1).]
  \textbf{The structure assumptions on the coefficient.}
\begin{enumerate}
  \item $A$ is real, measurable, and 1-periodic,
  i.e.,
  \begin{equation}\label{eq:1.3}
  A(y+z) = A(z) \quad
  \text{for}~y\in \omega ~\text{and}~ z\in \mathbb{Z}^d.
  \end{equation}
  \item The tensor $A$ satisfies elasticity condition:
\begin{equation}\label{eq:1.4}
\left\{\begin{aligned}
  &a_{ij}^{\alpha\beta}(y)=a_{ji}^{\beta\alpha}(y)
  =a_{\alpha j}^{i\beta}(y);\\
  &\mu_{0}|\xi|^{2}\leq a_{ij}^{\alpha\beta}(y)\xi^{\alpha}_{i}\xi_{j}^{\beta}\leq
  \mu_{1}|\xi|^{2}
\end{aligned}\right.
\end{equation}
for $y\in\omega$ and any symmetric matrix $\xi=\{\xi_{i}^{\alpha}
\}_{1\leq i,\alpha\leq d}$, where $\mu_{0},\mu_{1}>0$.
\end{enumerate}
  \item [(H2).]\textbf{The geometry assumptions on
  the reference domain $\omega$}.
  \begin{enumerate}
    \item $\omega$ is supposed to be connected.
    \item $\omega\cap Y$ is a Lipschitz domain,
    where $Y:=[-1/2,1/2)^{d}$ is the unit cube.
    \item $\mathbb{R}^d\backslash\omega$ represents
    the set of ``holes'', and any two of them
is kept a positive distance.
Specifically, let $\mathbb{R}^{d}
\backslash\omega=\bigcup_{k=1}^{\infty}H_{k}$
in which $H_k$ is connected and bounded for each $k$,
while there exists a constant $\mathfrak{g}^{\omega}$ such that
\begin{equation}\label{g^{w}}
  0<\mathfrak{g}^{\omega}\leq\inf_{i\neq j}\text{dist}(H_{i},H_j).
\end{equation}
  \end{enumerate}
\end{enumerate}

Then, the following qualitative homogenization result was well
known (see for example \cite{CDG,CP,ZR}, and
we presented it in the way of V. Zhikov and M. Rychago
\cite[Theorem 1.1]{ZR}). Let
$F\in L^{2}(\Omega;\mathbb{R}^d)$, and
$u_\varepsilon$ be the weak solution to $\eqref{pde:1.1}$
(see Definition $\ref{def:3}$). It is concluded that
$l_{\varepsilon}^{+}
u_\varepsilon \to u_0$ strongly in
$L^2(\Omega;\mathbb{R}^d)$, and
$l_{\varepsilon}^{+}A(x/\varepsilon)
\nabla u_\varepsilon\rightharpoonup\widehat{A}
\nabla u_0$ weakly in $L^2(\Omega;\mathbb{R}^{d\times d})$,
as $\varepsilon$ goes to zero,
where
$l_\varepsilon^{+}:=l^{+}(\cdot/\varepsilon)$
with $l^{+}$ being the indicator function of $\omega$,
and $u_0$ satisfies the effective (homogenized) equation:
\begin{equation}\label{pde:1.3}
\left\{\begin{aligned}
\mathcal{L}_0(u_0) :=
-\nabla\cdot\widehat{A}\nabla u_0 &= F &\qquad&\text{in}~~\Omega, \\
 u_0 &= g &\qquad& \text{on}~\partial\Omega.
\end{aligned}\right.
\end{equation}
Here,
the homogenized coefficient
$\widehat{A}=\{\widehat{a}_{ij}^{\alpha\beta}
\}_{1\leq i,j,\alpha,\beta\leq d}$ is defined by
\begin{equation}\label{eq:1.1}
\widehat{a}_{ij}^{\alpha\beta}
= \theta^{-1}\int_{Y\cap\omega} a_{ik}^{\alpha\gamma}(y)
\frac{\partial \mathbb{X}_{j}^{\gamma\beta}}{\partial y_{k}} dy
\quad
\text{and}
\quad \theta:=|Y\cap\omega|,
\end{equation}
where $\mathbb{X}^{\beta}_{j}
=\{\mathbb{X}^{\gamma\beta}_{j}\}_{1\leq\gamma\leq d}$
is the weak solution to the following cell problem:
\begin{equation}\label{pde:1.2}
\left\{\begin{aligned}
& \nabla\cdot A \nabla\mathbb{X}^{\beta}_{j} = 0
\qquad\text{in}~ Y\cap \omega,\\
& n\cdot A\nabla\mathbb{X}^{\beta}_{j}=0
\qquad\text{on}~Y\cap\partial\omega,\\
&\mathbb{X}_{j}^{\beta}-y_{j}e^{\beta}:
=\chi_{j}^{\beta}\in H^1_{\text{per}}(\omega;\mathbb{R}^d), \quad&
\int_{Y\cap \omega}&\chi_{j}^{\beta} dy = 0.
\end{aligned}\right.
\end{equation}
(The definition of
the space
$H^1_{\text{per}}(\omega;\mathbb{R}^d)$
will be found in
Subsection $\ref{notation}$.)

\vspace{0.2cm}
Now, we turn to quantitative estimates,
and state the main results of the paper as follows.
\begin{theorem}[suboptimal error estimates]\label{thm:1.1}
Let $0<\varepsilon\ll 1$, $0<\tau<1$ and $p=\frac{2d}{d-1-\tau}$.
Suppose that $\mathcal{L}_{\varepsilon}$
and $\omega$ satisfy the hypothesises
\emph{(H1)} and \emph{(H2)}. Assume that
$u_{\varepsilon}\in H^{1}(\Omega_{\varepsilon};\mathbb{R}^d)$ and
$u_{0}\in H^{1}(\Omega;\mathbb{R}^d)$ are weak
solutions to \eqref{pde:1.1} and \eqref{pde:1.3}, respectively.
Let $\Omega_{0}$ be
an extended region such that
$\bar{\Omega}\subset\Omega_{0}$
with \emph{dist}$(x,\Omega)\sim
10\varepsilon$ for every $x\in\partial \Omega_{0}$, and
$F\in H^{1}_{\emph{loc}}(\Omega_0;\mathbb{R}^d)\cap
L^2(\Omega_0;\mathbb{R}^d)$ satisfies
a finite square
function condition, i.e.,
$\big(\int_{\Omega_0}|\nabla F|^2\delta dx\big)^{1/2} <\infty$,
where $\delta(x)=\emph{dist}(x,\partial\Omega_0)$.
\begin{itemize}
\item[$\bullet$]
If $\Omega$ is a $C^{1,\eta}$ domain with $\eta\in(0,1]$, and
$g\in W^{1,q}(\partial\Omega;\mathbb{R}^d)$
with $q=p-p/d$.
Then one may derive that
\begin{equation}\label{pri:1.1}
  \|u_{\varepsilon}-u_{0}\|_{L^{p}
  (\Omega_{\varepsilon})}
  \lesssim \varepsilon^{1-\frac{\tau}{2}}
  \ln^{\frac{1}{2}}(1/\varepsilon)
\Bigg\{
\Big(\int_{\Omega_0}|\nabla F|^2
\delta dx\Big)^{\frac{1}{2}}
+
\|F\|_{L^{2}(\Omega_0)}+\|g\|_{W^{1,q}(\partial\Omega)}\Bigg\}.
\end{equation}
  \item[$\bullet$] If $\Omega$ is a Lipschitz domain
  and $g\in H^1(\partial\Omega;\mathbb{R}^d)$,
  then there holds
\begin{equation}\label{eq:1.8}
  \|u_{\varepsilon}-u_{0}\|_{L^{2}(\Omega_{\varepsilon})}
  \lesssim \varepsilon^{\frac{5}{6}} \ln^{\frac{2}{3}}(1/\varepsilon)
\Bigg\{
\Big(\int_{\Omega}|\nabla F|^2
\delta^2dx\Big)^{\frac{1}{2}}
+\|F\|_{L^{2}(\Omega)}+\|g\|_{H^{1}(\partial\Omega)}\Bigg\}.
\end{equation}
\end{itemize}
Here, $\lesssim$
means $\leq$ up to a multiplicative constant which
depend on $\mu_{0},\mu_{1},d,\mathfrak{g}^{\omega},\eta,
\tau, r_0$
and the characters of $\Omega$ and $\omega$, but never on
$\varepsilon$.
\end{theorem}

We refer the reader to
Subsection $\ref{notation}$ for the definition
of the notation like ``$r_0$'',``$\sim$'',``$\ll$''.

Although the following theorem is not
particularly new (probably well known by experts),
it still deserves to be presented in a more prominent position
because it is not only the first step towards
Theorem $\ref{thm:1.1}$ but also a proper perspective to
realize the hurdles of Theorem $\ref{thm:1.1}$.

\begin{theorem}[$H^{1}$-error estimates]\label{thm:2.1}
Let $\Omega\subset\mathbb{R}^d$ be a bounded Lipschitz
domain. Given $F\in H^{1-s}(\mathbb{R}^d;\mathbb{R}^d)$
with $0\leq s\leq 1$ and
$g\in H^{1}(\partial\Omega;\mathbb{R}^d)$,
let the weak solutions $u_{\varepsilon}$ and
$u_{0}$
be associated with the given data by \eqref{pde:1.1} and \eqref{pde:1.3},
respectively. Suppose that $\mathcal{L}_{\varepsilon}$
and $\omega$ satisfy the hypothesises
\emph{(H1)} and \emph{(H2)}.
Consider the first-order approximating
corrector
\begin{equation}\label{eq:3.1}
w_{\varepsilon} := u_{\varepsilon}-u_{0}+\varepsilon
\chi(\cdot/\varepsilon)
S_{\varepsilon}(\psi_{\varepsilon}\nabla u_{0}),
\end{equation}
in which $S_\varepsilon$ is a smoothing operator
(see Definition $\ref{def:2}$),
and $\psi_\varepsilon$ is a cut-off function
defined in $\eqref{eq:2.1}$.
Then one may derive that
\begin{equation}\label{pri:1.7}
\begin{aligned}
\|\nabla w_{\varepsilon}\|_{L^{2}(\Omega_{\varepsilon})}
\lesssim \varepsilon^{1-s}\|F\|_{H^{1-s}(\mathbb{R}^d)}
+ \varepsilon^{1/2}
\bigg\{\|F\|_{L^{2}(\Omega)}
+\|g\|_{H^{1}(\partial\Omega)}\bigg\}.
\end{aligned}
\end{equation}
Moreover, if $F=0$, then there holds
\begin{equation}\label{pri:1.4}
\|u_{\varepsilon}-u_{0}
\|_{L^{\frac{2d}{d-1}}(\Omega_{\varepsilon})}
\lesssim \varepsilon^{1/2}\|g\|_{H^{1}(\partial\Omega)},
\end{equation}
where the up to constant depends on $\mu_{0}, \mu_{1},
d, r_0$ and
the characters of $\omega$ and $\Omega$.
\end{theorem}

Compared to the case of unperforated domains,
the new phenomenon arising from the source term
$F$ will be observed from $\eqref{pri:1.7}$. On the one hand,
square-integrable $F$ as given data
is too rough to offer a convergence rate.
On the other hand, a sharp error estimate would be delivered
by $H^{-\frac{1}{2}}$-norms
(see Corollary $\ref{cor:4.1}$).

Since the boundary of $\Omega_\varepsilon$ may be
quite irregular, some fundamental analysis tools such as
Poincar\'e's inequalities and trace theorems can not be employed
anymore. Instead, people first manage to extend
functions to a larger and more regular region and then operate them
as usual. In this regard,
O. Oleinik, A. Shamaev, and G. Yosifian
\cite[pp.50, Theorem 4.3]{OSY} successfully constructed
a bounded linear extension operator
from
$H^{1}(\Omega_{\varepsilon},\Gamma_{\varepsilon};
\mathbb{R}^d)$ to $H^1_0(\mathbb{R}^d;\mathbb{R}^d)$,
in which
$H^{1}(\Omega_{\varepsilon},\Gamma_{\varepsilon};
\mathbb{R}^d)$
denotes the closure in $H^{1}(\Omega_{\varepsilon};
\mathbb{R}^d)$ of smooth vector-valued functions
vanishing on $\Gamma_{\varepsilon}$
(see Subsection $\ref{notation}$).
Here, to fulfill the new scheme on Theorem $\ref{thm:1.1}$,
their consequences have received further development in the following.


\begin{theorem}[weighted extension property]\label{thm:1.5}
Let $\Omega$ and $\Omega_{0}$ be
bounded Lipschitz domains with
$\bar{\Omega}\subset\Omega_{0}$
and \emph{dist}$(x,\Omega)\sim
10\varepsilon$ for $x\in\partial\Omega_0$.
Suppose that $\omega$ satisfies the hypothesis \emph{(H1)}
with $\Omega_\varepsilon = \Omega\cap(\varepsilon\omega)$.
Let $\delta(x)=\emph{dist}(x,\partial\Omega_0)$ be a weight.
Then there exists an extension operator
$\Lambda_\varepsilon: H^1(\Omega_\varepsilon,
\Gamma_\varepsilon;\mathbb{R}^d)\to
H_0^1(\Omega_0;\mathbb{R}^d)$
such that
$\Lambda_\varepsilon (w) = w$
a.e. on $\Omega_\varepsilon$ with
the trace of $\Lambda_\varepsilon (w)$ vanishing on
$\partial\Omega_0$. Moreover, there holds the following
weighted estimate
\begin{equation}\label{pri:1.9}
\int_{\Omega_0}|\nabla \Lambda_\varepsilon(w)|^2
\delta^\beta dx
\lesssim \int_{\Omega_\varepsilon}
|\nabla w|^2 \delta^\beta dx
\end{equation}
for $-1<\beta<1$,
where the up to constant depends
on $d$, $\mathfrak{g}^{\omega}$ and the characters
of $\Omega$ and $\omega$,
but independent of $\varepsilon$ and $r_0$.
\end{theorem}

Concerned with the above theorem,
we want to emphasize two points:
(i). Note that $\delta^\beta\in A_{2}$
(see for example \cite[Theorem 3.1]{DST}), which
guaranteed that we could rightfully employ some extension theorems
in weighted Sobolev spaces (see \cite[Theorem 1.2]{C}),
where the notation $A_p$ with $1\leq p\leq \infty$
is known as a Muckenhoupt's weight class (see
for example \cite[Chapter 7]{D}).
(ii). The stated estimates $\eqref{pri:1.9}$
is rooted in weighted $W^{1,p}$ estimates
for an auxiliary equation with mixed boundary conditions
(see $\eqref{pde:2.1}$ and Theorem $\ref{app:thm:1}$).

Although the equation $\eqref{pde:2.1}$
was constructed with constant coefficients,
its proof shared the same strategy
(repeating Shen's real methods twice) with the following theorem.
As we have mentioned in Abstract, these estimates below were
first known by A. Gloria, S. Neukamm, F. Otto
\cite{Gloria_Neukamm_Otto_2015} in a random setting.

\begin{theorem}[weighted quenched Calder\'on-Zygmund estimates]
\label{thm:1.4}
Let $0<\varepsilon\ll 1$
and $\Omega$ be a bounded $C^{1,\eta}$ domain with
$0<\eta\leq 1$.
Suppose that $\mathcal{L}_{\varepsilon}$
and $\omega$ satisfy the hypothesises
\emph{(H1)} and \emph{(H2)}.
Let $1<p<\infty$, and $\rho\in A_p$ be
a weight. For any
$f\in L^p(\Omega;\mathbb{R}^{d\times d})$,
assume that $u_\varepsilon$ is the weak solution to
$\mathcal{L}_\varepsilon(u_\varepsilon) =
\nabla\cdot f$ in $\Omega_\varepsilon$ with
$\sigma_\varepsilon(u_\varepsilon) =
-n\cdot f$ on $S_\varepsilon$
and $u_\varepsilon=0$ on $\Gamma_\varepsilon$.
Then there holds
\begin{equation}\label{pri:1.5}
\bigg(\int_{\Omega}
\Big(\dashint_{B_\varepsilon(x)\cap\Omega_\varepsilon}
|\nabla u_\varepsilon|^2
\Big)^{\frac{p}{2}}\rho dx\bigg)^{\frac{1}{p}}
\lesssim \bigg(\int_{\Omega}
\Big(\dashint_{B_{\varepsilon}(x)\cap\Omega_\varepsilon}
|f|^2
\Big)^{\frac{p}{2}}\rho dx\bigg)^{\frac{1}{p}},
\end{equation}
where $B_\varepsilon(x)$ is the ball centered at x with
radius $\varepsilon$, and the average integral
$\dashint$ is defined in $\eqref{n:1.1}$.
Moreover, if $\rho\in A_1$, then we have
\begin{equation}\label{pri:1.6}
\bigg(\int_{\Omega}
\Big(\dashint_{B_\varepsilon(x)\cap\Omega_\varepsilon}
|\nabla u_\varepsilon|^2
\Big)^{\frac{p}{2}}\rho dx\bigg)^{\frac{1}{p}}
\lesssim \bigg(\int_{\Omega_\varepsilon}
|f|^p\rho dx\bigg)^{\frac{1}{p}}.
\end{equation}
Here the up to constant depends on
$\mu_0,\mu_1,d,\mathfrak{g}^{\omega},\eta$, $p$ and
the characters of $\Omega$ and $\omega$.
\end{theorem}

Up to now,
the idea of Theorem $\ref{thm:1.4}$ is more or less
standard in harmonic analysis. One may
reduce it to a good-$\lambda$ inequality measured by weights
(see $\eqref{f:9}$),
which essentially parallels to the classical
counterparts developed for singular integrals (see \cite[Theorem 7.11]{D}).
The main difference is that
a pointwise estimate of
the composition of
sharp maximal operators with singular integrals
is displaced by Shen's real arguments \cite{S0,S3}.
Recently, Z. Shen \cite{S4} developed his kernel-free scheme
in weighted spaces.
Nevertheless, the important fact is that weighted estimates
are usually based upon non-weighted ones. Whereupon,
it is very natural to consider the next theorem.

\begin{theorem}[quenched Calder\'on-Zygmund estimate]
\label{thm:1.2}
Assume the same conditions as in Theorem $\ref{thm:1.4}$.
For any $f\in L^p(\Omega;\mathbb{R}^{d\times d})$,
let $u_\varepsilon$ be associated with $f$ by the equation
$\mathcal{L}_\varepsilon(u_\varepsilon) =
\nabla\cdot f$ in $\Omega_\varepsilon$ with
$\sigma_\varepsilon(u_\varepsilon) =
-n\cdot f$ on $S_\varepsilon$
and $u_\varepsilon=0$ on $\Gamma_\varepsilon$.
Then, for $1<p<\infty$, we obtain
\begin{equation}\label{pri:1.2}
\bigg(\int_{\Omega}
\Big(\dashint_{B_\varepsilon(x)\cap\Omega_\varepsilon}
|\nabla u_\varepsilon|^2
\Big)^{\frac{p}{2}}dx\bigg)^{\frac{1}{p}}
\lesssim \bigg(\int_{\Omega}
\Big(\dashint_{B_{\varepsilon}(x)\cap\Omega}
|f|^2
\Big)^{\frac{p}{2}}dx\bigg)^{\frac{1}{p}},
\end{equation}
where the up to constant relies on
$\mu_0,\mu_1,d,\mathfrak{g}^{\omega},\eta$, $p$ and
the characters of $\Omega$ and $\omega$.
\end{theorem}

The terminology ``quenched'' probably came from
some mathematical physicists, which was likely used to
underline some structural assumption on $A$
at macroscopic scales.
Regarding of this paper,
it refers to the 1-periodicity
condition $\eqref{eq:1.3}$.
In other words,
some higher regularity estimates
(beyond energy estimates) can still be expected,
if the coefficient $A$ is ruled by some ``law''
such as periodicity, quasi-periodicity, almost-periodicity,
and stationarity coupled
with quantitative ergodicity.
In the absence of smoothness of coefficients,
it is reasonable to replace classical
pointwise quantity with a specific average.
Also, we strongly refer the reader to \cite{JO}
for its tight
relationship with a quantitative stochastic homogenization theory.

Recently, the estimate $\eqref{pri:1.2}$ as an intermediate step
 was developed
for elliptic systems with stationary random coefficients
of integrable correlations by
M. Duerinckx and F. Otto \cite{DO}.
In their plot,
there were two notable ingredients.
One was Shen's lemma (see Lemma $\ref{lemma:5.4}$), inspired by
L. Caffarelli and I. Peral's work \cite{CP}. The other was
the large-scale Lipschitz regularity, as developed
in \cite{AKM,AM,Gloria_Neukamm_Otto_2015}.
In terms of a quasi-linear model, a similar quenched $W^{1,p}$ estimate
with $p>2$ was received by S. Armstrong, J.-P. Daniel \cite{AD},
and their approach
was closer to S. Byun and L. Wang's framework \cite{BW},
which also enlightened by the literature \cite{CP}.
In short, although their methods have their characteristics,
the common point is the dependence on the large-scale
Lipschitz estimates.

\begin{theorem}[boundary Lipschitz estimates at large-scales]
\label{thm:1.3}
Let $0<\varepsilon\ll 1$
and $\Omega$ be a bounded $C^{1,\eta}$ domain with
$0<\eta\leq 1$.
Suppose that $\mathcal{L}_{\varepsilon}$
and $\omega$ satisfy the hypothesises
\emph{(H1)} and \emph{(H2)}. Let
$u_\varepsilon\in H^1(D_1^{\varepsilon};\mathbb{R}^d)$ be
a weak solution to
\begin{equation}\label{pde:5.1}
\left\{\begin{aligned}
\mathcal{L}_\varepsilon (u_\varepsilon) &= 0
&\quad&\emph{in}~D_1^{\varepsilon},\\
\sigma_\varepsilon (u_\varepsilon) &= 0
&\quad&\emph{on}~D_1\cap\partial(\varepsilon\omega),\\
u_\varepsilon &= 0
&\quad&\emph{on}~\Delta_1\cap\varepsilon\omega,
\end{aligned}\right.
\end{equation}
in which the notation $D_1^\varepsilon$, $D_1$ and
$\Delta_1$ are referred to Subsection $\ref{notation}$.
Then there holds
\begin{equation}\label{pri:5.5}
 \Big(\dashint_{D_r^\varepsilon}|\nabla u_\varepsilon|^2
 \Big)^{1/2}
 \lesssim \Big(\dashint_{D_{1/2}^\varepsilon}|\nabla u_\varepsilon|^2
 \Big)^{1/2}
\end{equation}
for any $\varepsilon\leq r\leq (1/2)$, where the up to constant
depends on $\mu_0,\mu_1,d,\mathfrak{g}^{\omega},\eta$ and
the characters of $\Omega$ and $\omega$.
\end{theorem}

The large-scale (uniform) Lipschitz regularity was first
obtained by M. Avellaneda, F. Lin \cite{AL} via a
three-step compactness method.
Recently, S. Armstrong, T. Kuusi, J.-C. Mourrat, Z. Shen
\cite{AKM,AM,AS}
created a new approach towards $\eqref{pri:5.5}$ in
aperiodic settings. (In this respect,
a fair statement should not ignore
A. Gloria, S. Neukamm, F. Otto's work
\cite{Gloria_Neukamm_Otto_2015}, although
theirs has not been formally published yet.)
Roughly speaking, it includes two steps.
The first one is devoted to approximating $u_\varepsilon$ by
a ``good'' function in a quantitative way, which is
finally reduced to find a convergence rate
like the estimate $\eqref{pri:1.4}$. The second one is to carry out
the so-called  Campanato's iteration, which is similar to
a classical program of Schauder's estimates.

In the matter of the iteration,
we prefer the version given by Z. Shen \cite[Lemma 8.5]{S},
so it is free to use in this step. Consequently,
to complete the whole argument for $\eqref{pri:5.5}$,
we need to review Theorem $\ref{thm:2.1}$, which
also works as the first step in the proof of
large-scale regularities.

\begin{remark}
\emph{The authors pointed out that
the scheme for $\eqref{pri:1.1}$ would be
valid to a Lipschitz domain since we merely employ Theorem $\ref{thm:1.4}$ in the
case of $p=2$. Thus, it is reasonable to believe
that Theorem $\ref{thm:1.4}$ would be right for $|p-2|\ll 1$
if $\Omega$ were merely a Lipschitz domain, and
it would be linked up to the new progress in \cite{S4}.
On the other hand,
it is not optimal for assuming $\partial\Omega\in C^{1,\eta}$
in Theorems $\ref{thm:1.4}$, and $\ref{thm:1.2}$.
The potential
substitutions might be a $C^1$ region or some non-smooth domain
with small Lipschitz constants (such as
Reifenberg flat domains in \cite{BW}).
It is obviously a separate interest.}
\end{remark}

\begin{remark}
\emph{As mentioned in Abstract, one of
the new contributions
should be suboptimal error estimates for the square function of
the first-order approximating corrector $w_\varepsilon$,
which was shown in Corollary $\ref{cor:8.1}$.
Compared to arguments developed by C. Kenig,
F. Lin, Z. Shen \cite{KLS}, ours seems to be
much easier to be generalized in other models.}
\end{remark}

\begin{remark}
\emph{In terms of given data $F$ in Theorem $\ref{thm:1.1}$,
its regularities shown in $\eqref{pri:1.1}$
and $\eqref{eq:1.8}$ seem to be a minimum requirement,
although it is still stronger than the $H^{1/2}$-norm.
At least, the hypothesis of $H^{1/2}$-smoothness
of $F$ is deemed to be necessary for the
error estimate $O(\varepsilon^{1/2})$ in
Theorem $\ref{thm:2.1}$.}
\end{remark}


\subsection{Related to Previous Works on Error Estimates}

\textbf{Pioneering studies on the present model $\eqref{pde:1.1}$.}
Homogenization in the perforated domain has been
considered for decades, and most of the papers studied
qualitative results, for example, \cite{GA,CDG,CP,OSY,ZR}.
There were also some quantitative outcomes.
O. Oleinik, A. Shamaev, and G. Yosifian
\cite[pp.124, Theorem 1.2]{OSY}
obtained the error estimate:
\begin{equation}\label{eq:1.10}
  \|u_{\varepsilon}-u_{0}-\varepsilon\chi_{\varepsilon}
  \nabla u_{0}\|_{H^{1}(\Omega_{\varepsilon})}
  \lesssim\varepsilon^{1/2}
  \Big\{\|g\|_{H^{5/2}(\partial\Omega)}
  +\|F\|_{H^{1}(\Omega)}\Big\},
\end{equation}
under additional regularity assumptions on the
coefficient $A$ and $\omega\cap Y$.
Recently, B. Russell \cite[Theorem 1.4]{BR} improved the result \eqref{eq:1.10}
by receiving
\begin{equation*}
 \|u_{\varepsilon}-u_{0}-
 \varepsilon\chi_{\varepsilon}S_\varepsilon^2
 (\psi_\varepsilon\nabla u_{0})\|_{H^{1}(\Omega_{\varepsilon})}\lesssim \varepsilon^{1/2}
  \|g\|_{H^{1}(\partial\Omega)}
\end{equation*}
in the case that $F=0$, and $A$ merely
satisfies $\eqref{eq:1.3}$, $\eqref{eq:1.4}$. Meanwhile,
an interior large-scale Lipschitz regularity
(see \cite[Theorem 1.1]{BR}) was established, i.e.,
\begin{equation}\label{pri:1.10}
 \Big(\dashint_{B(0,r)\cap\varepsilon\omega}|\nabla u_\varepsilon|^2
 \Big)^{1/2}
 \lesssim \Big(\dashint_{B(0,1)\cap\varepsilon\omega}
 |\nabla u_\varepsilon|^2
 \Big)^{1/2}
\end{equation}
for $\varepsilon\leq r\leq (1/4)$,
provided $u_\varepsilon$ satisfies
$\mathcal{L}_\varepsilon(u_\varepsilon) = 0$ in
$B(0,1)\cap\varepsilon\omega$ and
$\sigma_\varepsilon(u_\varepsilon)=0$ on
$B(0,1)\cap\partial(\varepsilon\omega)$.
From our point of view, their notable contributions
are summarized below.
\begin{itemize}
  \item The literature \cite{OSY} developed
  some extension theorems on perforated domains, and
  their central ideas stimulated the creation of Theorem
  $\ref{thm:1.5}$.
  \item The so-called flux corrector was introduced
  by \cite{OSY,BR} in a different format, and the later one
  first extended the corrector $\chi$ from $\omega\cap Y$ to $Y$
  and then defined flux correctors on $Y$
  (see Lemma $\ref{lemma:2.6}$).
  \item Concerning the estimate $\eqref{pri:1.10}$,
  the literature \cite{BR} developed some techniques to
  make the scheme on large-scale regularities
  in \cite{S} valid for perforated domains, such as
  \cite[Lemma 4.4]{BR}.
  This provided us with a blueprint to
  Theorem $\ref{thm:1.3}$, while our work focused on overcoming
  the difficulties caused by boundaries.
\end{itemize}

\textbf{Earlier works on duality arguments for
sharp error estimates.} In this part of the presentation,
we limited ourselves in the case of $\omega =\mathbb{R}^d$.
To our best knowledge, T. Suslina \cite{TS,TS1} was the first to
use a duality approach to acquire an optimal convergence rate
in $L^2$-norm
for elliptic homogenization problems with Dirichlet or
Neumann boundary conditions.
However, her literature was a little biased towards operator
algebra fields. Later on, Z. Shen \cite[Theorem 1.3]{S}
described her ideas from PDEs and
originally obtained a sharp error estimate measured by
$L^{\frac{2d}{d-1}}$-norm. It is important to note that their results
were independent of the smoothness
of the coefficients.

Regarding Lipschitz domains,
under additional smoothness assumption on coefficients,
C. Kenig, F. Lin, Z. Shen \cite{KLS} obtained almost-sharp error estimates in $L^2$-norm. Most importantly, their
results included
an error estimate
of the first-order approximating corrector
in the sense of the ``square function'', i.e.,
\begin{equation}\label{f:1.1}
 \bigg(\int_{\Omega} |\nabla v_\varepsilon|^2\text{dist}(x,\partial\Omega) dx
 \bigg)^{1/2} = O(\varepsilon),
\end{equation}
where $v_\varepsilon = u_\varepsilon-u_0
 -\varepsilon\chi(\cdot/\varepsilon)\nabla u_0$ in $\Omega$.
Considering nontangential maximal function estimates
had been established by C. Kenig, Z. Shen
\cite{SZW24}, it was possible to apply a duality
scheme to some quantity on boundaries, in which
the so-called Rellich estimates were instrumental.
On the other hand, the nontangential maximal function and
the square function are comparable if the solution satisfies a
homogeneous elliptic equation (i.e.,
its right-hand side equals zero). Also,
some crucial analysis tools and observations were introduced.
\begin{itemize}
  \item Weighted estimates for potentials.
  \item For any
  $\mathcal{L}_1$-harmonic function, its
  non-tangential maximal function must be
  equivalent to its radial maximal function.
\end{itemize}

Recently, without smoothness assumption on coefficients,
the second author \cite{Q1}
obtained the same result as in \cite{KLS}
regarding the almost-sharp error estimates in $L^2$-norm.
Roughly speaking, his project combined the duality arguments in
\cite{S,TS,TS1} and ideas on weighted estimates in \cite{KLS}.
The individual contribution in \cite{Q1} was some
weighted-type estimates
for the smoothing operator at $\varepsilon$-scales,
i.e., Lemmas $\ref{lemma:2.10}$ and $\ref{lemma:2.04}$,
which still played a vital role in this literature.
Also, we mention that the approach developed in \cite{Q1}
can be flexibly extended to other models, such as
Stokes systems,  parabolic systems,
and higher-order elliptic systems.
However, it was failed to establish
the estimate $\eqref{f:1.1}$ in \cite{Q1},
while this was precisely a concealed interest of the current work.
In this regard, we will grasp more about it in Subsection
$\ref{section4}$.

Also, we refer the reader to \cite{GSS,GSS1,SZ}
for a recent development on quantitative estimates of
elasticity systems in periodic mediums.

\subsection{Outline of the Ingredients for $\eqref{pri:1.1}$}\label{section4}

As mentioned previously,
the extension idea is conventional in the study
of homogenization problems
on perforated domains. However,
we remind the reader that the extended region $\Omega_0$
is only slightly larger than the original one $\Omega$,
whose difference in radial size can be measured by
$\varepsilon$-scales. The primary purpose is to generate
a boundary layer (see, for example
handling the term $T_3$
of $\eqref{f:4.1}$). At the same time,
the inverse of the distance function $\delta$ has no singularity
around the boundary of $\Omega$,
which benefits the computations near to $\partial\Omega$ (see,
for example, Remark $\ref{remark:8.1}$).

Recalling the weight $\rho=\delta^\beta$
with $-1<\beta<1$, as well as, $w_\varepsilon$ defined in $\eqref{eq:3.1}$,
we now describe the main ingredients for
the estimate $\eqref{pri:1.1}$ in the following
(it is fine to assume that $\partial\Omega$ and $g$
are sufficiently smooth).
\begin{enumerate}
  \item Estimate for a weak formulation (a duality argument):
\begin{equation}\label{f:1.2}
\begin{aligned}
\Big|\int_{\Omega_\varepsilon}\nabla w_\varepsilon\cdot f dx\Big|
&\lesssim \varepsilon^{\frac{1+\beta}{2}}
\bigg(\int_{\Omega}\dashint_{B_\varepsilon(x)
\cap\Omega_\varepsilon}|\nabla\phi_\varepsilon|^2
dy\rho^{-1} dx
\bigg)^{\frac{1}{2}}
\times
\Big\{\text{given~data}\Big\},
\end{aligned}
\end{equation}
where $\phi_\varepsilon$ is associated with $f$ by the adjoint
equation $\eqref{adjoint2}$ (see Theorem $\ref{thm:8.1}$).
\item Weighted quenched Calder\'on-Zygmund estimates
``$+$'' \eqref{f:1.2}
imply the square function  estimates
\begin{equation*}
 \Big(\int_{\Omega_\varepsilon}|\nabla
  w_\varepsilon|^2 \delta(x) dx\Big)^{\frac{1}{2}}
 = O\big(\varepsilon^{\frac{1+\beta}{2}}\big)
\end{equation*}
(see Corollary $\ref{cor:8.1}$). Although this result
is suboptimal compared to $\eqref{f:1.1}$, it does not
rely on any smoothness assumption on the coefficients.
\item The weighted extension theorem
``$+$'' weighted Hardy-Sobolev inequalities
(\cite[Theorem 2.1]{LV}) leads to
\begin{equation*}
\Big(\int_{\Omega_\varepsilon}|\nabla w_\varepsilon
|^{2}\delta^{\beta}dx
\Big)^{\frac{1}{2}}
\gtrsim
\Big(\int_{\Omega_0}|\nabla \Lambda_\varepsilon(w_\varepsilon)
|^{2}\delta^{\beta}dx
\Big)^{\frac{1}{2}}
\gtrsim
\Big(\int_{\Omega_0}|\Lambda_\varepsilon(w_\varepsilon)
|^{\frac{2d}{d-2+\beta}}dx
\Big)^{\frac{d-2+\beta}{2d}},
\end{equation*}
in which $0<\beta<1$, and we mention that
the critical case $\beta=1$ is invalid above,
while the case $\beta=0$ is valid for $d>2$.
\end{enumerate}



\subsection{Tricks on the Estimate $\eqref{eq:1.8}$}

 For the  interpolation inequality between
 the Hilbert spaces $H^1$ and $H^{-s}$ (with $s\geq 0$)
 was well known on the whole
 space $\mathbb{R}^d$ (see for example \cite[Proposition 1.52]{BHCD}),
 we consider the extension of
 $w_\varepsilon$, denoted by $\tilde{w}_\varepsilon$, in
 the way of Lemma $\ref{extensiontheory}$ (whose previous
 version was \cite[pp.50, Theorem 4.3]{OSY}).
 Here the improved effects
 guarantee
 that $\text{supp}(\tilde{w}_\varepsilon)\subseteq \Omega_0$
 with $\Omega_0\supseteq \bar{\Omega}$ being such that
 $\text{dist}(x,\partial\Omega)\sim 10\varepsilon$
 for $x\in\partial\Omega_0$. Let $\Phi\in
 H^{\sigma}(\mathbb{R}^d;\mathbb{R}^d)$
 with $(1/2)\leq\sigma\leq 1$ be any test function,
 and then it is natural to analyze the following quantity
 \begin{equation*}
 \begin{aligned}
 \int_{\mathbb{R}^d} \theta\tilde{w}_\varepsilon\Phi dx,
 \end{aligned}
 \end{equation*}
 which one may formulate into
 \begin{equation}\label{f:4.1}
 \begin{aligned}
\int_{\mathbb{R}^d} \theta\tilde{w}_\varepsilon\Phi dx
&=\int_{\Omega} \theta\tilde{w}_\varepsilon\Phi dx
 + \int_{\mathbb{R}^d\setminus\Omega}
 \theta\tilde{w}_\varepsilon\Phi dx\\
& = \int_{\Omega}l_\varepsilon^+\tilde{w}_\varepsilon\Phi dx
 + \int_{\Omega} (\theta-l_\varepsilon^+)
 \tilde{w}_\varepsilon\Phi dx
 +\int_{\Omega_0\setminus\Omega}
 \theta\tilde{w}_\varepsilon\Phi dx \\
&=\underbrace{\int_{\Omega_{\varepsilon}} A_{ij}(x/\varepsilon)\nabla_{x_{j}}
 w_{\varepsilon}\nabla_{x_{i}}\phi_{\varepsilon}dx}_{T_1}
 + \underbrace{\int_{\Omega} (\theta-l_\varepsilon^+)
 \tilde{w}_\varepsilon\Phi dx}_{T_2}
 +\underbrace{\int_{\Omega_0\setminus\Omega}
 \theta\tilde{w}_\varepsilon\Phi dx}_{T_3}.
 \end{aligned}
 \end{equation}
 Here $\phi_\varepsilon$ is associated with $\Phi$ by
 the adjoint equation $\eqref{adjoint}$, and
 we also use the observation that
 $\int_{\Omega}l_\varepsilon^+\tilde{w}_\varepsilon\Phi dx
 =\int_{\Omega_{\varepsilon}}w_{\varepsilon}\Phi dx$
 in the last equality.
 It is a suitable weak formulation for $\eqref{eq:1.8}$.
 Now, there are some small tricks for every term
 in the bottom line of $\eqref{f:4.1}$.
 \begin{enumerate}
   \item The quantity $\phi_\varepsilon$
   in $T_1$ will be substituted
   with two-scale expansions (first-order) of
   $\phi_\varepsilon$ (see the formula $\eqref{sharp9}$),
   and this idea was originally shown in \cite{TS,TS1}.
   \item Appeal to auxiliary equation $\eqref{auxi1}$
   to engage with the oscillating factor
   $(\theta-l_\varepsilon^{+})$ in $T_2$.
   \item As $\Omega_0\setminus\Omega$
   is a boundary layer part of
   $\Omega_0$, it is reduced to apply
   Lemma $\ref{lemma:2.8}$ to $T_3$.
 \end{enumerate}



 \subsection{Organization of the Paper}

In Section $\ref{sec:2}$, we first introduced
some properties of homogenized coefficients and flux correctors.
A basic and essential fact is that the effective operator
$\mathcal{L}_0$ and $\mathcal{L}_\varepsilon$
belong to the same type class of operators
(see Lemma $\ref{lemma:2.7}$), which
is crucial to the idea of ``borrowing''
good regularities from effective equations in all scales.
In terms of the antisymmetry (see Lemma $\ref{lemma:2.6}$),
the flux corrector played a central role purely in the
algebraical level, which finally leads to
an informative weak formulation for the first-order approximating
corrector (see $\eqref{eq:2.2}$).
Then, the so-called periodic cancellations
measured by different norms were presented in Subsection
$\ref{subsec:2.1}$, in which Lemma $\ref{lemma:2.3}$
seems to be newly imposed. Subsection $\ref{subsec:2.2}$
was devoted to some extension theories.
The first notable consequence was stated in Lemma
$\ref{extensiontheory}$, which
served as a technical starting point in the scheme
to $\eqref{eq:1.8}$, and therefore played a fundamental role
in Sections $\ref{section3}$ and $\ref{sec:4}$.
Then Lemmas $\ref{lemma:extension}$,
$\ref{lemma:2.8}$ as preparations were used
to show a proof for Theorem $\ref{thm:1.5}$,
which was assigned at the end of this section.

\vspace{0.1cm}

In Section $\ref{section3}$,
the most important result was a proper weak formulation
for $w_\varepsilon$ defined in $\eqref{eq:3.1}$, which
was shown in Lemma $\ref{lemma:2.1}$. The
reason why we call it
weak formulation instead of bilinear form was
that the formula $\eqref{eq:2.2}$ was weaker than
the bilinear form given in Definition $\ref{def:3}$.
In other words, our computations in Sections
$\ref{section3}, \ref{sec:4}, \ref{sec:6}$
were based upon
the weak formulation instead of the bilinear form as usual.
The proof of Theorem $\ref{thm:2.1}$ was reached in this
section.
In Section $\ref{sec:4}$,
we proved the estimate $\eqref{eq:1.8}$
(see Corollary $\ref{cor:4.1}$). However, the most
exciting result around $\eqref{f:4.1}$
was summarized in Theorem $\ref{thm:3.1}$.
Technically, it together with Lemma $\ref{lemma:3.2}$ was
most highly consistent with \cite[Lemma 3.5]{Q1}
throughout the paper.

In Section $\ref{sec:6}$, the main innovation of this work
would be reflected in Theorem $\ref{thm:8.1}$, which was, in fact,
parallel to Theorem $\ref{thm:3.1}$
from the duality. Another impressive
result was the suboptimal error estimates for
the square function of the first-order approximating corrector,
which was presented in Corollary $\ref{cor:8.1}$.
However, arguments in this section relied on quenched
regularities that we would discuss later on, which might
bring the reader some uncomfortable feeling. So,
we recommended that the reader would skip this section
for a moment if he or she were not familiar with
quenched regularity theories.

In Section $\ref{sec:5}$, Shen's real methods were
presented in Lemma $\ref{lemma:5.4}$. Other useful
results were related to primary geometry on integrals (see
Lemma $\ref{lemma:5.5}$). Then, we provided
the proofs of Theorems $\ref{thm:1.2}$ and $\ref{thm:1.4}$,
respectively,
based upon the large-scale Lipschitz estimate $\eqref{pri:5.5}$.
In Section $\ref{sec:7}$, we completed the main plot
for Theorem $\ref{thm:1.3}$. Compared with the case
on unperforated domains, people had first to
investigate an equivalence
of two quantities defined in $\eqref{eq:5.1}$ at large-scales
(see Lemma $\ref{lemma:5.1}$). Also,
a quenched boundary $L^2$-error estimate
(see Lemma $\ref{lemma:5.2}$)
involved more techniques than those engaged in interior cases.
The rest parts were standard nowadays,
and we preferred to omit them.

In Section $\ref{sec:8}$,
there were two significant results. One was
weighted $W^{1,p}$ estimates for a
mixed boundary problem on Lipschitz domains
and stated in Theorem $\ref{app:thm:1}$. The other
was layer and co-layer type estimates for
effective equations, which we
summarized in Theorem $\ref{app:thm:2}$.
They should be known by experts most likely,
but there were no precise references for
Theorem $\ref{app:thm:1}$ and we merely
outlined some necessary steps and computations
for the reader's convenience.



\subsection{Notation}\label{notation}

\begin{enumerate}
  \item Notation for estimates.
\begin{enumerate}
  \item $\lesssim$ and $\gtrsim$ stand for $\leq$ and $\geq$
  up to a multiplicative constant,
  which may depend on some given parameters imposed in the paper,
  but never on $\varepsilon$.
  \item We use $\gg$ instead of $\gtrsim$ to indicate that the multiplicative constant is much larger than 1 (but still finite).
  \item We write $\sim$ when both $\lesssim$ and $\gtrsim$ hold.
\end{enumerate}
  \item Notation for derivatives.
  \begin{enumerate}
    \item
  $\nabla v = (\nabla_1 v, \cdots, \nabla_d v)$ is the gradient of $v$, where
  $\nabla_i v = \partial v /\partial x_i$ denotes the
  $i^{\text{th}}$ derivative of $v$.
  $\nabla^2 v = (\nabla^2_{ij} v)_{d\times d}$  denotes the Hessian matrix of $v$, where
  $\nabla^2_{ij} v = \frac{\partial^2 v}{\partial x_i\partial x_j}$.
   \item
  $\nabla\cdot v=\sum_{i=1}^d \nabla_i v_i$
  denotes the divergence of $v$, where
  $v = (v_1,\cdots,v_d)$ is a vector-valued function.
  \item $\nabla_y v$ indicates the gradient of $v$ with respective
  to the variable $y$, while $\Delta_{x} v$ denotes
  the Laplace operator with respective to the variable $x$,
  where $\Delta :=\nabla\cdot\nabla$.
  \item $(\nabla v)^{T}$ represents
  the transpose of $\nabla v$, and
  $e(v):=\frac{1}{2}\big[\nabla v + (\nabla v)^{T}\big]$.
  \end{enumerate}

  \item Geometric notation.
  \begin{enumerate}
  \item $d\geq 2$ is the dimension.
    \item Let $B:=B(x,r)=B_r(x)$, and $nB=B(x,nr)$
    denote the concentric balls as $n>0$ varies.
    \item The layer set of $\Omega$ is denoted
    by $O_{n\varepsilon}
    :=\{x\in\Omega:\text{dist}(x,\partial\Omega)<n\varepsilon\}$.
    The co-layer set is defined by $\Sigma_{n\varepsilon}:=
    \Omega\setminus O_{n\varepsilon}$.
  \item $r_0$ represents
  the diameter of $\Omega$.
\item $\Omega_0$ is the extended region such that
$\Omega_0\supseteq\bar{\Omega}$ and
$\text{dist}(x,\Omega)\sim 10\varepsilon$ for every
$x\in\partial\Omega_0$.
    \item $\delta(x) := \text{dist}(x,\partial\Omega_0)$
    represents the distance function to $\partial\Omega_0$.
    \item Let $\vartheta:\mathbb{R}^{d-1}\to\mathbb{R}$ be
    a Lipschitz function (or $C^{1,\eta}$ function)
    such that $\vartheta(0) = 0$ and
$\|\nabla\vartheta\|_{L^\infty(\mathbb{R}^{d-1})}\leq M_0$
(or $\|\nabla\vartheta\|_{C^{0,\eta}(\mathbb{R}^{d-1})}
\leq M_0$).
For any $r>0$, let
\begin{equation*}
\begin{aligned}
&\Delta_r = \big\{(x^\prime,\vartheta(x^\prime))
\in\mathbb{R}^d:|x^\prime|<r\big\};\\
& D_r= \big\{(x^\prime,t)\in\mathbb{R}^d:
|x^\prime|<r~\text{and}~\vartheta(x^\prime)<t<
\vartheta(x^\prime)+10(M_0+1)r\big\}.
\end{aligned}
\end{equation*}
One may roughly write $D_r = B(0,r)\cap\Omega$, and
$\Delta_r = B(0,r)\cap\partial\Omega$.
    \item $D_r^{\varepsilon}:= D_r \cap \varepsilon\omega$
(half-ball with holes);
 $\Delta_r^\varepsilon :=
\partial D_r\cap \partial\Omega\cap\varepsilon\omega$
(lower bottom of half-ball).
  \end{enumerate}

  \item Notation for spaces and functions.
  \begin{enumerate}
    \item $H^{1}(\Omega_{\varepsilon},\Gamma_{\varepsilon};
\mathbb{R}^d)$
denotes the closure in $H^{1}(\Omega_{\varepsilon};
\mathbb{R}^d)$ of smooth vector-valued functions
vanishing on $\Gamma_{\varepsilon}$
(see \cite[pp.3]{OSY}).
    \item $H^1_{\text{per}}(\omega;\mathbb{R}^d)$
    denotes
the closure in $H^1(Y\cap\omega;\mathbb{R}^d)$
of the set of 1-periodic
$C^{\infty}(\bar{\omega};\mathbb{R}^d)$ functions
(see \cite[pp.5]{OSY}).
    \item The average integral is defined as
\begin{equation}\label{n:1.1}
\dashint_{U}f := \frac{1}{|U|}\int_{U} f(x) dx.
\end{equation}
    \item
Let ${\psi}_{\varepsilon}^{\prime},
{\psi}_{\varepsilon}\in C_{0}^{\infty}(\Omega)$
be cut-off functions, satisfying
\begin{equation}\label{eq:2.1}
\left\{\begin{aligned}
  & 0\leq \psi_{\varepsilon},\psi_{\varepsilon}^{\prime}\leq 1\quad \text{for}~~x\in\Omega ,\\
  & \text{supp}(\psi_{\varepsilon})
  \subset \Omega\backslash O_{3\varepsilon},~~
  \text{supp}(\psi_{\varepsilon}^{\prime})\subset
  \Omega\backslash O_{7\varepsilon},\\
  & \psi_{\varepsilon}=1\quad \text{in}~~
  \Omega\backslash O_{4\varepsilon},\quad
  \psi_{\varepsilon}^{\prime}=1\quad \text{in}~~
  \Omega\backslash O_{8\varepsilon},\\
  & \max\{|\nabla \psi_{\varepsilon}|,
  |\nabla \psi_{\varepsilon}^{\prime}|\}\leq C \varepsilon^{-1}.
  \end{aligned}\right.
\end{equation}
(By the above definition, it's known that
$(1-\psi_{\varepsilon})\psi_{\varepsilon}^{\prime}
=0~\text{in}~ \Omega$.)
\item
The radial maximal operator is defined as
\begin{equation}\label{def:4}
\mathrm{M}_{\text{r}}(\phi)(Q)
= \sup_{t\in(0,c_0)}\big|\phi(Q-t n(Q))\big|
\qquad \text{for~a.e.~} Q\in\partial\Omega,
\end{equation}
in which $c_0$ is a very small number but
$c_0\gg \varepsilon$.
\item The non-tangential maximal function of $u$ is defined by
\begin{equation}\label{def:5}
(\phi)^*(Q) = \sup\big\{ |\phi(x)|:x\in \Gamma_{N_0}(Q)\big\}
\qquad \text{for~a.e.}~ Q\in\partial\Omega,
\end{equation}
where $\Gamma_{N_0}(Q) = \{x\in\Omega:|x-Q|\leq N_0
\text{dist}(x,\partial\Omega)\}$ is the cone with
vertex $Q$ and aperture $N_0$,
and $N_0>1$ depends on the character of $\Omega$.
  \end{enumerate}
\end{enumerate}

Finally, we mention that:
(1)
when we say that the multiplicative constant depends on the character of the domain,
it means that the constant relies on $M_0$;
(2) we shall make a little effort
to distinguish vector-valued functions or
function spaces from their real-valued counterparts,
and they will be clear from the context;
(3) the Einstein's summation convention for repeated indices is
used throughout.


\section{Preliminaries}\label{sec:2}

\begin{definition}\label{def:3}
We call $u_\varepsilon$ the weak solution to
the equation $\eqref{pde:1.1}$, if
there holds
\begin{equation*}
\int_{\Omega_\varepsilon}A(x/\varepsilon)\nabla u_\varepsilon
\nabla\phi dx
= \int_{\Omega_\varepsilon} F\phi dx
\end{equation*}
for any $\phi\in H^1(\Omega_\varepsilon,
\Gamma_\varepsilon;\mathbb{R}^d)$, and
$u_\varepsilon-\tilde{g}\in H^1(\Omega_\varepsilon,
\Gamma_\varepsilon;\mathbb{R}^d)$, where
$\tilde{g}$ is the $H^1$-extension function of $g$, satisfying
$\tilde{g}=g$ on $\partial\Omega$ in the sense of the
trace.
\end{definition}

\begin{lemma}\label{lemma:2.7}
 Suppose that $A$ satisfies $\eqref{eq:1.3}$ and
 \eqref{eq:1.4}, then the effective matrix $\widehat{A}=(\widehat{a}_{ij}^{\alpha\beta})$ defined in \eqref{eq:1.1} satisfies
\begin{equation}\label{a:2}
\left\{\begin{aligned}
  &\widehat{a}_{ij}^{\alpha\beta}=\widehat{a}_{ji}^{\beta\alpha}
  =\widehat{a}_{\alpha j}^{i\beta};\\
  &\widehat{\mu}_{0}|\xi|^{2}\leq
  \widehat{a}_{ij}^{\alpha\beta}\xi^{\alpha}_{i}\xi_{j}^{\beta}
  \leq  \widehat{\mu}_{1}|\xi|^{2}
\end{aligned}\right.
\end{equation}
for any symmetric matrix
$\xi=\{\xi_{i}^{\alpha}\}_{1\leq i,\alpha\leq d}$,
where $\widehat{\mu}_{0},\widehat{\mu}_{1}>0$
depend on $\mu_{0} $ and $\mu_{1}$.
\end{lemma}
\begin{proof}
  See either \cite{JKO} or \cite{OSY}.
\end{proof}

\begin{remark}
The estimate $\eqref{a:2}$ is significant both in
qualitative and quantitative homogenization theories,
especially for some nonlinear models,
and we refer the reader to a recent work \cite{WXZ},
as well as \cite{APMP,ZR}, for this attention.
\end{remark}

\begin{lemma}[flux corrector]\label{lemma:2.6}
  Suppose $B=\{b_{ij}^{\alpha\beta}\}_{1\leq i,j,\alpha,\beta\leq d}
  $ is 1-periodic and satisfies
  $b_{ij}^{\alpha\beta}\in L^{2}_{\emph{loc}}(\mathbb{R}^d)$ with
\begin{equation*}
\emph{(i)}~~\frac{\partial}{\partial y_{i}}b_{ij}^{\alpha\beta}=0,\qquad\qquad
\emph{(ii)}~~\int_{Y}b_{ij}^{\alpha\beta}=0.
\end{equation*}
Then
there exists $E=\{E_{kij}^{\alpha\beta}\}_{1\leq i,j,k,
\alpha,\beta\leq d}$
with $E_{kij}^{\alpha\beta}\in H^{1}_{\emph{loc}}(\mathbb{R}^d)$,
which  is 1-periodic and satisfies
\begin{equation*}
 \frac{\partial}{\partial y_{k}}E_{kij}^{\alpha\beta}=
 b_{ij}^{\alpha\beta}\quad \text{and}\quad
 E_{kij}^{\alpha\beta}=-E_{ikj}^{\alpha\beta}.
\end{equation*}
\end{lemma}
\begin{proof}
See \cite{JKO} and also \cite[Lemma 3.1]{KLS}.
\end{proof}



\subsection{Periodic Cancellations}\label{subsec:2.1}

\begin{definition}\label{def:2}
Fix a nonnegative function $\zeta\in C_0^\infty(B(0,1/2))$,
and $\int_{\mathbb{R}^d}\zeta(x)dx = 1$. Define the smoothing
operator for $f\in L^p(\mathbb{R}^d)$ with $1\leq p<\infty$ as
\begin{equation*}
S_\varepsilon(f)(x) := f*\zeta_\varepsilon(x) = \int_{\mathbb{R}^d} f(x-y)\zeta_\varepsilon(y) dy,
\end{equation*}
where $\zeta_\varepsilon(y)=\varepsilon^{-d}\zeta(y/\varepsilon)$.
\end{definition}

\begin{lemma}\label{lemma:2.3}
Let $f\in C^{\infty}_0(\mathbb{R}^d)$
and $0\leq s\leq 1$.
Then, for any $\varpi\in W^{1,\infty}_{\emph{per}}(Y)$, there holds
\begin{equation}\label{pri:2.2}
\big\|\varpi(\cdot/\varepsilon)f
\big\|_{H^{-s}(\mathbb{R}^d)}
\leq C\varepsilon^{-s}\big\|\varpi\big\|_{W^{1,\infty}(Y)}
\big\|f\big\|_{H^{-s}(\mathbb{R}^d)},
\end{equation}
where the constant $C$ depends only on $d$.
\end{lemma}

\begin{proof}
The idea relies on complex interpolation inequalities and a duality
argument. First of all, we define
$T_\varepsilon (f) := \varpi(\cdot/\varepsilon)f$ on $\mathbb{R}^d$.
Then it is not hard to derive that
\begin{equation*}
\begin{aligned}
\|T_\varepsilon(f)\|_{L^2(\mathbb{R}^d)}
&\leq \|\varphi\|_{L^{\infty}(Y)}\|f\|_{L^{2}(\mathbb{R}^d)};\\
\|T_\varepsilon(f)\|_{H^1(\mathbb{R}^d)}
&\leq C\varepsilon^{-1}
\|\varphi\|_{W^{1,\infty}(Y)}\|f\|_{H^1(\mathbb{R}^d)}.
\end{aligned}
\end{equation*}
This implies the operator norms:
$\|T_\varepsilon\|_{L^2\to L^2}\leq M_1$
and $\|T_\varepsilon\|_{H^1\to H^1}\leq M_2/\varepsilon$,
where $\max\{M_1,M_2\}\lesssim \|\varphi\|_{W^{1,\infty}(Y)}$.
By the complex interpolation inequality (see
for example \cite[Theorem 2.6]{L}), there holds
$$\|T_\varepsilon \|_{H^s\to H^{s}} \leq
\varepsilon^{-s} M_1^{1-s} M_2^{s}
\leq C\varepsilon^{-s} \|\varpi\|_{W^{1,\infty}(Y)},$$
where we use the fact that $H^s(\mathbb{R}^d) :=
\big[L^2(\mathbb{R}^d),H^1(\mathbb{R}^d)\big]_{s}$
(see \cite[pp.57]{L}). For any
$\zeta\in H^{s}(\mathbb{R}^d)$ one may obtain
\begin{equation*}
\int_{\mathbb{R}^d} \varphi(\cdot/\varepsilon)f\zeta dx
\leq \|f\|_{H^{-s}(\mathbb{R}^d)}
\|\varphi(\cdot/\varepsilon)\zeta\|_{H^{s}(\mathbb{R}^d)}
\leq C\varepsilon^{-s}\|\varpi\|_{W^{1,\infty}(Y)}
\|f\|_{H^{-s}(\mathbb{R}^d)}\|\zeta\|_{H^{s}(\mathbb{R}^d)},
\end{equation*}
and this gives the stated estimate $\eqref{pri:2.2}$.
We have completed the proof.
\end{proof}

\begin{lemma}\label{lemma:2.2}
Let $f\in L^p(\mathbb{R}^d)$ with $1\leq p<\infty$
and $\varpi\in L_{\emph{per}}^p(\mathbb{R}^d)$.
Then there holds
\begin{equation}\label{pri:2.3}
\big\|\varpi(\cdot/\varepsilon)S_\varepsilon(f)\big\|_{L^p(\mathbb{R}^d)}
\leq C\big\|\varpi\big\|_{L^p(Y)}\big\|f\big\|_{L^p(\mathbb{R}^d)},
\end{equation}
where $C$ depends on $d$. Moreover,
if $f\in W^{1,p}(\mathbb{R}^d)$ for some $1<p<\infty$, then we have
\begin{equation}\label{pri:2.4}
\big\|S_\varepsilon(f)-f\big\|_{L^p(\mathbb{R}^d)}
\leq C\varepsilon\big\|\nabla f\big\|_{L^p(\mathbb{R}^d)},
\end{equation}
where the constant $C$ depends only on $d$.
\end{lemma}

\begin{proof}
See \cite[Lemmas 2.1 and 2.2]{S}.
\end{proof}

Recall the notation $\delta$ and $\Sigma_{2\varepsilon}$
defined in Subsection $\ref{notation}$.

\begin{lemma}\label{lemma:2.10}
Let $f\in L^2(\Omega)$ be supported in $\Sigma_{2\varepsilon}$,
and $\varpi\in L^{2}_{\emph{per}}(Y)$. Then we obtain
\begin{equation}\label{pri:2.11}
\bigg(\int_{\Sigma_{2\varepsilon}}
|\varpi(x/\varepsilon)S_{\varepsilon}(f)|^2
\delta^{\pm 1}dx\bigg)^{1/2}
\leq C \|\varpi\|_{L^{2}(Y)}\bigg(\int_{\Sigma_{2\varepsilon}}
|f|^2\delta^{\pm 1}dx\bigg)^{1/2},
\end{equation}
in which the constant $C$ depends only on $d$.
\end{lemma}
\begin{proof}
See \cite[Lemma 3.2]{Q1}.
\end{proof}
\begin{lemma}\label{lemma:2.04}
Let $f\in H^1(\Omega)$ be supported
in $\Sigma_{\varepsilon}$.
Then one may acquire that
\begin{equation}\label{pri:2.7}
\bigg(\int_{\Sigma_{2\varepsilon}}
|f-S_{\varepsilon}(f)|^2\delta dx\bigg)^{1/2}
\leq C \varepsilon\bigg(\int_{\Sigma_{\varepsilon}}
|\nabla f|^2\delta dx\bigg)^{1/2},
\end{equation}
where the constant $C$ depends only on $d$.
\end{lemma}
\begin{proof}
See \cite[Lemma 3.3]{Q1}.
\end{proof}

\subsection{Extension Theories}\label{subsec:2.2}

Our next goal is to extend function in
$H^{1}(\Omega_{\varepsilon},\Gamma_{\varepsilon};\mathbb{R}^d)$
(see Subsection $\ref{notation}$)
to become functions in Sobolev space $H^1_0(\Omega_0;\mathbb{R}^d)$.
Later on, we generalize it with a homogenous weight.

\begin{lemma}[improved extension property]\label{extensiontheory}
Let $0<\varepsilon<1$.
Suppose that $\Omega$ and $\Omega_{0}$ are
bounded Lipschitz domains satisfying
$\bar{\Omega}\subset\Omega_{0}$
and \emph{dist}$(x,\Omega)\sim 10\varepsilon$
for $x\in\partial\Omega_0$.
Then there exists a linear extension operator
$P_{\varepsilon}:H^{1}(\Omega_{\varepsilon},
\Gamma_{\varepsilon};\mathbb{R}^d)
\rightarrow H^{1}_{0}(\Omega_{0};\mathbb{R}^d)$ such that
\begin{equation}\label{pri:2.1}
  \begin{aligned}
  & \|P_{\varepsilon}w\|_{H^{1}_{0}(\Omega_{0})}
  \leq C_{1} \|w\|_{H^{1}(\Omega_{\varepsilon})};\\
  & \|\nabla P_{\varepsilon}w\|_{L^{2}(\Omega_{0})}
  \leq C_{2} \|\nabla w\|_{L^{2}(\Omega_{\varepsilon})};\\
  &\|e(P_{\varepsilon}w)\|_{L^{2}(\Omega_{0})}\leq C_{3} \|e( w)\|_{L^{2}(\Omega_{\varepsilon})}
  \end{aligned}
\end{equation}
hold for some constants $C_{1},C_{2},C_{3}$
depending only on the characters of $\Omega$ and $\omega$,
where $e(w)$ denotes the symmetric part of $\nabla w$, defined in
Subsection $\ref{notation}$.
\end{lemma}
\begin{proof}
Compared to \cite[pp.50, Theorem 4.3]{OSY}, the only
modification is the condition
``$\text{dist}(x,\partial\Omega)\sim 10\varepsilon$ for
$x\in\partial\Omega_0$''.
We denote the corresponding notation
in \cite[Theorem 4.3]{OSY} by $\tilde{P}_\varepsilon$,
$\tilde{\Omega}_0$, where $\tilde{\Omega}_0\supseteq \Omega$ with
$\text{dist}(\partial\tilde{\Omega}_0,\Omega)>1$. To achieve our
goal, define
\begin{equation*}
 P_\varepsilon (w) := \psi_{\varepsilon}\tilde{P}_\varepsilon(w)
 \quad \text{and}\quad
 \Omega_0^{2\varepsilon} :=
 \{x\in\tilde{\Omega}_0: \text{dist}(x,\partial\Omega)<2\varepsilon\}
 \cup\Omega,
\end{equation*}
where $\psi_\varepsilon\in C_0^1(\Omega_0)$ be a cut-off function,
satisfying $\psi_{\varepsilon} = 1$ on $\Omega$,
$\psi_{\varepsilon} = 0$ outside $\Omega_0^{2\varepsilon}$ and
$|\nabla\psi_{\varepsilon}|\lesssim 1/\varepsilon$.
Then we have $P_\varepsilon(w)\in H_0^1(\Omega_0)$, and
\begin{equation}\label{f:2.1}
\begin{aligned}
\|\nabla P_\varepsilon(w)\|_{L^2(\Omega_0)}
&\lesssim \varepsilon^{-1}
\|\tilde{P}_\varepsilon (w)
\|_{L^2(\Omega_0^{2\varepsilon}\setminus\Omega)}
+ \|\nabla \tilde{P}_\varepsilon (w)\|_{L^2(\Omega_0)}\\
&\lesssim \|\nabla \tilde{P}_\varepsilon (w)
\|_{L^2(\Omega_0^{2\varepsilon}\setminus\Omega)}
+ \|\nabla \tilde{P}_\varepsilon (w)\|_{L^2(\Omega_0)}
\lesssim \|\nabla w\|_{L^2(\Omega_\varepsilon)},
\end{aligned}
\end{equation}
where we employ Poincar\'e's inequality (noting that
$w=0$ on $\Gamma_\varepsilon$ and
using the same argument as in Lemma $\ref{lemma:2.5}$) in the second step, and
\cite[Theorem 4.3]{OSY} in the last one.
This proved the second inequality of $\eqref{pri:2.1}$.
The third one
of $\eqref{pri:2.1}$ follows from the same approach given for
$\eqref{f:2.1}$, in which Poincar\'e's inequality is replaced by
Korn's inequality. We have completed the proof.
\end{proof}

\begin{lemma}\label{lemma:2.4}
There exists a constant $C$, independent of $\varepsilon$, such that
\begin{equation*}
\|w\|_{H^1(\Omega_{\varepsilon})}\leq C\|e(w)\|_{L^2(\Omega_{\varepsilon})}
\end{equation*}
for any $w\in H^{1}(\Omega_{\varepsilon},\Gamma_{\varepsilon};\mathbb{R}^d)$.
\end{lemma}
\begin{proof}
  See \cite[pp.53, Theorem 4.5]{OSY}.
\end{proof}


\begin{lemma}\label{lemma:2.5}
For $w\in H^{1}(\Omega_{\varepsilon},
\Gamma_{\varepsilon};\mathbb{R}^d)$, let
$\tilde{w}$ be the extension of $w$ in the way of
Lemma $\ref{extensiontheory}$. Then there holds
\begin{equation}\label{pri:2.8}
  \|\tilde{w}\|_{L^{2}(O_{4\varepsilon})}\leq C \varepsilon
   \|\nabla\tilde{w}\|_{L^{2}(\Omega)},
\end{equation}
where $C$ depends on $d,r_0$ and the characters of
$\Omega$ and $\omega$.
\end{lemma}
\begin{proof}
  See \cite[Lemma 3.4]{BR}.
\end{proof}

For the ease of the statement, we impose
the following notation in this subsection.
Let $\rho$ be a weight, and the $L^2$-weighted norm
is defined by $\|\cdot\|_{L^2_\rho(\Omega)}
:=(\int_{\Omega}|\cdot|^2\rho dx)^{1/2}$. Then
we denote the homogenous weighted
space by $H_\rho^1(\Omega):=\{f\in  H^1_{\text{loc}}(\Omega):\|\nabla f\|_{L^2_\rho(\Omega)}<\infty\}$.

\begin{lemma}\label{lemma:extension}
Let $G\subset D\subset\mathbb{R}^d$ and let
each of the sets $G,D$ and $D\setminus\bar{G}$ be a non-empty
bounded Lipschitz domain.
Suppose that $\partial G\cap D$ is
non-empty.
Let $\rho\in A_{2}$. Then for vector-valued functions in
$H^1_{\rho}(D\setminus\bar{G})$ there is a linear extension operator
$\Lambda: H^1_{\rho}(D\setminus\bar{G})
\to H_{\rho}^1(D)$ such that
\begin{equation}\label{pri:2.6}
\|\nabla \Lambda(w)\|_{L^2_\rho(D)}
\lesssim \|\nabla w\|_{L^2_\rho(D\setminus\bar{G})},
\end{equation}
where the up to constant depends on  $d$ and
the characters of $D\setminus\bar{G}$ and $D$.
\end{lemma}

\begin{proof}
The main idea in the proof is paralleled to \cite[Lemma 4.1]{OSY},
while we impose some weighted estimates here, and
we provide a proof for the reader's convenience.
To do so, we construct an auxiliary equation
\begin{equation}\label{pde:2.1}
\left\{\begin{aligned}
\nabla\cdot\bar{a}\nabla W &= 0 &\quad&\text{in}~G,\\
n\cdot\bar{a}\nabla W &= 0 &\quad&\text{on}~
\partial G\cap\partial D,\\
W &= \tilde{w} &\quad&\text{on}~
\partial G\cap  D,
\end{aligned}\right.
\end{equation}
where $\tilde{w}$ is an weighted $H_{\rho}^1$-extension of $w$ such that
$\tilde{w} = w$ on $D\setminus\bar{G}$ and
\begin{equation}\label{f:2.2}
 \int_{\mathbb{R}^d} |\nabla\tilde{w}|^2\rho dx
 \lesssim \int_{D\setminus G}
 |\nabla w|^2\rho dx,
\end{equation}
in which the up to constant is independent of radius of
$D\setminus G$ (see \cite[Theorem 1.2]{C}). Then we construct
the extension map as
\begin{equation}\label{}
 \Lambda (w)(x) = \left\{
 \begin{aligned}
 w(x) &~ &\quad&\text{if}~~ x\in D\setminus \bar{G};\\
 W(x) &~ &\quad&\text{if}~~ x\in \bar{G}.
 \end{aligned}\right.
\end{equation}
It is known that $W\in H^1(G)$, and one may further derive
the weighted estimate
\begin{equation}\label{f:2.3}
 \int_{G}|\nabla \Lambda(w)|^2\rho dx
 \lesssim \int_{G}|\nabla \tilde{w}|^2\rho dx
\end{equation}
for any $\rho\in A_2$. Admitting the estimate
$\eqref{f:2.3}$ for a while and this together with
the estimate $\eqref{f:2.2}$ gives the desired estimate
$\eqref{pri:2.6}$. In fact,
the estimate $\eqref{f:2.3}$ is
included in Theorem $\ref{app:thm:1}$, and
we have completed the proof.
\end{proof}

The terminology of
``perforated domains of type II'' is taken from
\cite[pp.43-44]{OSY}.

\begin{lemma}[extension of functions in
perforated domains of type II]\label{lemma:2.8}
Let $\Omega_\varepsilon$ be a perforated domain of type II.
Let $\delta_{\partial\Omega}(x)=\emph{dist}(x,\partial\Omega)$ be a distance function
and $\rho =\delta^{\beta}_{\partial\Omega}$ with $-1<\beta<1$.
Then there exits a linear extension operator
$\Lambda_\varepsilon: H^1_{\rho}(\Omega_\varepsilon)\to
H_{\rho}^1(\Omega)$ such that
\begin{equation}\label{pri:2.9}
 \int_{\Omega}|\nabla \Lambda_\varepsilon(w)|^2\rho dx
 \lesssim \int_{\Omega_\varepsilon}|\nabla w|^2\rho dx,
\end{equation}
where the up to constant is independent of
$\varepsilon, \beta$ and
$w$.
\end{lemma}

\begin{proof}
The main idea is similar to that in \cite[Lemma 4.2]{OSY}.
Let $\Lambda_\varepsilon(w)(x) = \Lambda_1(\tilde{w})(x/\varepsilon)$,
where $\tilde{w}(y)=w(\varepsilon y) = w(x)$ and
$y\in\Omega/\varepsilon$ and $x\in\Omega$.
Let $\tilde{\rho}(y) = \big[
\text{dist}(y,\partial\Omega/\varepsilon)\big]^\beta$.
Due to the homogeneity of the Euclidean distance,
we have $ \tilde{\rho}(y) = \varepsilon^{-\beta}\rho(x)$,
and therefore there still holds $\tilde{\rho}\in A_2$. Hence, if
the holes have no intersection with $\partial Q$, then
it follows from Lemma $\ref{lemma:extension}$ that
\begin{equation*}
\int_{Q+z}|\nabla_y\Lambda_1(\tilde{w})(y)|^2
\tilde{\rho}(y)dy
\lesssim \int_{(Q+z)\cap\omega}|\nabla_y\tilde{w}(y)|^2
\tilde{\rho}(y)dy,
\end{equation*}
where the up to constant is independent of $z$.
By changing variable $y=x/\varepsilon$, we obtain
\begin{equation}\label{f:2.5}
\int_{\Omega}|\nabla \Lambda_\varepsilon(w)|^2\rho(x)dx
\lesssim \int_{\Omega_\varepsilon}|\nabla w|^2\rho(x)dx,
\end{equation}
in which we employ the fact that $ \tilde{\rho}(y) =
\varepsilon^{-\beta}\rho(x)$.
This proved the estimate $\eqref{pri:2.9}$ in such the case.

Now, we turn to the case that the hole $\bar{Q}\setminus \omega$
has a non-empty intersection with $\partial Q$. It is known that
$\Lambda_1(\tilde{w})$ may not
belong to $H^1(\Omega/\varepsilon)$, since its traces on
the adjacent faces of the cubes
(taking limitation from different sides) may not be equal.
However, we can mimic the idea in the proof of
\cite[Lemma 4.2]{OSY} to overcome the same difficulty, and
we provide a proof for the sake of the completeness.
Let $\gamma_1,\cdots,\gamma_{N}$ be these holes
which has non-empty intersection with $\partial Q$, where
$l=0,1$, and $N$ goes to infinity as $\varepsilon\to 0$.
Moreover, for any $m=1,\cdots,N$,
the notation $\gamma_m^{l}$ with $l=0,1$ represents one of
two parts of $\gamma_m$, which is divided by $\partial Q$.
It is fine to assume that we first
extend the function from $\gamma_m^{1}$ parts to the
related boundary $\gamma_m^{1}\cap\partial Q$, and then
setting $\gamma_0 = \cup_{m=1}^{N}\gamma_{m}^{0}$ one may
extend the function from the region
$\Omega^\prime:=(\Omega/\varepsilon)\setminus \gamma_0$ to the whole
domain $\Omega/\varepsilon$. Although $N$ tends to infinity
as $\varepsilon$ goes to zero, the family of $\gamma_{m}^0$
comes from the shifts of a finite number of bounded Lipschiz
domains. So, the up to constant in the later computations
is independent of $N$.

According to the previous step and the result of Lemma
$\ref{lemma:extension}$, one may construct
the extension function of $\Lambda_1(\tilde{w})(y)$ such that
$\Lambda_1(\tilde{w})$ belongs to $H^1(\Omega^\prime)$, and
\begin{equation}\label{f:2.4}
\int_{\Omega^\prime}|\nabla_y\Lambda_1(\tilde{w})(y)|^2
\tilde{\rho}(y)dy
\lesssim \int_{U\cap\omega}|\nabla_y\tilde{w}(y)|^2
\tilde{\rho}(y)dy,
\end{equation}
where we use the notation $\bar{U}
:= \cup_{z\in T_\varepsilon}(z+\bar{Q})$,
and $T_\varepsilon$ is the subset of
$\mathbb{Z}^n$ consisting of all $z$ such that
$\varepsilon(z+Q)\subset\Omega$ and
$\text{dist}(\varepsilon(z+Q),\partial \Omega)\geq\varepsilon$.
Then for some $m=1,\cdots,N$,
we may choose the neighbourhood of $\gamma_{m}^0$, denoted by
$\tilde{\gamma}_{m}^0$, such that
$\tilde{\gamma}_{m}^0\supset \gamma_m^0$,
and it follows from Lemma $\ref{lemma:extension}$ that
\begin{equation*}
\int_{\gamma_m^0}|\nabla_y\Lambda\Lambda_1(\tilde{w})(y)|^2
\tilde{\rho}(y)dy
\lesssim \int_{\tilde{\gamma}_m^0\setminus\gamma_{m}^0}
|\nabla_y\Lambda_1(\tilde{w})(y)|^2
\tilde{\rho}(y)dy.
\end{equation*}
On account of the periodicity, the above estimate implies
\begin{equation*}
\int_{\gamma_0}
|\nabla_y\Lambda\Lambda_1(\tilde{w})(y)|^2
\tilde{\rho}(y)dy
\lesssim
\int_{\Omega^\prime}|\nabla_y\Lambda_1(\tilde{w})(y)|^2
\tilde{\rho}(y)dy
\end{equation*}
and this together with $\eqref{f:2.4}$ and
$\Lambda\Lambda_1(\tilde{w}) = \Lambda_1(\tilde{w})$ on
$\Omega^\prime$ gives
\begin{equation*}
\int_{\Omega/\varepsilon}
|\nabla_y\Lambda\Lambda_1(\tilde{w})(y)|^2
\tilde{\rho}(y)dy
\lesssim \int_{U\cap\omega}|\nabla_y\tilde{w}(y)|^2
\tilde{\rho}(y)dy.
\end{equation*}
In such the case, set $\Lambda_\varepsilon (w)(x)
 := \Lambda\Lambda_1(\tilde{w})(x/\varepsilon)$ and
this completes the whole argument.
\end{proof}

\noindent \textbf{The proof of Theorem \ref{thm:1.5}}.
The main idea comes from \cite[Theorem 4.3]{OSY}, and
we still take the notation imposed in
Lemmas $\ref{lemma:extension}$ and $\ref{lemma:2.8}$.
For any $z\in\mathbb{Z}^d$ such that
$\varepsilon(z+Q\cap\omega)\cap\Omega\not=\emptyset$,
collect such $z$ to be the index set $T_\varepsilon$.
Then recalling the interior of
$\cup_{z\in T_\varepsilon}
(z+\bar{Q})$ written
by $U$,
we denote the interior of $\cup_{z\in T_\varepsilon}
\varepsilon(z+\overline{Q\cap\omega})$ by $U_\varepsilon$.
For any $w\in H^1(\Omega_\varepsilon,\Gamma_\varepsilon;\mathbb{R}^d)$,
we impose the following zero-extension:
\begin{equation*}
\bar{w}(x)
= \left\{\begin{aligned}
&w(x), &\quad& x\in\Omega_\varepsilon;\\
&0, &\quad& x\in U_\varepsilon\setminus\Omega;\\
&0, &\quad& x\in\Omega_0\setminus (\varepsilon U),
\end{aligned}\right.
\end{equation*}
and $\bar{w}\in H^1(U_\varepsilon;\mathbb{R}^d)$.
Hence, the problem has been reduced to the perforated domain of
type II, and
$\bar{w}$ may be further extended
to the whole region $\Omega_0$ via Lemma $\ref{lemma:2.8}$,
denoted by $\Lambda_\varepsilon(\bar{w})$.
Setting $\bar{\Lambda}_\varepsilon(w):
=\Lambda_\varepsilon(\bar{w})$,
it is clear to see $\bar{\Lambda}_\varepsilon(w)
= 0$ on $\partial\Omega_0$. Moreover, we have
\begin{equation*}
 \int_{\Omega_0}|\nabla \bar{\Lambda}_\varepsilon(w)|^2\delta^{\beta} dx
 \lesssim^{\eqref{pri:2.9}}
 \int_{U_\varepsilon}|\nabla \bar{w}|^2\delta^{\beta} dx
 \lesssim
 \int_{\Omega_\varepsilon}|\nabla w|^2\delta^{\beta} dx.
\end{equation*}
This yields the desired estimate $\eqref{pri:1.9}$.
Abusing
notation to rewrite $\bar{\Lambda}_\varepsilon$ as
$\Lambda_\varepsilon$, we
have completed the whole proof.
\qed


\section{Convergence Rates in $H^1$-norm}\label{section3}

\begin{lemma}[weak formulation]\label{lemma:2.1}
Assume the same conditions as in Theorem $\ref{thm:2.1}$.
Let $w_\varepsilon$ be given in $\eqref{eq:3.1}$,
and $E = (E_{kij})$ be the flux
corrector defined in Lemma $\ref{lemma:2.6}$,
and $\theta$ defined in $\eqref{eq:1.1}$. Then,
for any $\phi\in H^{1}_{0}(\Omega;\mathbb{R}^d)$, we have
\begin{equation}\label{eq:2.2}
\begin{aligned}
\int_{\Omega_{\varepsilon}}A(x/\varepsilon)
\nabla w_{\varepsilon}\nabla \phi dx&
=\int_{\Omega}(l_{\varepsilon}^{+}-\theta)F\phi dx\\
&+\int_{\Omega}(\theta\widehat{A}-l_{\varepsilon}^{+}A^{\varepsilon})
(\nabla u_{0}-\varphi)\nabla\phi dx
-\varepsilon\int_{\Omega}\varpi(x/\varepsilon)
\nabla\varphi\nabla\phi dx,
\end{aligned}
\end{equation}
where the notation $\varpi
:= E +
l^{+}A
\tilde{\chi}$ belongs to $L^2_{\emph{per}}(Y)$,
and $\tilde{\chi}$ is an extension of $\chi$
satisfying $\tilde{\chi}=\chi$ in $\omega$ and
$\tilde{\chi}=0$ in $\mathbb{R}^d\backslash\omega$.
\end{lemma}
\begin{proof}
In terms the definition of the weak solution to
\eqref{pde:1.1} and \eqref{pde:1.3}, respectively, there hold
\begin{equation}\label{eq:2.3}
 \int_{\Omega_{\varepsilon}} A(x/\varepsilon)\nabla u_\varepsilon\nabla \phi dx=\int_{\Omega_{\varepsilon}}F\phi dx
=\int_{\Omega}l_{\varepsilon}^{+}F\phi dx
\end{equation}
(see Definition $\ref{def:3}$) and
\begin{equation}\label{eq:2.4}
\int_{\Omega} \widehat{A}\nabla u_0\nabla\phi dx
=\int_{\Omega}F\phi dx
\end{equation}
for any $\phi\in H_0^1(\Omega;\mathbb{R}^d)$,
where we mention that $\phi\in H^{1}(\Omega_{\varepsilon},
\Gamma_{\varepsilon};\mathbb{R}^d)$ in \eqref{eq:2.3}.
Recalling $w_{\varepsilon} = u_\varepsilon - v_\varepsilon$ with
$v_\varepsilon = u_0 + \varepsilon\chi_\varepsilon\varphi$,
it follows that
 \begin{equation}\label{sharp1}
 \begin{aligned}
 &\quad\int_{\Omega_{\varepsilon}}A_{\varepsilon}
 \nabla w_{\varepsilon}\nabla \phi dx \\
 & =\int_{\Omega_{\varepsilon}}A_{\varepsilon}\nabla u_{\varepsilon}\nabla \phi dx-\int_{\Omega_{\varepsilon}}A_{\varepsilon}\nabla v_{\varepsilon}\nabla \phi dx\\
 & =\int_{\Omega_{\varepsilon}}A_{\varepsilon}\nabla u_{\varepsilon}\nabla \phi dx
 -\int_{\Omega}\theta\widehat{A}\nabla u_{0}\nabla \phi dx+\int_{\Omega}\theta\widehat{A}\nabla u_{0}\nabla \phi dx
 -\int_{\Omega_{\varepsilon}}A_{\varepsilon}\nabla v_{\varepsilon}\nabla \phi dx\\
 & =^{\eqref{eq:2.3},\eqref{eq:2.4}}
 \int_{\Omega}l_{\varepsilon}^{+}F\phi dx
 -\int_{\Omega}\theta F\phi dx+\int_{\Omega}
 \theta\widehat{A}\nabla u_{0}\nabla \phi dx
 -\int_{\Omega}\theta\widehat{A}\varphi\nabla \phi dx
 +\int_{\Omega}\theta\widehat{A}\varphi\nabla \phi dx\\
 &\qquad\quad -\int_{\Omega}l_{\varepsilon}^{+}A_{\varepsilon}\nabla u_{0}\nabla\phi dx-\int_{\Omega}l_{\varepsilon}^{+}A_{\varepsilon}\nabla \tilde{\chi}_{\varepsilon}\varphi\nabla\phi dx-\varepsilon\int_{\Omega}l_{\varepsilon}^{+}A_{\varepsilon} \tilde{\chi}_{\varepsilon}\nabla \varphi\nabla\phi dx
 \end{aligned}
 \end{equation}
By a routine calculation, the right-hand side of \eqref{sharp1} is equal to
\begin{equation*}
\begin{aligned}
\int_{\Omega}(l_{\varepsilon}^{+}-\theta)F\phi dx
  &+\int_{\Omega}(\theta\widehat{A}-l_{\varepsilon}^{+}A^{\varepsilon})
  (\nabla u_{0}-\varphi)\nabla\phi dx\\
  &\quad
  +\underbrace{\int_{\Omega}(\theta\widehat{A}-l_{\varepsilon}^{+}
  A^{\varepsilon}
  -l_{\varepsilon}^{+}
  A^{\varepsilon}\nabla\tilde{\chi}_{\varepsilon})
  \varphi\nabla\phi dx}_{T}
  -\varepsilon\int_{\Omega}l_{\varepsilon}^{+}A^{\varepsilon} \tilde{\chi}_{\varepsilon}\nabla \varphi\nabla\phi dx,
  \end{aligned}
\end{equation*}
and so the remainder of the proof is to handle the term
$T$. Let
$b(y) = \theta \widehat{A}-l^{+}A(y)-l^+A(y)\nabla\tilde{\chi}$.
On account of
the antisymmetry property of flux corrector in
Lemma $\ref{lemma:2.6}$, one may derive that
\begin{equation*}
\begin{aligned}
T=\int_{\Omega}b(x/\varepsilon)
  \varphi\nabla\phi dx
& =\varepsilon\int_{\Omega}\frac{\partial}{\partial x_{k}}\{E_{kij}(x/\varepsilon)\}\varphi_{j}
  \frac{\partial\phi }{\partial x_{i}}dx \\
& = \varepsilon\underbrace{\int_{\Omega}\frac{\partial}{\partial x_{k}}
\{E_{kij}(x/\varepsilon)\varphi_{j}\}
\frac{\partial\phi }{\partial x_{i}}dx}_{=0}
-\varepsilon\int_{\Omega}
E_{kij}(x/\varepsilon)\frac{\partial}{\partial x_{k}}
\big\{\varphi_{j}\big\}
\frac{\partial\phi }{\partial x_{i}}dx.
\end{aligned}
\end{equation*}
and then set $\varpi := E+l^{+}A\tilde{\chi}$ on $Y$.
This completes the proof.
\end{proof}

\begin{lemma}\label{lemma:3.3}
Assume the same conditions as in Theorem $\ref{thm:2.1}$.
Let $w_\varepsilon$ be given in $\eqref{eq:3.1}$
and the periodic tensor $\varpi$ be defined in Lemma $\ref{lemma:2.1}$.
Then
there holds
\begin{equation}\label{pri:2.5}
\begin{aligned}
&\|\nabla w_{\varepsilon}\|_{L^{2}(\Omega_{\varepsilon})}
\lesssim\varepsilon^{1-s}\|F\|_{H^{1-s}(\mathbb{R}^d)}\\
&\qquad\quad+ \Bigg\{\varepsilon \bigg(\|F\|_{L^2(\Omega)}
+\|\varpi(\cdot/\varepsilon)\nabla\varphi
\big\|_{L^{2}(\Omega)}\bigg)
+\|\nabla u_{0}-\varphi\|_{L^{2}(\Omega)}+\|\nabla u_{0}\|_{L^{2}(O_{2\varepsilon})}\Bigg\},
\end{aligned}
\end{equation}
where
the up to constant depends on $\mu_{0},\mu_{1},d,\omega$ and
$\Omega$.
\end{lemma}

\begin{proof}
Firstly, we note that $w_{\varepsilon}\in H^{1}(\Omega_{\varepsilon},
\Gamma_{\varepsilon};\mathbb{R}^d)$.
By the extension results stated in Lemma \ref{extensiontheory},
one may extend $w_{\varepsilon}$ from
$H^{1}(\Omega_{\varepsilon}, \Gamma_{\varepsilon};
\mathbb{R}^d)$ to $H^{1}_{0}(\Omega_{0};\mathbb{R}^d)$,
denoted by $\tilde{w}_{\varepsilon}$.
 Then we take $w_{\varepsilon}$ and
 $\psi_{\varepsilon}^{\prime}\tilde{w}_{\varepsilon}$
 as the test function in \eqref{eq:2.3} and \eqref{eq:2.4},
 respectively, in which
 the cut-off function $\psi^\prime_\varepsilon$ is defined in $\eqref{eq:2.1}$.
 By a similar calculation as we did in Lemma \ref{lemma:2.1}, we obtain
\begin{equation}\label{eq:3.2}
\begin{aligned}
  \int_{\Omega_{\varepsilon}}A(x/\varepsilon)\nabla w_{\varepsilon}\nabla w_{\varepsilon}dx
  =&\int_{\Omega}(l_{\varepsilon}^{+}
  -\theta\psi_{\varepsilon}^{\prime})
  F\tilde{w}_{\varepsilon}dx\\
  &~\quad-\theta\int_{\Omega} \widehat{A}\nabla u_{0}
  \nabla[(1-\psi_{\varepsilon}^{\prime})\tilde{w}_{\varepsilon}]dx\\
  &~\qquad \quad+\int_{\Omega}
  [\theta\widehat{A}-l_{\varepsilon}^+A(x/\varepsilon)]
  [\nabla u_{0}-\varphi]\nabla \tilde{w}_{\varepsilon}dx\\
  &\qquad \qquad \qquad-\varepsilon\int_{\Omega} \varpi(x/\varepsilon)
  \nabla \varphi\nabla \tilde{w}_{\varepsilon}dx
  =:I_{1}+I_{2}+I_{3}+I_{4},
  \end{aligned}
\end{equation}
recalling $\varpi= E+
  l^{+}A\tilde{\chi}$ defined in Lemma $\ref{lemma:2.1}$.

We will compute each $I_{i}$ for $i=1,2,3,4$.
The arguments developed for the first term $I_1$ seem to
be new while the techniques used in the remainder terms
are similar to those shown in the case of unperforated domains.
We first handle the term $I_{1}$, and
\begin{equation*}
  I_{1}=\int_{\Omega}(l_{\varepsilon}^{+}-\theta)
  F\tilde{w}_{\varepsilon}\psi_{\varepsilon}^{\prime}dx
  +\int_{\Omega}(1-\psi_{\varepsilon}^{\prime})l_{\varepsilon}^{+}
  F\tilde{w}_{\varepsilon}dx:=I_{11}+I_{12}.
\end{equation*}
Since $\text{supp}(1-\psi_{\varepsilon}^{\prime})
\subseteq O_{8\varepsilon}$,
we have
\begin{equation*}
  \begin{aligned}
  |I_{12}|& \leq \int_{O_{8\varepsilon}}|F\tilde{w}_{\varepsilon}|dx
  \leq \|F\|_{L^{2}(O_{8\varepsilon})}
  \|\tilde{w}_{\varepsilon}\|_{L^{8}(O_{8\varepsilon})}
  \lesssim^{\eqref{pri:2.8}} C\varepsilon \|F\|_{L^{2}({\Omega})}
  \|\nabla \tilde{w}_{\varepsilon}\|_{L^{2}(\Omega)}.
  \end{aligned}
\end{equation*}
To deal with the first term $I_{11}$, we consider the
auxiliary equation:
\begin{equation}\label{auxi1}
  \left\{\begin{aligned}
   -\Delta \Psi(y)&=l^{+}(y)-\theta~~\text{in}~ Y,\\
   \dashint_{Y}\Psi dy&=0,~\Psi\in H^{1}_{\text{per}}(Y).
  \end{aligned}\right.
\end{equation}

On account of $\int_{Y}(l^{+}(y)-\theta) dy=0,$
it is known that \eqref{auxi1} owns an unique solution
$\Psi\in H^{1}_{\text{per}}(Y)$. By interior Schauder's estimates
we obtain $\|\nabla\Psi\|_{C^{1,\alpha}(Y)}\lesssim 1$. This gives
\begin{equation*}
  \begin{aligned}
  |I_{11}|
  & =\Big|-\varepsilon^2
  \int_{\Omega}\Delta_{x}\Psi(x/\varepsilon)
  F\tilde{w}_{\varepsilon}\psi_{\varepsilon}^{\prime}dx\Big|
   =\varepsilon\Big|\int_{\Omega}
  \nabla_y \Psi(y)
  \cdot\nabla(F\tilde{w}_{\varepsilon}
  \psi_{\varepsilon}^{\prime})dx
  \Big|\\
  &\leq \varepsilon\underbrace{\Big|\int_{\Omega}
  \nabla_y \Psi(y)
  \cdot\nabla F\tilde{w}_{\varepsilon}\psi_{\varepsilon}^{\prime}dx
  \Big|}_{I_{11a}}
  +\varepsilon\underbrace{\Big|\int_{\Omega}
  \nabla_y \Psi(y)
  \cdot \nabla(\tilde{w}_{\varepsilon}\psi_{\varepsilon}^{\prime})Fdx
  \Big|}_{I_{11b}},
  \end{aligned}
\end{equation*}
in which $y=x/\varepsilon$.
The easier term is
\begin{equation*}
\begin{aligned}
I_{11b}
&\lesssim
\|F\|_{L^{2}(\Omega)}\Big\{
  \|\nabla\tilde{w}_{\varepsilon}\|_{L^{2}(\Omega)}
+ \varepsilon^{-1}
  \|\tilde{w}_{\varepsilon}\|_{L^{2}(O_{8\varepsilon})}\Big\}
\lesssim^{\eqref{pri:2.8}}
\|F\|_{L^{2}(\Omega)}
  \|\nabla\tilde{w}_{\varepsilon}\|_{L^{2}(\Omega)},
\end{aligned}
\end{equation*}
while we proceed to address the term $I_{11a}$, and
\begin{equation*}
\begin{aligned}
I_{11a}
&\leq  \|\nabla\Psi(\cdot/\varepsilon)
  \cdot\nabla F\|_{H^{-s}(\mathbb{R}^d)}\|\tilde{w}_{\varepsilon}
  \psi_{\varepsilon}^{\prime}\|_{H^{s}_{0}(\Omega)}\\
&\lesssim^{\eqref{pri:2.2}} \varepsilon^{-s} \|\nabla F\|_{H^{-s}(\mathbb{R}^d)}
\|\tilde{w}_{\varepsilon}
  \psi_{\varepsilon}^{\prime}\|_{H^1(\Omega)}\\
&\lesssim \varepsilon^{-s} \|F\|_{H^{1-s}(\mathbb{R}^d)}
\bigg\{
\|\nabla \tilde{w}_{\varepsilon}\|_{L^2(\Omega)}
+\varepsilon^{-1}\|\tilde{w}_\varepsilon\|_{L^2(O_{8\varepsilon})}
\bigg\}\\
&\lesssim^{\eqref{pri:2.8}} \varepsilon^{-s} \|F\|_{H^{1-s}(\mathbb{R}^d)}
\|\nabla \tilde{w}_{\varepsilon}\|_{L^2(\Omega)},
\end{aligned}
\end{equation*}
where we also employ \cite[Proposition 2.2]{DPV} in the second
step, as well as,
Poincar\'e's inequality in the last one.
(In the third step, we use
that $\|\nabla F\|_{H^{-s}(\mathbb{R}^d)}
\leq \|F\|_{H^{1-s}(\mathbb{R}^d)}$,
and it may be observed from Plancherel's identity coupled
with Fourier transform, while the definition of $H^{s}(\mathbb{R}^d)$
with $s\in\mathbb{R}$ via
Fourier transform is equivalent to that given by Gagliardo norm
(see for example \cite[Proposition 3.4]{DPV})).
Plugging
the terms $I_{11a}, I_{11b}$ back into $I_{11}$
and then combining $I_{12}$ leads to
\begin{equation}\label{f:3.3}
I_1 \lesssim \varepsilon^{1-s} \|F\|_{H^{1-s}(\mathbb{R}^d)}
\|\nabla \tilde{w}_{\varepsilon}\|_{L^2(\Omega)}
+ \varepsilon
\|F\|_{L^{2}(\Omega)}
\|\nabla \tilde{w}_{\varepsilon}\|_{L^2(\Omega)}.
\end{equation}

By the upper boundedness of $\widehat{A}$ in Lemma \ref{lemma:2.7},
we have
\begin{equation}\label{f:3.4}
  \begin{aligned}
  |I_{2}|
  &\lesssim \int_{\Omega} |\nabla u_{0}|
  \big|\nabla\big((1-\psi_{\varepsilon}^{\prime})
  \tilde{w}_{\varepsilon}\big)\big|dx
  \lesssim^{\eqref{pri:2.8}}\|\nabla u_{0}\|_{L^{2}(O_{2\varepsilon})}
  \|\nabla\tilde{w}_{\varepsilon}\|_{L^{2}(\Omega)};\\
  |I_{3}|
  &\lesssim \big\|\nabla u_{0}-\varphi\big\|_{L^{2}(\Omega)}
  \|\nabla\tilde{w}_{\varepsilon}\|_{L^{2}(\Omega)};\\
  |I_{4}|
  &\lesssim \varepsilon
  \big\|\varpi(\cdot/\varepsilon)\nabla\varphi
\big\|_{L^2(\Omega)}
\|\nabla\tilde{w}_{\varepsilon}\|_{L^2(\Omega)},
  \end{aligned}
\end{equation}
where we merely use H\"{o}lder's inequality.

Inserting the estimates $\eqref{f:3.3}$, $\eqref{f:3.4}$
into $\eqref{eq:3.2}$,
the desired estimate \eqref{pri:2.5} finally follows from
the elasticity \eqref{eq:1.4} and the estimate $\eqref{pri:2.1}$.
We have completed the proof.
\end{proof}

\noindent\textbf{Proof of Theorem $\ref{thm:2.1}$}.
We first address the estimate $\eqref{pri:1.7}$, and
thanks to Lemma $\ref{lemma:3.3}$ it is reduced to handle
the right-hand side of $\eqref{pri:2.5}$. Recalling
the notation $\varpi=E+l^{+}A\tilde{\chi}$ imposed in
Lemma $\ref{lemma:2.1}$, as well as, the cut-off function
$\psi_\varepsilon$ satisfying $\eqref{eq:2.1}$,  we have
\begin{equation}\label{f:3.1}
\begin{aligned}
\|\varpi(\cdot/\varepsilon)\nabla
S_{\varepsilon}(\psi_{\varepsilon}\nabla u_{0})
\|_{L^{2}(\Omega)}
&\lesssim^{\eqref{pri:2.3}}
\|\nabla(\psi_\varepsilon\nabla u_0)\|_{L^2(\mathbb{R}^d)}\\
&\lesssim
\|\nabla^{2} u_{0}\|_{L^{2}
(\Omega\backslash O_{3\varepsilon})}
+\varepsilon^{-1}\|\nabla u_{0}\|_{L^{2}(O_{4\varepsilon})},
\end{aligned}
\end{equation}
and
\begin{equation}\label{f:3.2}
  \begin{aligned}
 \|\nabla u_{0}-\varphi\|_{L^{2}(\Omega)}
 & \lesssim
 \|\psi_{\varepsilon}\nabla u_{0}-S_{\varepsilon}(\psi_{\varepsilon}\nabla u_{0})\|_{L_{2}(\Omega)}+\|(1-\psi_{\varepsilon})\nabla u_{0}\|_{L_{2}(\Omega)}\\
&\lesssim^{\eqref{pri:2.4}} \varepsilon
\|\nabla(\psi_{\varepsilon}\nabla u_{0})\|_{L_{2}(\Omega)}
+\|(1-\psi_{\varepsilon})\nabla u_{0}\|_{L_{2}(\Omega)}\\
&\lesssim \varepsilon\|\nabla^{2}
u_{0}\|_{L^{2}(\Omega\backslash O_{3\varepsilon})}+
\|\nabla  u_{0}\|_{L^{2}(O_{4\varepsilon})}.
  \end{aligned}
\end{equation}

Thus, plugging the estimates $\eqref{f:3.1}$ and $\eqref{f:3.2}$
back into $\eqref{pri:2.5}$ leads to
\begin{equation}\label{pri:3.1}
\|\nabla w_\varepsilon\|_{L^2(\Omega_{\varepsilon})}
\lesssim
\varepsilon^{1-s}\|F\|_{H^{1-s}(\mathbb{R}^d)}
+\Big\{\|\nabla u_0\|_{L^2(O_{4\varepsilon})}
+ \varepsilon\|\nabla^2 u_0\|_{L^2(
\Omega\backslash O_{3\varepsilon})}
+\varepsilon\|F\|_{L^2(\Omega)}\Big\},
\end{equation}
and this coupled with estimates $\eqref{pri:9.2}$,
$\eqref{pri:9.1}$ gives
the desired estimate
$\eqref{pri:1.7}$.

Then we proceed
to show the estimate $\eqref{pri:1.4}$. To do so, setting
$\tilde{w}_\varepsilon = P_\varepsilon(w_\varepsilon)$,
it follows from the estimate $\eqref{pri:2.1}$ that
\begin{equation*}
\begin{aligned}
\|\nabla w_\varepsilon\|_{L^2(\Omega_\varepsilon)}
\gtrsim\|\nabla \tilde{w}_\varepsilon\|_{L^2(\Omega_0)}
&\geq \|\nabla \tilde{w}_\varepsilon\|_{L^{\frac{2d}{d+1}}(\Omega_0)}\\
&\gtrsim\|\tilde{w}_\varepsilon\|_{L^{\frac{2d}{d-1}}(\Omega_0)}
\geq \|\tilde{w}_\varepsilon\|_{L^{\frac{2d}{d-1}}(\Omega_\varepsilon)}
=\|w_\varepsilon\|_{L^{\frac{2d}{d-1}}(\Omega_\varepsilon)},
\end{aligned}
\end{equation*}
where the second line is due to Sobolev's inequality. This implies
\begin{equation*}
\begin{aligned}
\|u_\varepsilon - u_0\|_{L^{\frac{2d}{d-1}}(\Omega_\varepsilon)}
&\leq \|\nabla w_\varepsilon\|_{L^2(\Omega_\varepsilon)}
+ \varepsilon\|\chi_{\varepsilon}
S_{\varepsilon}(\psi_{\varepsilon}\nabla u_{0})
\|_{L^{\frac{2d}{d-1}}(\Omega_\varepsilon)}\\
&\lesssim^{\eqref{pri:3.1},\eqref{pri:2.3}}
\Big\{\|\nabla u_{0}\|_{L^2(O_{4\varepsilon})}
+\varepsilon\|\nabla^{2} u_{0}\|_{L^2(\Omega\backslash O_{3\varepsilon})}
+\varepsilon\|\nabla u_0\|_{L^{\frac{2d}{d-1}}
(\Omega)}\Big\}\\
&\lesssim \varepsilon^{\frac{1}{2}}\|g\|_{H^1(\partial\Omega)},
\end{aligned}
\end{equation*}
where the last inequality follows from
Theorem \ref{app:thm:2} and this completes the whole proof.
\qed

\begin{corollary}
Assume the same conditions as in Theorem $\ref{thm:2.1}$.
Let $\tilde{w}_\varepsilon$ be the extension of $w_\varepsilon$
in the way of Lemma $\ref{extensiontheory}$, and
$\|F\|_{H^{1-s}(\mathbb{R}^d)} + \|g\|_{H^1(\partial\Omega)}
= 1$ with $0\leq s\leq 1$.
Then there holds
\begin{equation}\label{pri:3.2}
 \|\tilde{w}_\varepsilon\|_{H^{1}(\mathbb{R}^d)}
 \lesssim \max\{\varepsilon^{1-s},\varepsilon^{1/2}\},
\end{equation}
where the up to constant is independent of $s$ and $\varepsilon$.
\end{corollary}

\begin{proof}
The estimate $\eqref{pri:3.2}$ directly follows from the extension
result $\eqref{pri:2.1}$ and the estimate $\eqref{pri:1.7}$,
and we complete the proof.
\end{proof}


\section{Convergence Rates in $H^{-s}$-norm}\label{sec:4}
In this section, we manage to calculate some sharp error
estimates by a duality argument, and then appealing to
interpolation techniques one may accelerate the convergence
rates derived from
energy estimates.
This method is insensitive to the smoothness assumption
on domains.
The adjoint operator of $\mathcal{L}_{\varepsilon}$
is written by $\mathcal{L^{*}_{\varepsilon}}
:=-\nabla\cdot A^{*}(\cdot/\varepsilon)\nabla$,
and there holds $A^*=A$
according to \eqref{eq:1.4}. In order to show the duality
argument (independent of the symmetry condition), we still keep the notation $\mathcal{L}^*_\varepsilon$ in the proof.
For any $\Phi\in H^{\sigma}(\mathbb{R}^d;\mathbb{R}^d)$ with
$(1/2)\leq \sigma\leq 1$,
we have the adjoint problem:
\begin{equation}\label{adjoint}
\left\{\begin{aligned}
  \mathcal{L^{*}_{\varepsilon}}(\phi_{\varepsilon})&=\Phi &\qquad& \text{in~}\Omega_{\varepsilon},\\
  \sigma^*_\varepsilon(\phi_{\varepsilon})&
  =0&\qquad&\text{on~}S_{\varepsilon},\\
  \phi_{\varepsilon}&=0 &\qquad&\text{on~} \Gamma_{\varepsilon},
  \end{aligned}\right.
\end{equation}
where $\sigma^*_\varepsilon:=n\cdot
A^*(\cdot/\varepsilon)\nabla$ is
the related conormal derivative operator.
The corresponding homogenized equation is given by
\begin{equation}\label{eq:4.2}
 \left\{\begin{aligned}
\mathcal{L}^*_{0} \phi_0 \equiv
-\nabla\cdot \widehat{A^{*}}\nabla \phi_0 &= \Phi &\qquad&\text{in}~~\Omega, \\
 \phi_0 &= 0 &\qquad& \text{on}~\partial\Omega,
\end{aligned}\right.
\end{equation}
where the matrix $\widehat{A^{*}}$ is defined by
\begin{equation*}
\widehat{a^{*}}_{ij}^{\alpha\beta} = \dashint_{Y\cap\omega} a_{ik}^{*\alpha\gamma}(y)\frac{\partial \mathbb{X}_{j}^{*\gamma\beta}}{\partial y_{k}} dy,
\end{equation*}
in which $\mathbb{X^*}^{\beta}_{j}=\{\mathbb{X^*}^{\gamma\beta}_{j}\}_{1\leq\gamma\leq d}$ is the weak solution to the following cell problem
\begin{equation}\label{pde:1.2*}
\left\{\begin{aligned}
& \nabla \cdot A^*\nabla\mathbb{X}^{*\beta}_{j} = 0 \qquad\text{in}~ Y\cap \omega,\\
& \vec{n}\cdot A^*\nabla\mathbb{X}^{*\beta}_{j}=0 \qquad\text{on}~Y\cap\partial\omega,\\
&\mathbb{X}_{j}^{*\beta}-y_{j}e^{\beta}:=\chi_{j}^{*\beta}\in
H^1_{\text{per}}(\omega;\mathbb{R}^d), \quad& \dashint_{Y\cap \omega}&\chi_{j}^{*\beta} dy = 0.
\end{aligned}\right.
\end{equation}

Let $z_{\varepsilon}=\phi_{\varepsilon}-\phi_{0}-\varepsilon
\chi^{*}(x/\varepsilon)\varphi^*$
be the related first order approximating corrector with $\varphi^*
=S_{\varepsilon}(\psi_{\varepsilon}\nabla \phi_{0})$,
and it follows from the estimate $\eqref{pri:1.7}$ that
\begin{equation}\label{sharp0}
\|\nabla z_{\varepsilon}\|_{L^{2}(\Omega_{\varepsilon})}
\lesssim \varepsilon^{1/2}\|\Phi\|_{H^{\sigma}(\mathbb{R}^d)}.
\end{equation}

Later on, we will employ the equality
$\phi_\varepsilon=z_{\varepsilon}
+\phi_{0}+\varepsilon
\chi^{*}(x/\varepsilon)\varphi^*$ to show
\begin{equation}\label{sharp9}
\begin{aligned}
  \int_{\Omega_{\varepsilon}}w_{\varepsilon}\Phi dx
 &=\int_{\Omega_{\varepsilon}} A_{ij}(x/\varepsilon)\nabla_{x_{j}}
 w_{\varepsilon}\nabla_{x_{i}}\phi_{\varepsilon}dx\\
 &=
 \underbrace{\int_{\Omega_{\varepsilon}}A(x/\varepsilon)
 \nabla w_{\varepsilon}
 \nabla[\phi_{0}+
 \varepsilon\chi^*_{\varepsilon}\varphi^{*}]dx}_{R_1}
 +\underbrace{\int_{\Omega_{\varepsilon}}A(x/\varepsilon)
 \nabla w_{\varepsilon}\nabla z_{\varepsilon}dx}_{R_2},
  \end{aligned}
\end{equation}
where $w_{\varepsilon}=u_{\varepsilon}
 -u_{0}-\varepsilon\chi_{\varepsilon}\varphi$
 with $\varphi=S_{\varepsilon}(\psi_{\varepsilon}\nabla u_{0})$
 is defined in $\eqref{eq:3.1}$.
 This together with $\eqref{f:4.1}$ leads to
 \begin{equation}\label{eqkey}
 \begin{aligned}
 \int_{\mathbb{R}^d} \theta\tilde{w}_\varepsilon\Phi dx
&= \int_{\Omega_{\varepsilon}}A(x/\varepsilon)
 \nabla w_{\varepsilon}
 \nabla[\phi_{0}+
 \varepsilon\chi^*_{\varepsilon}\varphi^{*}]dx
 +\int_{\Omega_{\varepsilon}}A(x/\varepsilon)
 \nabla w_{\varepsilon}\nabla z_{\varepsilon}dx\\
&+ \int_{\Omega} (\theta-l_\varepsilon^+)
 \tilde{w}_\varepsilon\Phi dx
 + \int_{\Omega_0\setminus\Omega}
 \theta\tilde{w}_\varepsilon\Phi dx\\
&:= R_1 + R_2 + R_3 + R_4.
 \end{aligned}
 \end{equation}

 Obviously, obtaining the quantity
 $\|\tilde{w}_\varepsilon\|_{H^{-\sigma}(\mathbb{R}^d)}$
 may be reduced
 to estimate the right-hand side of $\eqref{eqkey}$,
 and we
 state the main result of this section as follows.
\begin{theorem}[duality argument I]\label{thm:3.1}
Let $\Omega\subset\mathbb{R}^d$ be a Lipschitz domain.
Suppose that $\mathcal{L}_{\varepsilon}$
and $\omega$ satisfy the hypothesises
\emph{(H1)} and \emph{(H2)}. The
given data $F\in H^{1}(\Omega;\mathbb{R}^d)$
and $g\in H^1(\partial\Omega;\mathbb{R}^d)$ are assumed to
meet the condition
$$
\Big(
\int_{\Omega}|\nabla F|^2\delta^2 dx\Big)^{1/2}
+\|F\|_{L^{2}(\Omega)}
+ \|g\|_{H^{1}(\partial\Omega)} = 1.$$
Let $w_\varepsilon$ be given in $\eqref{eq:3.1}$, and
$\tilde{w}_\varepsilon$ be the extension
of $w_\varepsilon$ in the way of Lemma
$\ref{extensiontheory}$.
Then, for any
$\Phi\in H^{\sigma}(\mathbb{R}^d;\mathbb{R}^d)$ with
$(1/2)\leq \sigma\leq 1$ and
$\|\Phi\|_{H^{\sigma}(\mathbb{R}^d)} = 1$, there hold
the weak formulation $\eqref{eqkey}$
and
the following estimate
\begin{equation}\label{pri:4.1}
  \bigg|\int_{\mathbb{R}^d} \theta\tilde{w}_\varepsilon\Phi dx\bigg|
  \lesssim \varepsilon\ln(1/\varepsilon),
\end{equation}
where the up to constant depends on
$\mu_0,\mu_1,d$ and the characters of $\Omega$ and $\omega$.
 \end{theorem}

\begin{remark}
If $\Omega\subset\mathbb{R}^d$ is a $C^{1,1}$ domain, then
there holds the sharp error estimate
$\big|\int_{\mathbb{R}^d} \theta\tilde{w}_\varepsilon\Phi dx\big|
\lesssim \varepsilon$, provided
the boundary data $g\in H^{3/2}(\partial\Omega;\mathbb{R}^d)$.
\end{remark}

\begin{corollary}\label{cor:4.1}
Assume the same conditions as in Theorem $\ref{thm:3.1}$,
and fix $\sigma=1/2$ therein.
Let $w_\varepsilon$ be given in $\eqref{eq:3.1}$, and
$\tilde{w}_\varepsilon$ be the corresponding extension.
Then we have the following results.
\begin{enumerate}
  \item Convergence rates in $H^{-1/2}(\mathbb{R}^d)$, i.e.,
\begin{equation}\label{pri:4.2}
\|\tilde{w}_\varepsilon\|_{H^{-1/2}(\mathbb{R}^d)}
\lesssim \varepsilon\ln(1/\varepsilon).
\end{equation}
  \item Convergence rates in $L^{2}(\mathbb{R}^d)$, i.e.,
\begin{equation}\label{pri:4.3}
\|\tilde{w}_\varepsilon\|_{L^{2}(\mathbb{R}^d)}
\lesssim \varepsilon^{\frac{5}{6}}\ln^{\frac{2}{3}}(1/\varepsilon).
\end{equation}
\end{enumerate}
Here the up to constant is independent of $\varepsilon$.
\end{corollary}

\begin{proof}
The estimate $\eqref{pri:4.2}$ follows from
$\eqref{pri:4.1}$ immediately by noting the arbitrariness of $\Phi$,
and the up to constant additionally depends
on the constant $\theta$.
To show the estimate $\eqref{pri:4.3}$, we seek for
an interpolation's inequality argument
(see \cite[Proposition 1.52]{BHCD}). Thus,
there holds
\begin{equation*}
\|\tilde{w}_\varepsilon\|_{L^{2}(\mathbb{R}^d)}
\leq \|\tilde{w}_\varepsilon\|_{H^{-1/2}(\mathbb{R}^d)}^{\frac{2}{3}}
\|\tilde{w}_\varepsilon\|_{H^{1}(\mathbb{R}^d)}^{\frac{1}{3}}
\lesssim \varepsilon^{\frac{5}{6}}\ln^{\frac{2}{3}}(1/\varepsilon),
\end{equation*}
where we employ the estimates $\eqref{pri:3.2}$
and $\eqref{pri:4.2}$ in the last inequality,
and this ends the proof.
\end{proof}


To handle $R_1$ in the right-hand side of $\eqref{eqkey}$,
we have the following result.

\begin{lemma}\label{lemma:3.2}
Given $F, \Phi\in L^2(\Omega;\mathbb{R}^d)$ and
$g\in H^1(\partial\Omega;\mathbb{R}^d)$,
let $w_\varepsilon$ be given in $\eqref{eq:3.1}$,
and the weak solution $\phi_0$ be associated
with $\Phi$ by $\eqref{eq:4.2}$.
Assume $\chi^{*}_{\varepsilon}$ is the corrector satisfying
the equation $\eqref{pde:1.2*}$ and
$\varphi^{*}=S_{\varepsilon}
(\psi_{\varepsilon}^\prime\nabla \phi_{0})$,
where $\psi_{\varepsilon}^\prime$ is cut-off function
given in $\eqref{eq:2.1}$.
Then
one may have
\begin{equation}\label{pri:4.5}
\begin{aligned}
\Big|\int_{\Omega_{\varepsilon}}
A(x/\varepsilon)
&\nabla w_{\varepsilon}
\nabla \big(\phi_{0}
+ \varepsilon\chi^{*}_{\varepsilon}\varphi^{*}\big)dx
\Big|\\
&\lesssim \varepsilon \ln(1/\varepsilon)\|\Phi\|_{L^{2}(\Omega)}
\bigg\{
\Big(\int_{\Omega}|\nabla F|^2\delta^2 dx
+\|F\|_{L^{2}(\Omega)}
+\|g\|_{H^{1}(\partial\Omega)}
\bigg\},
\end{aligned}
\end{equation}
where the up to constant is independent of $\varepsilon$.
\end{lemma}

\begin{proof}
The main idea of the proof is inspired by \cite{Q1}, while
we provide a proof for the sake of the completeness.
Observing that $\phi:=\phi_{0}
+ \varepsilon\chi^{*}_{\varepsilon}\varphi^{*}
\in H_0^1(\Omega;\mathbb{R}^d)$, one may treat it
as a whole in computations at first,
and then handle each term of $\phi$. So, the proof
is divided into two parts.

\textbf{Part 1.} The main job is to establish the following
estimate
\begin{equation}\label{f:4.5}
\begin{aligned}
&\int_{\Omega_{\varepsilon}}A(x/\varepsilon)
\nabla w_{\varepsilon}
\nabla \phi dx
\lesssim \varepsilon^{1/2}\big\|\mathrm{M}_{\text{r}}(\nabla u_0)
\big\|_{L^2(\partial\Omega)}
\|\nabla \phi\|_{L^2(O_{4\varepsilon})} \\
& + \varepsilon\Bigg\{\Big(\int_{\Omega}|\nabla F|^2\delta^2 dx
\Big)^{\frac{1}{2}}
+\Big(\int_{\Omega}|F|^2dx\Big)^{\frac{1}{2}}\Bigg\}
\Big(\int_{\Omega}|\nabla \phi|^2 dx\Big)^{\frac{1}{2}}\\
& + \varepsilon\Bigg\{
\|\mathrm{M}_{\text{r}}(\nabla u_0)\|_{L^2(\partial\Omega)}
+\Big(\int_{\Omega\setminus O_{3\varepsilon}}
|\nabla^2 u_0|^2\delta dx\Big)^{\frac{1}{2}}\Bigg\}
\Big(\int_{\Omega\setminus O_{3\varepsilon}}
|\nabla\phi|^2\delta^{-1} dx\Big)^{\frac{1}{2}},
\end{aligned}
\end{equation}
where the notation $\mathrm{M}_{\text{r}}$ is referred to
as the radial maximal operator, defined in
$\eqref{def:4}$. First of all,
in view of Lemma $\ref{lemma:2.1}$ and
the same trick employed to deal with $I_1$ in
$\eqref{eq:3.2}$, we start from
\begin{equation}\label{f:4.4}
\begin{aligned}
&\qquad\int_{\Omega_{\varepsilon}}A(x/\varepsilon)
\nabla w_{\varepsilon}\nabla \phi dx\\
&=^{\eqref{eq:2.2}}\int_{\Omega}(l_{\varepsilon}^{+}-\theta)F\phi dx
+\int_{\Omega}(\theta\widehat{A}-l_{\varepsilon}^{+}A^{\varepsilon})
(\nabla u_{0}-\varphi)\nabla\phi dx
-\varepsilon\int_{\Omega}\varpi(x/\varepsilon)
\nabla\varphi\nabla\phi dx\\
&=^{\eqref{auxi1}}\varepsilon
\int_{\Omega}\nabla_y\Psi(y)
\cdot \big(\nabla F\phi+\nabla\phi F\big) dx\\
&\qquad
+\int_{\Omega}(\theta\widehat{A}-l_{\varepsilon}^{+}A^{\varepsilon})
(\nabla u_{0}-\varphi)\nabla\phi dx
-\varepsilon\int_{\Omega}\varpi(x/\varepsilon)
\nabla\varphi\nabla\phi dx
:= J_1 + J_2 + J_3,
\end{aligned}
\end{equation}
where $y=x/\varepsilon$,
and we remark that $\|\nabla\Psi\|_{L^\infty(Y)}\lesssim 1$.
Then we have
\begin{equation}\label{f:4.3}
\begin{aligned}
|\varepsilon^{-1}J_1|
&\lesssim \Big(\int_{\Omega}|\nabla F|^2\delta^2 dx\Big)^{1/2}
\Big(\int_{\Omega}|\phi|^2\delta^{-2} dx\Big)^{1/2}
+ \Big(\int_{\Omega}|F|^2dx\Big)^{1/2}
\Big(\int_{\Omega}|\nabla \phi|^2 dx\Big)^{1/2}\\
&\lesssim \bigg\{\Big(\int_{\Omega}|\nabla F|^2\delta^2 dx
\Big)^{1/2}
+\Big(\int_{\Omega}|F|^2dx\Big)^{1/2}\bigg\}
\Big(\int_{\Omega}|\nabla \phi|^2 dx\Big)^{1/2},
\end{aligned}
\end{equation}
where the second line follows from Hardy's inequality
(see for example \cite[Proposition III.2.40]{BF}).

In order to handle the terms $J_2$ and $J_3$,
we appeal to the
radial maximal function (see $\eqref{def:4}$), again.
Now,
we proceed to study the term $J_2$, and
\begin{equation*}
\begin{aligned}
|J_2|
&\lesssim \int_{O_{4\varepsilon}}
|\nabla u_\varepsilon||\nabla \phi| dx
+ \int_{\Omega\setminus O_{3\varepsilon}}
|\psi_\varepsilon\nabla u_0
-S_\varepsilon(\psi_\varepsilon\nabla u_0)||\nabla\phi|dx\\
&\lesssim \|\nabla u_0\|_{L^2(O_{4\varepsilon})}
\|\nabla \phi\|_{L^2(O_{4\varepsilon})}
+\Big(\int_{\Omega\setminus O_{3\varepsilon}}
|\psi_\varepsilon\nabla u_0
-S_\varepsilon(\psi_\varepsilon\nabla u_0)|^2\delta dx
\Big)^{\frac{1}{2}}
\Big(\int_{\Omega\setminus O_{3\varepsilon}}
|\nabla\phi|^2\delta^{-1} dx\Big)^{\frac{1}{2}}\\
&\lesssim^{\eqref{pri:2.7}}
 \|\nabla u_0\|_{L^2(O_{4\varepsilon})}
\|\nabla \phi\|_{L^2(O_{4\varepsilon})}
+ \varepsilon \Big(\int_{\Omega\setminus O_{2\varepsilon}}
|\nabla (\psi_\varepsilon\nabla u_0)|^2\delta dx
\Big)^{\frac{1}{2}}
\Big(\int_{\Omega\setminus O_{3\varepsilon}}
|\nabla\phi|^2\delta^{-1} dx\Big)^{\frac{1}{2}}.
\end{aligned}
\end{equation*}
By the co-area formula \cite[Theorem 3.13]{LCE1} coupled
with the definition of the radial maximal
operator, the right-hand side above is controlled by
\begin{equation}\label{f:4.2}
\begin{aligned}
&\varepsilon^{1/2}\|\mathrm{M}_{\text{r}}(\nabla u_0)
\|_{L^2(\partial\Omega)}
\|\nabla \phi\|_{L^2(O_{4\varepsilon})}\\
&\qquad\qquad + \varepsilon\bigg\{
\|\mathrm{M}_{\text{r}}(\nabla u_0)\|_{L^2(\partial\Omega)}
+\Big(\int_{\Omega\setminus O_{3\varepsilon}}
|\nabla^2 u_0|^2\delta dx\Big)^{1/2}\bigg\}
\Big(\int_{\Omega\setminus O_{3\varepsilon}}
|\nabla\phi|^2\delta^{-1} dx\Big)^{\frac{1}{2}}
\end{aligned}
\end{equation}
(up to an universal constant). Finally, since
$\text{supp}(\varphi) \subseteq \Omega\setminus O_{2\varepsilon}$
according to $\eqref{eq:2.1}$,
we acquire
\begin{equation*}
\begin{aligned}
|J_3|
&\leq \varepsilon
\Big(\int_{\Omega\setminus O_{2\varepsilon}}|\varpi(x/\varepsilon)
\nabla S_\varepsilon(\psi_\varepsilon\nabla u_0)|^2\delta dx\Big)^{1/2}
\Big(\int_{\Omega\setminus O_{2\varepsilon}}
|\nabla\varphi|^2\delta^{-1}dx\Big)^{1/2}\\
&\lesssim^{\eqref{pri:2.11}}
\varepsilon
\Big(\int_{\Omega\setminus O_{2\varepsilon}}
|\nabla(\psi_\varepsilon\nabla u_0)|^2\delta dx\Big)^{1/2}
\Big(\int_{\Omega\setminus O_{2\varepsilon}}
|\nabla\varphi|^2\delta^{-1}dx\Big)^{1/2}.
\end{aligned}
\end{equation*}
By the same token, the right-hand side of
the above estimate is dominated by the second line
of $\eqref{f:4.2}$. Thus, one may conclude that
\begin{equation*}
\begin{aligned}
|J_2| &+ |J_3| \\
&\lesssim \varepsilon^{1/2}\|\mathrm{M}_{\text{r}}(\nabla u_0)
\|_{L^2(\partial\Omega)}
\|\nabla \phi\|_{L^2(O_{4\varepsilon})}\\
&\qquad\qquad + \varepsilon\bigg\{
\|\mathrm{M}_{\text{r}}(\nabla u_0)\|_{L^2(\partial\Omega)}
+\Big(\int_{\Omega\setminus O_{3\varepsilon}}
|\nabla^2 u_0|^2\delta dx\Big)^{1/2}\bigg\}
\Big(\int_{\Omega\setminus O_{3\varepsilon}}
|\nabla\phi|^2\delta^{-1} dx\Big)^{\frac{1}{2}},
\end{aligned}
\end{equation*}
and this together with $\eqref{f:4.3}$ and $\eqref{f:4.4}$
yields the desired estimate $\eqref{f:4.5}$.

\textbf{Part 2.}
Recalling $\phi = \phi_{0}
+ \varepsilon\chi^{*}_{\varepsilon}
\varphi^{*}$,
and $\varphi^{*}
= S_{\varepsilon}
(\psi_{\varepsilon}^\prime\nabla \phi_{0})$
with $\psi_{\varepsilon}^\prime$ being cut-off function
given in $\eqref{eq:2.1}$,
we need to compute the following quantities:
\begin{equation}\label{f:4.6}
\|\nabla \phi\|_{L^2(O_{4\varepsilon})};
\quad \Big(\int_{\Omega}|\nabla \phi|^2 dx\Big)^{1/2};
\quad
\Big(\int_{\Omega\setminus O_{3\varepsilon}}
|\nabla\phi|^2\delta^{-1} dx\Big)^{1/2}.
\end{equation}
Since $\text{supp}(\varphi^{*})\cap O_{4\varepsilon} = \emptyset$,
there simply holds
\begin{equation}\label{f:4.7}
\|\nabla \phi\|_{L^2(O_{4\varepsilon})}
\lesssim \varepsilon^{1/2}
\|\mathrm{M}_{\text{r}}(\nabla\phi_0)\|_{L^2(\partial\Omega)}.
\end{equation}

Then we are interested in the third term of $\eqref{f:4.6}$,
and its leading term is
$\int_{\Omega\setminus O_{3\varepsilon}}
|\nabla(\varepsilon\chi_\varepsilon^{*}
\varphi^*)|^2\delta^{-1}dx$.
So we calculate it
as follows:
\begin{equation*}
\begin{aligned}
&\quad\int_{\Omega\setminus O_{3\varepsilon}}
|\nabla (\varepsilon\chi_\varepsilon^*
S_\varepsilon(\psi_\varepsilon^\prime\nabla \phi_0))|^2
\delta^{-1}dx\\
&\lesssim
\int_{\Omega\setminus O_{3\varepsilon}}
|\nabla\chi^*(y)S_\varepsilon
(\psi_\varepsilon^\prime\nabla\phi_0)|^2\delta^{-1}dx
+
\varepsilon^2\int_{\Omega\setminus O_{3\varepsilon}}
\big|\chi^*(y)S_\varepsilon
\big(\nabla(\psi_\varepsilon^\prime\nabla\phi_0)\big)
\big|^2\delta^{-1}dx\\
&\lesssim^{\eqref{pri:2.11}}
\int_{\Omega\setminus O_{7\varepsilon}}
|\nabla\phi_0|^2\delta^{-1}dx
+\varepsilon^2
\int_{\Omega\setminus O_{3\varepsilon}}
\big|\nabla(\psi_\varepsilon^\prime\nabla\phi_0)
\big|^2\delta^{-1}dx\\
&\lesssim
\int_{\Omega\setminus O_{7\varepsilon}}
|\nabla\phi_0|^2\delta^{-1}dx
+ \int_{O_{8\varepsilon}\setminus O_{3\varepsilon}}
|\nabla\phi_0|^2\delta^{-1}dx
+\varepsilon^2
\int_{\Omega\setminus O_{3\varepsilon}}
|\nabla^2\phi_0
|^2\delta^{-1}dx.
\end{aligned}
\end{equation*}
Using the co-area formula again, the above estimate
implies that
\begin{equation}\label{f:4.8}
\begin{aligned}
&\qquad\Big(\int_{\Omega\setminus O_{3\varepsilon}}
|\nabla \phi|^2 \delta^{-1}dx\Big)^{\frac{1}{2}}\\
&\lesssim
\Big(\int_{\Omega\setminus O_{3\varepsilon}}
|\nabla\phi_0|^2\delta^{-1}dx\Big)^{\frac{1}{2}}
+ \Big(\int_{O_{8\varepsilon}\setminus O_{3\varepsilon}}
|\nabla\phi_0|^2\delta^{-1}dx\Big)^{\frac{1}{2}}
+\varepsilon
\Big(\int_{\Omega\setminus O_{3\varepsilon}}
|\nabla^2\phi_0
|^2\delta^{-1}dx\Big)^{\frac{1}{2}}\\
&\lesssim
\|\nabla\phi_0\|_{L^2(\Omega\setminus O_{c_0})}
+\ln^{\frac{1}{2}}(c_0/\varepsilon)
\|\mathrm{M}_{\text{r}}(\nabla\phi_0)\|_{L^2(\partial\Omega)}
+\varepsilon^{\frac{1}{2}}\|
\nabla^2\phi_0\|_{L^2(\Omega\setminus O_{3\varepsilon})},
\end{aligned}
\end{equation}
where we use the fact that $\varepsilon\leq \delta(x)\leq r_0$
for $x\in\Omega\setminus O_{3\varepsilon}$ in the last inequality.
Then, the same arguments that we used for $\eqref{f:4.8}$ leads to
\begin{equation}\label{f:4.9}
\begin{aligned}
\Big(\int_{\Omega}|\nabla \phi|^2 dx\Big)^{1/2}
\lesssim^{\eqref{pri:2.3}}
\|\nabla\phi_0\|_{L^2(\Omega)}
+\varepsilon^{\frac{1}{2}}
\|\mathrm{M}_{\text{r}}(\nabla\phi_0)\|_{L^2(\partial\Omega)}
+ \varepsilon
\|\nabla^2\phi_0\|_{L^2(\Omega\setminus O_{7\varepsilon})}.
\end{aligned}
\end{equation}
Thus, in view of the estimates
$\eqref{f:4.7}, \eqref{f:4.8}, \eqref{f:4.9}$
and $\eqref{f:4.6}$ we consequently obtain
\begin{equation}\label{f:4.10}
\begin{aligned}
\max\bigg\{
&\|\nabla \phi\|_{L^2(O_{4\varepsilon})},
~\Big(\int_{\Omega}|\nabla \phi|^2 dx\Big)^{1/2},
~\Big(\int_{\Omega\setminus O_{3\varepsilon}}
|\nabla\phi|^2\delta^{-1} dx\Big)^{1/2}\bigg\}\\
&\lesssim
\|\nabla\phi_0\|_{L^2(\Omega)}
+\ln^{\frac{1}{2}}(1/\varepsilon)
\|\mathrm{M}_{\text{r}}(\nabla\phi_0)\|_{L^2(\partial\Omega)}
+\varepsilon^{\frac{1}{2}}\|
\nabla^2\phi_0\|_{L^2(\Omega\setminus O_{3\varepsilon})}\\
&\lesssim^{\eqref{pri:9.3},\eqref{pri:9.4},\eqref{pri:9.1}}
\ln^{\frac{1}{2}}(1/\varepsilon)\|\Phi\|_{L^2(\mathbb{R}^d)}.
\end{aligned}
\end{equation}

To complete the whole arguments, we still need
\begin{equation}\label{f:4.11}
\begin{aligned}
\quad\max\bigg\{
\|\mathrm{M}_{\text{r}}(\nabla u_0)\|_{L^2(\partial\Omega)},
~&\Big(\int_{\Omega\setminus O_{3\varepsilon}}
|\nabla^2 u_0|^2\delta dx\Big)^{1/2}\bigg\}\\
&\lesssim^{\eqref{pri:9.4},\eqref{pri:9.1}}
\ln^{\frac{1}{2}}(1/\varepsilon)
\bigg\{\|F\|_{L^2(\mathbb{R}^d)}
+\|g\|_{H^{1}(\partial\Omega)}\bigg\}.
\end{aligned}
\end{equation}

Consequently, plugging the estimates
$\eqref{f:4.10}$ and $\eqref{f:4.11}$ back into
$\eqref{f:4.5}$, we have proved the stated estimate
$\eqref{pri:4.5}$, and therefore completed the whole proof.
\end{proof}

\noindent \textbf{The proof of Theorem \ref{thm:3.1}}.
To show the estimate $\eqref{pri:4.1}$, it suffices to
study the right-hand side of the identity $\eqref{eqkey}$
term by term. The tricky one is $R_1$ and
it has already been studied in Lemma $\ref{lemma:3.2}$.
From the estimate $\eqref{pri:4.5}$, we have
\begin{equation}\label{f:4.13}
  |R_1|
 \lesssim \varepsilon\ln(1/\varepsilon).
\end{equation}
Obviously, the easy one in $\eqref{eqkey}$ is the term $R_2$, and
it follows from $\eqref{pri:1.7}$, $\eqref{sharp0}$ and
H\"older's inequality that
\begin{equation}\label{f:4.14}
|R_2|=\bigg|\int_{\Omega_{\varepsilon}}A(x/\varepsilon)
 \nabla w_{\varepsilon}\nabla z_{\varepsilon}dx\bigg|
 \lesssim \varepsilon
 \|\Phi\|_{H^{1/2}(\mathbb{R}^d)}\Big\{
 \|F\|_{H^{1/2}(\mathbb{R}^d)}+
 \|g\|_{H^{1}(\partial\Omega)}\Big\}
 \lesssim \varepsilon.
\end{equation}

We now address the estimates for  $R_3$ and $R_4$.
In terms of $R_3$, we first observe that
\begin{equation*}
\begin{aligned}
\int_{\Omega} (\theta-l_\varepsilon^+)
 \tilde{w}_\varepsilon\Phi dx
 = \int_{\Omega} \psi_\varepsilon^\prime
 (\theta-l_\varepsilon^+)
 \tilde{w}_\varepsilon\Phi dx
 + \int_{\Omega} (1-\psi_\varepsilon^\prime)
 (\theta-l_\varepsilon^+)
 \tilde{w}_\varepsilon\Phi dx
 := R_{31} + R_{32},
\end{aligned}
\end{equation*}
where the cut-off function $\psi_\varepsilon^\prime$ satisfies
$\eqref{eq:2.1}$.
Thanks to the auxiliary equation $\eqref{auxi1}$,
the arguments used for $I_{11}$ in the proof
of Lemma $\ref{lemma:3.3}$ lead to
\begin{equation*}
\begin{aligned}
R_{31}
&\lesssim
\varepsilon^{\sigma}
\Big\{\|\Phi\|_{H^{\sigma}(\mathbb{R}^d)}
+ \varepsilon^{1-\sigma}
\|\Phi\|_{L^{2}(\Omega)}\Big\}
\|\nabla \tilde{w}_{\varepsilon}\|_{L^2(\Omega)}\\
&\lesssim^{\eqref{pri:2.1},\eqref{pri:1.7}}
\varepsilon^{\sigma+\frac{1}{2}}
\|\Phi\|_{H^{1/2}(\mathbb{R}^d)}\Big\{
 \|F\|_{H^{1/2}(\mathbb{R}^d)}+
 \|g\|_{H^{1}(\partial\Omega)}\Big\}
 \lesssim \varepsilon^{\frac{1}{2}+\sigma},
\end{aligned}
\end{equation*}
while one may have
 \begin{equation*}
   R_{32}
   \lesssim
   \int_{O_{4\varepsilon}}
   |\Phi\tilde{w}_\varepsilon| dx
   \leq \|\Phi\|_{L^2(\Omega)}
   \|\tilde{w}_\varepsilon\|_{L^2(O_{4\varepsilon})}
   \lesssim^{\eqref{pri:2.8}}
   \varepsilon
   \|\Phi\|_{L^2(\Omega)}
   \|\nabla\tilde{w}_\varepsilon\|_{L^2(\Omega)}
   \lesssim^{\eqref{pri:2.1},\eqref{pri:1.7}} \varepsilon.
 \end{equation*}
 This together with the estimates on $R_{31}$ yields
 \begin{equation}\label{f:4.12}
 |R_3| = \Big|\int_{\Omega} (\theta-l_\varepsilon^+)
 \tilde{w}_\varepsilon\Phi dx\Big|
 \lesssim \varepsilon^{\frac{1}{2}+\sigma}.
 \end{equation}

 Since $\Omega_0\setminus\Omega\subseteq \tilde{O}_{20\varepsilon}
 =:\{x\in\Omega_0:\text{dist}(x,\partial\Omega_0)
 \leq 20\varepsilon\}$, we similarly arrive at
 \begin{equation}\label{f:4.15}
  |R_4|
 \leq
  \theta\int_{\tilde{O}_{20\varepsilon}}
 |\tilde{w}_\varepsilon\Phi| dx
   \lesssim
   \|\tilde{w}_\varepsilon\|_{L^2(\tilde{O}_{20\varepsilon})}
   \|\Phi\|_{L^2(\mathbb{R}^d)}
   \lesssim^{\eqref{pri:2.8}} \varepsilon
   \|\nabla \tilde{w}_\varepsilon\|_{L^2(\Omega_0)}
   \lesssim^{\eqref{pri:2.1},
   \eqref{pri:1.7}} \varepsilon^{\frac{3}{2}}.
 \end{equation}

Hence, plugging the estimates $\eqref{f:4.13}$,
$\eqref{f:4.14}$, $\eqref{f:4.12}$ and
$\eqref{f:4.15}$ back into $\eqref{eqkey}$, we have
the desired estimate $\eqref{pri:4.1}$.
We have completed the whole proof.
\qed


\section{Estimates for Weak Formulation from Duality}\label{sec:6}

\begin{theorem}[duality argument II]\label{thm:8.1}
Let $\Omega$ be a bounded
$C^{1,\eta}$ domain with $\eta\in(0,1]$. Suppose that
$\mathcal{L}_\varepsilon$ and $\omega$ satisfy
the hypothesises \emph{(H1)} and \emph{(H2)}. Given
$F\in H^{1}(\Omega;\mathbb{R}^d)$ and $g\in
H^{1}(\partial\Omega;\mathbb{R}^d)$, let
$u_\varepsilon$ and $u_0$ be the weak solutions to
the equations $\eqref{pde:1.1}$ and $\eqref{pde:1.3}$, respectively.
Let $w_\varepsilon$ be given in $\eqref{eq:3.1}$,
and $\rho = \delta^{1-\tau}$ with $0<\tau<1$.
Then, for any $f\in L^2
(\Omega;\mathbb{R}^{d\times d})$,
we obtain
\begin{equation}\label{pri:8.4}
\begin{aligned}
&\quad\Big|\int_{\Omega_\varepsilon}
 \nabla w_\varepsilon\cdot f dx\Big| \\
& \lesssim \varepsilon^{1-\frac{\tau}{2}}
\ln^{\frac{1}{2}}(1/\varepsilon)
 \Bigg\{
 \Big(\int_{\Omega}|\nabla F|^2
\delta^{3-\tau}dx\Big)^{\frac{1}{2}}
 +\|F\|_{H^{1/2}(\Omega)}
 +\|g\|_{H^1(\partial\Omega)}\Bigg\}
 \Big(\int_{\Omega_\varepsilon}|f|^2\rho^{-1} dx\Big)^{\frac{1}{2}},
\end{aligned}
\end{equation}
where the up to constant is independent of $\varepsilon$.
\end{theorem}

\begin{corollary}[square function estimates]\label{cor:8.1}
Assume the same conditions as in Theorem $\ref{thm:8.1}$.
Given $F\in H^{1}(\Omega_0;\mathbb{R}^d)$ and $g\in
H^{1}(\partial\Omega;\mathbb{R}^d)$,
let $w_\varepsilon$ be given in $\eqref{eq:3.1}$.
Then there holds the square function estimate
\begin{equation}\label{pri:8.2}
\bigg(\int_{\Omega_\varepsilon}
|\nabla w_\varepsilon|^2 \delta dx\bigg)^{\frac{1}{2}}
\lesssim \varepsilon^{1-\frac{\tau}{2}}
\ln^{\frac{1}{2}}(1/\varepsilon)
 \Bigg\{
 \Big(\int_{\Omega_0}|\nabla F|^2
\delta dx\Big)^{\frac{1}{2}}
 +\|F\|_{L^{2}(\Omega_0)}
 +\|g\|_{H^1(\partial\Omega)}\Bigg\},
\end{equation}
where the up to constant is independent of $\varepsilon$.
\end{corollary}

\begin{proof}
Recall that $\rho = \delta^{1-\tau}$ with
$\delta(x)=\text{dist}(x,\partial\Omega_0)$ and $0<\tau<1$.
On the one hand,
it follows from the estimate $\eqref{pri:8.4}$ that
\begin{equation}\label{f:8.2}
\begin{aligned}
\Big(\int_{\Omega_\varepsilon}|\nabla w_\varepsilon|^2\rho
dx\Big)^{\frac{1}{2}}
&\lesssim \varepsilon^{1-\frac{\tau}{2}}
\ln^{\frac{1}{2}}(1/\varepsilon)
 \Bigg\{
 \Big(\int_{\Omega}|\nabla F|^2
\delta^{3-\tau}dx\Big)^{\frac{1}{2}}
 +\|F\|_{H^{1/2}(\Omega)}
 +\|g\|_{H^1(\partial\Omega)}\Bigg\}\\
&\lesssim
\varepsilon^{1-\frac{\tau}{2}}
\ln^{\frac{1}{2}}(1/\varepsilon)
 \Bigg\{
 \Big(\int_{\Omega_0}|\nabla F|^2
\delta dx\Big)^{\frac{1}{2}}
 +\|F\|_{L^{2}(\Omega_0)}
 +\|g\|_{H^1(\partial\Omega)}\Bigg\},
\end{aligned}
\end{equation}
where we employ \cite[Lemma 8.11.3]{S} in the second step. On the
other hand, we observe that
\begin{equation*}
\Big(\int_{\Omega_\varepsilon}|\nabla w_\varepsilon|^2\rho
dx\Big)^{\frac{1}{2}}
\gtrsim \Big(\int_{\Omega_\varepsilon}
|\nabla w_\varepsilon|^2\delta dx\Big)^{\frac{1}{2}},
\end{equation*}
since  $\delta\leq r_0$ on $\Omega_\varepsilon$
with $r_0 = \text{diam}(\Omega)$.
This together with $\eqref{f:8.2}$ gives the stated
estimate $\eqref{pri:8.2}$ and we have completed the proof.
\end{proof}

\noindent \textbf{The proof of Theorem \ref{thm:8.1}}.
The main idea is the duality argument.
For any $f\in L^2(\Omega;\mathbb{R}^{d\times d})$,
we construct the adjoint equation:
\begin{equation}\label{adjoint2}
\left\{\begin{aligned}
  \mathcal{L^{*}_{\varepsilon}}(\phi_{\varepsilon})
  &=\nabla\cdot f &\qquad& \text{in~}\Omega_{\varepsilon},\\
  \sigma_\varepsilon^*(\phi_{\varepsilon})&
  =-n\cdot f&\qquad&\text{on~}S_{\varepsilon},\\
  \phi_{\varepsilon}&=0 &\qquad&\text{on~} \Gamma_{\varepsilon}.
  \end{aligned}\right.
\end{equation}
Integration by parts, we have
\begin{equation}
\begin{aligned}
  \int_{\Omega_{\varepsilon}}\nabla w_{\varepsilon}
  \cdot f dx
  &=-\int_{\Omega_{\varepsilon}}
  w_{\varepsilon}\mathcal{L}_\varepsilon^*(\phi_\varepsilon)  dx
  - \int_{S_\varepsilon}n\cdot f w_\varepsilon dS\\
 &=\int_{\Omega_{\varepsilon}} A_{ij}(x/\varepsilon)
 \nabla_{x_{j}}w_{\varepsilon}\nabla_{x_{i}}\phi_{\varepsilon}dx\\
 &= \underbrace{\int_{\Omega_{\varepsilon}} A(x/\varepsilon)
 \nabla w_{\varepsilon}\nabla(\phi_\varepsilon-\phi) dx}_{K_1}
 +
\underbrace{ \int_{\Omega_{\varepsilon}} A(x/\varepsilon)
 \nabla w_{\varepsilon}\nabla\phi dx}_{K_2},
  \end{aligned}
\end{equation}
where $\phi = \psi_{\varepsilon}^{\prime}\tilde{\phi}_\varepsilon$
with $\tilde{\phi}_\varepsilon:
=\Lambda_\varepsilon(\phi_\varepsilon)$ being
the related extension function
in the way of Theorem $\ref{thm:1.5}$.
Owning to the equality $\eqref{eq:2.2}$ and
the auxiliary equation $\eqref{auxi1}$, one may  have
\begin{equation}\label{f:8.6}
\begin{aligned}
K_2
&=\varepsilon\int_{\Omega}\nabla_y\Psi(y)\cdot\nabla F
\phi dx
+\varepsilon\int_{\Omega}\varpi_{1}(y)
\nabla\varphi\nabla\phi dx\\
&+\int_{\Omega}
\varpi_{2}(y)(\nabla u_0 - \varphi)\nabla\phi dx
+\varepsilon\int_{\Omega}
\nabla_y\Psi(y) F \nabla\phi dx
=:K_{21} + K_{22} + K_{23} + K_{24},
\end{aligned}
\end{equation}
where $y=x/\varepsilon$, and $\varpi_1 := -\varpi$ (see
Lemma $\ref{lemma:2.1}$) with
$\varpi_2:=\theta\widehat{A}-l_{\varepsilon}^{+}A^{\varepsilon}$.

We proceed to handle the term $K_1$, by definition of $\phi$ we have
\begin{equation*}
\begin{aligned}
K_1
&\lesssim
\frac{1}{\varepsilon}\int_{O_\varepsilon\cap\Omega_\varepsilon}
|\nabla w_\varepsilon||\phi_\varepsilon|
+\int_{O_{2\varepsilon}\cap\Omega_\varepsilon}
|\nabla w_\varepsilon||\nabla\phi_\varepsilon|\\
&\lesssim \|\nabla w_\varepsilon\|_{L^2(\Omega_\varepsilon)}
\|\nabla\phi_{\varepsilon}\|_{L^2(O_{2\varepsilon}\cap\Omega_\varepsilon)}
\lesssim
\|\nabla w_\varepsilon\|_{L^2(\Omega_\varepsilon)}
\|\rho\|_{L^\infty(O_{4\varepsilon})}^{1/2}
\Big(\int_{O_{2\varepsilon}\cap\Omega_\varepsilon}
|\nabla\phi_\varepsilon|^2\rho^{-1}dx\Big)^{1/2}\\
&\lesssim^{\eqref{pri:8.3}}
\|\nabla w_\varepsilon\|_{L^2(\Omega_\varepsilon)}
\|\rho\|_{L^\infty(O_{4\varepsilon})}^{1/2}
\Big(\int_{\Omega}\dashint_{B_\varepsilon(x)
\cap\Omega_\varepsilon}|\nabla\phi_\varepsilon|^2dy\rho^{-1} dx
\Big)^{1/2},
\end{aligned}
\end{equation*}
where we recall $\rho(x) = [\delta(x)]^{1-\tau}$
with $\delta(x)=\text{dist}(x,\partial\Omega_0)$.
It is not hard to see that $\rho(x)\sim
\varepsilon^{1-\tau}$ whenever
$x\in O_{2\varepsilon}$.
Then it follows from the weighted
quenched Calder\'on-Zygmund estimate
$\eqref{pri:1.6}$ that
\begin{equation*}
\begin{aligned}
\Big(\int_{\Omega}\dashint_{B_\varepsilon(x)
\cap\Omega_\varepsilon}|\nabla\phi_\varepsilon|^2
dy\rho^{-1} dx
\Big)^{1/2}
&\lesssim
\Big(\int_{\Omega_\varepsilon}
|f|^2\rho^{-1} dx
\Big)^{1/2}.
\end{aligned}
\end{equation*}
Together with this, we have
\begin{equation}\label{f:8.11}
\begin{aligned}
K_1 &\lesssim
\|\nabla w_\varepsilon\|_{L^2(\Omega_\varepsilon)}
\|\rho\|_{L^\infty(O_{4\varepsilon})}^{1/2}
\Big(\int_{\Omega_\varepsilon}
|f|^2\rho^{-1} dx
\Big)^{1/2}\\
&\lesssim^{\eqref{pri:1.7}}
\varepsilon^{1-\frac{\tau}{2}}
\Big\{\|F\|_{H^{1/2}(\Omega)}
 +\|g\|_{H^1(\partial\Omega)}\Big\}
\Big(\int_{\Omega_\varepsilon}
|f|^2\rho^{-1} dx
\Big)^{1/2}.
\end{aligned}
\end{equation}

Then on account of $\eqref{f:8.6}$
we turn to address the expression $K_2$ term by term.
The first one is $K_{21}$, and
\begin{equation}\label{f:8.7}
\begin{aligned}
K_{21}
&\leq \varepsilon\Big(\int_{\Omega}|\nabla F|^2
\delta^{3-\tau}dx\Big)^{\frac{1}{2}}
\Big(\int_{\Omega_0}|\tilde{\phi}_\varepsilon|^2
\delta^{\tau-3}dx\Big)^{\frac{1}{2}}\\
&\lesssim \varepsilon\Big(\int_{\Omega}|\nabla F|^2
\delta^{3-\tau}dx\Big)^{\frac{1}{2}}
\Big(\int_{\Omega_0}|\nabla\tilde{\phi}_\varepsilon|^2
\delta^{\tau-1} dx\Big)^{\frac{1}{2}}\\
&\lesssim^{\eqref{pri:1.9}} \varepsilon\Big(\int_{\Omega}|\nabla F|^2
\delta^{3-\tau}dx\Big)^{\frac{1}{2}}
\Big(\int_{\Omega_\varepsilon}
|\nabla\tilde{\phi}_\varepsilon|^2\rho^{-1} dx\Big)^{\frac{1}{2}}\\
&\lesssim^{\eqref{f:8.12}} \varepsilon
\Big(\int_{\Omega}|\nabla F|^2
\delta^{3-\tau}dx\Big)^{\frac{1}{2}}
\Big(\int_{\Omega_\varepsilon}|f|^2\rho^{-1}dx\Big)^{\frac{1}{2}},
\end{aligned}
\end{equation}
where we use weighted Hardy's inequality
(see for example \cite[Theorem 1.1]{Le} or \cite{N}) in the second step.
Since the support of $\nabla\varphi$ is included in the set
$\Omega\setminus O_{4\varepsilon}$, there holds
\begin{equation*}\label{}
\begin{aligned}
K_{22}
&\leq
\varepsilon\Big(\int_{\Omega\setminus O_{2\varepsilon}}
\big|\varpi_1(x/\varepsilon)
\nabla\varphi\big|^2 \delta(x)dx\Big)^{\frac{1}{2}}
\Big(\int_{\Omega\setminus O_{2\varepsilon}}
|\nabla\tilde{\phi}_\varepsilon|^2[\delta(x)]^{-1} dx
\Big)^{\frac{1}{2}}\\
&\leq
\varepsilon^{1-\frac{\tau}{2}}
\Big(\int_{\Omega\setminus O_{2\varepsilon}}
\big|
\nabla(\psi_{\varepsilon}\nabla u_0)\big|^2
\delta(x)dx\Big)^{\frac{1}{2}}
\Big(\int_{\Omega_0}
|\nabla\tilde{\phi}_\varepsilon|^2\rho^{-1} dx
\Big)^{\frac{1}{2}}\\
&\lesssim \varepsilon^{-\frac{\tau}{2}}
\bigg\{
\Big(\int_{O_{4\varepsilon}}
\big|\nabla u_0\big|^2 \delta dx\Big)^{\frac{1}{2}}
+\varepsilon
\Big(\int_{\Omega\setminus O_{2\varepsilon}}
\big|\nabla^2 u_0\big|^2 \delta dx\Big)^{\frac{1}{2}}
\bigg\}
\Big(\int_{\Omega_\varepsilon}
|\nabla\phi_\varepsilon|^2\rho^{-1} dx
\Big)^{\frac{1}{2}},
\end{aligned}
\end{equation*}
where we employ Lemma $\ref{lemma:2.10}$ in the second inequality,
and the last one follows from the weighted extension result
$\eqref{pri:1.9}$.
Note that
\begin{equation}\label{f:8.12}
\Big(\int_{\Omega_\varepsilon}
|\nabla\phi_\varepsilon|^2\rho^{-1} dx
\Big)^{\frac{1}{2}}
\lesssim^{\eqref{pri:8.3}} \bigg(\int_{\Omega}
\dashint_{B_\varepsilon(x)\cap\Omega_\varepsilon}
|\nabla\phi_\varepsilon|^2dy\rho^{-1}dx
\bigg)^{\frac{1}{2}}
\lesssim^{\eqref{pri:1.6}}
\Big(\int_{\Omega_\varepsilon}|f|^2\rho^{-1}dx\Big)^{\frac{1}{2}},
\end{equation}
and this together with the estimates
$\eqref{pri:9.2}$, $\eqref{pri:9.1}$
and the previous computations leads to
\begin{equation}\label{f:8.8}
K_{22}
\lesssim \varepsilon^{1-\frac{\tau}{2}}
\ln^{\frac{1}{2}}(1/\varepsilon)
\Big\{\|F\|_{L^{2}(\Omega)}
 +\|g\|_{H^1(\partial\Omega)}\Big\}
\Big(\int_{\Omega_\varepsilon}|f|^2\rho^{-1}dx
\Big)^{\frac{1}{2}}.
\end{equation}
By the same token, we have
\begin{equation*}\label{}
\begin{aligned}
K_{23}
&\lesssim
\int_{\Omega}
|(\nabla u_0 - \varphi)\nabla\phi| dx\\
&\lesssim
\int_{O_{4\varepsilon}\setminus O_\varepsilon}
|\nabla u_0||\nabla\tilde{\phi}_\varepsilon|dx
+\frac{1}{\varepsilon}\int_{O_{2\varepsilon}\setminus O_\varepsilon}
|\nabla u_0||\tilde{\phi}_\varepsilon|dx
+\int_{\Omega\setminus O_{2\varepsilon}}
|\psi_\varepsilon \nabla u_0-
S_\varepsilon(\psi_\varepsilon\nabla u_0)|
|\nabla\tilde{\phi}_\varepsilon|dx\\
&\lesssim^{\eqref{pri:2.7}}
\varepsilon^{-\frac{\tau}{2}}
\bigg\{
\Big(\int_{O_{4\varepsilon}}
|\nabla u_0|^2\delta dx\Big)^{\frac{1}{2}}
+\varepsilon\Big(\int_{\Omega}
|\nabla(\psi_\varepsilon\nabla u_0)|^2\delta dx
\Big)^{\frac{1}{2}}
\bigg\}
\Big(\int_{\Omega}
|\nabla\tilde{\phi}_\varepsilon|\rho^{-1}dx\Big)^{\frac{1}{2}}\\
&\lesssim^{\eqref{pri:1.9}}
\varepsilon^{-\frac{\tau}{2}}
\bigg\{
\Big(\int_{O_{4\varepsilon}}
|\nabla u_0|^2\delta dx\Big)^{\frac{1}{2}}
+\varepsilon\Big(\int_{\Omega}
|\nabla^2 u_0|^2\delta dx
\Big)^{\frac{1}{2}}
\bigg\}
\Big(\int_{\Omega_\varepsilon}
|\nabla\phi_\varepsilon|\rho^{-1}dx\Big)^{\frac{1}{2}},
\end{aligned}
\end{equation*}
in which we also employ the following computation
\begin{equation*}
\begin{aligned}
\frac{1}{\varepsilon}\int_{O_{2\varepsilon}\setminus O_\varepsilon}
|\nabla u_0||\tilde{\phi}_\varepsilon|dx
&\leq
\varepsilon^{-\frac{3}{2}}
\Big(\int_{O_{2\varepsilon}}|\nabla u_0|^2\delta dx
\Big)^{\frac{1}{2}}
\Big(\int_{O_{2\varepsilon}}|\tilde{\phi}_\varepsilon|^2dx
\Big)^{\frac{1}{2}}\\
&\lesssim^{\eqref{pri:2.8}}
\varepsilon^{-\frac{1}{2}}
\Big(\int_{O_{2\varepsilon}}|\nabla u_0|^2\delta dx
\Big)^{\frac{1}{2}}
\Big(\int_{O_{2\varepsilon}}|\nabla\tilde{\phi}_\varepsilon|^2dx
\Big)^{\frac{1}{2}}\\
&\lesssim
\varepsilon^{-\frac{\tau}{2}}
\Big(\int_{O_{2\varepsilon}}|\nabla u_0|^2\delta dx
\Big)^{\frac{1}{2}}
\Big(\int_{O_{2\varepsilon}}|\nabla\tilde{\phi}_\varepsilon|^2
\rho^{-1}dx
\Big)^{\frac{1}{2}}
\end{aligned}
\end{equation*}
in the third step.
Moreover, it follows from the estimates $\eqref{pri:9.2}$, $\eqref{pri:9.1}$
and $\eqref{pri:8.3}$ that
\begin{equation}\label{f:8.9}
\begin{aligned}
K_{23}
&\lesssim
\varepsilon^{1-\frac{\tau}{2}}
\ln^{\frac{1}{2}}(1/\varepsilon)
\Big\{\|F\|_{L^{2}(\Omega)}
 +\|g\|_{H^1(\partial\Omega)}\Big\}
\int_{\Omega}
\dashint_{B_\varepsilon(x)\cap\Omega_\varepsilon}
|\nabla\phi_\varepsilon|^2
dy\rho^{-1}dx\\
&\lesssim^{\eqref{pri:1.6}}
\varepsilon^{1-\frac{\tau}{2}}\ln^{\frac{1}{2}}(1/\varepsilon)
\Big\{\|F\|_{L^{2}(\Omega)}
 +\|g\|_{H^1(\partial\Omega)}\Big\}
\Big(\int_{\Omega_\varepsilon}|f|^2\rho^{-1}dx
\Big)^{\frac{1}{2}}.
\end{aligned}
\end{equation}
Now, we turn to study the term $K_{24}$, and
\begin{equation}\label{f:8.10}
\begin{aligned}
K_{24}
&\lesssim \varepsilon\int_{\Omega}|F|
|\nabla\tilde{\phi}_\varepsilon| dx
+ \int_{O_{2\varepsilon}}|F||\tilde{\phi}_\varepsilon|dx\\
&\lesssim^{\eqref{pri:2.8}}  \varepsilon\|F\|_{L^2(\Omega)}
\|\nabla\tilde{\phi}_\varepsilon\|_{L^2(\Omega)}
+ \varepsilon\|F\|_{L^2(\Omega)}
\|\nabla\tilde{\phi}_\varepsilon\|_{L^2(O_{2\varepsilon})}\\
&\lesssim
\varepsilon\|F\|_{L^2(\Omega)}
\|\nabla\tilde{\phi}_\varepsilon\|_{L^2(\Omega)}
\lesssim \varepsilon
\|F\|_{L^2(\Omega)}
\Big(\int_{\Omega}|\nabla\tilde{\phi}_\varepsilon|^2
\rho^{-1}dx
\Big)^{\frac{1}{2}}\\
&\lesssim^{\eqref{pri:1.9}} \varepsilon
\|F\|_{L^2(\Omega)}
\Big(\int_{\Omega_\varepsilon}|\nabla\phi_\varepsilon|^2
\rho^{-1}dx
\Big)^{\frac{1}{2}}\\
&\lesssim^{\eqref{f:8.12}}
\varepsilon\|F\|_{L^2(\Omega)}
\Big(\int_{\Omega_\varepsilon}|f|^2\rho^{-1}dx
\Big)^{\frac{1}{2}}.
\end{aligned}
\end{equation}

Consequently, combining the estimates
$\eqref{f:8.7}$, $\eqref{f:8.8}$, $\eqref{f:8.9}$
and $\eqref{f:8.10}$ leads to
\begin{equation*}
 K_2
 \lesssim
 \varepsilon^{1-\frac{\tau}{2}}
 \ln^{\frac{1}{2}}(1/\varepsilon)
\bigg\{
\Big(\int_{\Omega}|\nabla F|^2
\delta^{3-\tau}dx\Big)^{\frac{1}{2}}
+\|F\|_{L^{2}(\Omega)}
 +\|g\|_{H^1(\partial\Omega)}\bigg\}
\Big(\int_{\Omega_\varepsilon}|f|^2\rho^{-1}dx
\Big)^{\frac{1}{2}}.
\end{equation*}
This together with the estimate $\eqref{f:8.11}$ shows
the desired estimate $\eqref{pri:8.4}$ and we have
completed the whole proof.
\qed

\begin{remark}\label{remark:8.1}
Set $\rho(x) = [\delta(x)]^{1-\tau}$
with $\delta(x)=\text{dist}(x,\partial\Omega_0)$
and $0<\tau<1$.
Assume that $\phi_\varepsilon$ is the solution to
$\eqref{adjoint2}$ and let $\tilde{\phi}_\varepsilon$ be
the related extension function of $\phi_\varepsilon$ according
to Theorem $\ref{thm:1.5}$. Then for any fixed
$U\subseteq\Omega$ we have
\begin{equation}\label{pri:8.3}
\begin{aligned}
\int_{U\cap\Omega_\varepsilon}|\nabla\phi_\varepsilon|^2
\rho^{-1}dx
&=\int_{U}l_\varepsilon^{+}|\nabla\tilde{\phi}_\varepsilon|^2
\rho^{-1}dx\\
&\lesssim^{\eqref{pri:5.4}} \int_{\Omega}
\dashint_{B_\varepsilon(x)}
l_\varepsilon^{+}|\nabla\tilde{\phi}_\varepsilon|^2
\rho^{-1}dydx
\lesssim \int_{\Omega}
\dashint_{B_\varepsilon(x)\cap\Omega_\varepsilon}
|\nabla\phi_\varepsilon|^2
dy\rho^{-1}dx,
\end{aligned}
\end{equation}
where the last inequality follows from the fact that
$\rho(x)\sim\rho(y)$ since $|x-y|<\varepsilon$.
\end{remark}

\noindent \textbf{The proof of Theorem \ref{thm:1.1}.}
We first handle the estimate $\eqref{pri:1.1}$, which
is in fact based upon the duality arguments sated in
Theorem $\ref{thm:8.1}$.
On the one hand,  recalling $\rho=\delta^{1-\tau}$ with
$\delta(x)=\text{dist}(x,\partial\Omega_0)$,
for any fixed $\tau\in(0,1)$, one may derive that
\begin{equation}\label{f:8.14}
\begin{aligned}
\int_{\Omega_\varepsilon}
|\nabla w_\varepsilon|^2 \delta^{1-\tau} dx
&\gtrsim^{\eqref{pri:1.9}}
\int_{\Omega_0}|\nabla
\Lambda_\varepsilon(w_\varepsilon)|^2
\delta^{1-\tau} dx\\
&\gtrsim
\Big(\int_{\Omega_0}
|\Lambda_\varepsilon(w_\varepsilon)
|^{\frac{2d}{d-1-\tau}} dx\Big)^{\frac{d-1-\tau}{d}}
\geq
\Big(\int_{\Omega_\varepsilon}
|w_\varepsilon|^{\frac{2d}{d-1-\tau}} dx\Big)^{\frac{d-1-\tau}{d}},
\end{aligned}
\end{equation}
in which we apply the
weighted Hardy-Sobolev inequality \cite[Theorem 2.1]{LV}
to the last step for the extended function
$\Lambda_\varepsilon(w_\varepsilon)$
vanishes near $\partial\Omega_0$. Also,
the condition $\tau\not=0$ is very important, which
avoided the critical case of the weighted Hardy inequality
(see \cite[Theorem 1.1]{Le}).
On the other hand,
it follows from the estimate
$\eqref{f:8.2}$ that
\begin{equation*}
\begin{aligned}
\Big(\int_{\Omega_\varepsilon}|\nabla w_\varepsilon|^2\rho
dx\Big)^{\frac{1}{2}}
\lesssim \varepsilon^{1-\frac{\tau}{2}}\ln^{\frac{1}{2}}(1/\varepsilon)
 \Bigg\{
 \Big(\int_{\Omega_0}|\nabla F|^2
\delta dx\Big)^{\frac{1}{2}}
 +\|F\|_{L^{2}(\Omega_0)}
 +\|g\|_{H^1(\partial\Omega)}\Bigg\},
\end{aligned}
\end{equation*}
and this together with $\eqref{f:8.14}$ leads to
\begin{equation*}
\|w_\varepsilon\|_{L^{\frac{2d}{d-1-\tau}}(\Omega_\varepsilon)}
\lesssim \varepsilon^{1-\frac{\tau}{2}}
\ln^{\frac{1}{2}}(1/\varepsilon)
 \Bigg\{
 \Big(\int_{\Omega_0}|\nabla F|^2
\delta dx\Big)^{\frac{1}{2}}
 +\|F\|_{L^{2}(\Omega_0)}
 +\|g\|_{H^1(\partial\Omega)}\Bigg\}.
\end{equation*}
Thus, setting $p=\frac{2d}{d-1-\tau}$ and
$q=\frac{2(d-1)}{d-1-\tau}$,
the above estimate together with
\begin{equation*}
\|\varepsilon\chi_{\varepsilon}S_\varepsilon(\psi_\varepsilon
\nabla u_0)\|_{L^{p}(\Omega)}
\lesssim^{\eqref{pri:2.3}} \varepsilon\|\nabla u_0\|_{L^{p}(\Omega)}
\lesssim^{\eqref{pri:9.5}} \varepsilon
\Big\{\|F\|_{L^{2}(\Omega)}
  +\|g\|_{W^{1,q}(\partial\Omega)}\Big\}
\end{equation*}
and a triangle inequality leads to the
almost-sharp error estimate $\eqref{pri:1.1}$.

By virtue of $\eqref{pri:4.3}$,
we have
the desired estimate $\eqref{eq:1.8}$,
where we employed the estimate $\eqref{pri:9.3}$
instead of $\eqref{pri:9.5}$ in the computations.
This completes the whole proof.
\qed


\section{Weighted Quenched Calder\'on-Zygmund Estimates}
\label{sec:5}

\begin{lemma}[Shen's lemma]\label{lemma:5.4}
Let $q>2$ and $\Omega$ be a bounded Lipschitz domain.
Let $F\in L^2(\Omega)$ and $f\in L^p(\Omega)$
for some $2<p<q$. Suppose that for each ball $B$
with the property that $|B|\leq  c_0|\Omega|$
and either $4B\subset\Omega$ or $B$ is centered on
$\partial\Omega$, there exist two measurable functions
$F_B$ and $R_B$ on $\Omega\cap2B$,
such that $|F|\leq |F_B| + |R_B|$ on $\Omega\cap2B$,
\begin{equation}\label{pri:5.6}
\begin{aligned}
\Big(\dashint_{2B\cap\Omega}|R_B|^q\Big)^{\frac{1}{q}}
&\leq N_1 \bigg\{
\Big(\dashint_{4B\cap\Omega}|F|^2\Big)^{\frac{1}{2}}
+\sup_{4B_0\supseteq B^\prime\supseteq B}
\Big(\dashint_{B^\prime\cap\Omega}|f|^2\Big)^{\frac{1}{2}}
\bigg\}\\
\Big(\dashint_{2B\cap\Omega}|F_B|^2\Big)^{\frac{1}{2}}
&\leq N_2
\sup_{4B_0\supseteq B^\prime\supseteq B}
\Big(\dashint_{B^\prime\cap\Omega}|f|^2\Big)^{\frac{1}{2}}
\end{aligned}
\end{equation}
where $N_1, N_2>0$
and $0<c_0<1$. Then $F\in L^p(\Omega)$
and
\begin{equation}\label{}
\Big(\int_{\Omega}|F|^p\Big)^{\frac{1}{p}}
\leq C\bigg\{\Big(\int_{\Omega}|F|^2\Big)^{\frac{1}{2}}
+\Big(\int_{\Omega}|f|^p\Big)^{\frac{1}{p}}\bigg\},
\end{equation}
where $C$ depends at most on $N_1$,
 $N_2$, $c_0$, $p$, $q$ and the Lipschitz character of $\Omega$.
\end{lemma}

\begin{proof}
See \cite[Theorem 4.2.6]{S0} or \cite[Theorem 4.13]{S3}.
\end{proof}

\begin{remark}
Recently, we noticed that Z. Shen \cite{S4} extended
real arguments to the weighted Sobolev spaces as
we were preparing this project.
So, it is very likely to follow his new scheme to have a
proof for Theorem $\ref{thm:1.4}$ concerned with a Lipschitz domain.
In fact, the idea on the proof of Theorem $\ref{thm:1.4}$
was inspired by he and his cooperator's work \cite{SZ1}.
Moreover, weighted quenched Calder\'on-Zygmund estimates
were noticed by the second author because of F. Otto's personal
interests.
\end{remark}

\begin{lemma}[primary geometry on integrals]
\label{lemma:5.5}
Let $f\in L^1_{\emph{loc}}(\mathbb{R}^d)$, and
$\Omega\subset\mathbb{R}^d$ be a bounded domain. Then
there hold the following inequalities:
\begin{itemize}
  \item If $0<r<(\varepsilon/4)$ and
  $D_r(x_0)$ is given, then for any $x\in D_r(x_0)$ we have
  \begin{equation}\label{pri:5.3}
  \dashint_{B_\varepsilon(x)\cap\Omega}|f|
  \lesssim
  \dashint_{D_{4r}(x_0)}
  \dashint_{B_\varepsilon(x)\cap\Omega}|f| dx.
  \end{equation}
  \item If $r\geq(\varepsilon/4)$ and $D_r(x_0)$ is given,
  then there holds
  \begin{equation}\label{pri:5.4}
  \dashint_{D_r(x_0)} |f|
  \lesssim \dashint_{D_{2r}(x_0)}
  \dashint_{B_\varepsilon(x)\cap\Omega}|f|
  \lesssim \dashint_{D_{6r}(x_0)} |f|.
  \end{equation}
  \item If $\rho\in A_1$ and $0<r<r_0/10$,
  then one may derive that
  \begin{equation}\label{pri:5.9}
  \int_{\tilde{\Omega}}
  \dashint_{B_r(x)\cap\tilde{\Omega}}|f(y)|dy \rho(x) dx
  \lesssim \int_{\Omega}
  |f|\rho dx,
  \qquad \tilde{\Omega}:=\{x\in\Omega:\emph{dist}(x,\partial\Omega)>r\},
  \end{equation}
\end{itemize}
where the up to constant depends only on $d$.
\end{lemma}

\begin{proof}
The proof is standard and we provide a proof for the sake of
the completeness, while in the case of $\Omega=\mathbb{R}^d$ we
refer the reader to \cite[Lemma 6.5]{DO} for a detail.
We first show the estimate $\eqref{pri:5.3}$.
For any set $K\subset\mathbb{R}^d$,
let $I_K$ be the indicator function of $K$.
There holds
\begin{equation*}
\begin{aligned}
\dashint_{D_{4r}(x_0)}
  \dashint_{B(x,\varepsilon)\cap\Omega}|f| dx
&=\frac{1}{|D_{4r}||B_\varepsilon|}\int_{\mathbb{R}^d}|f(y)|
\int_{\mathbb{R}^d} I_{D_{4r}(x_0)}(x)
I_{\{(x,y)\in\Omega\times\Omega:|x-y|<\varepsilon\}}(x,y)dx dy\\
&\gtrsim \dashint_{D_{2\varepsilon}(x_0)}|f(y)|dy
\geq \dashint_{B_{\varepsilon}(x)\cap\Omega}|f(y)|dy
\end{aligned}
\end{equation*}
for any $x\in D_r(x_0)$,
where we note that $B_\varepsilon(x)\cap\Omega\subset
D_{2\varepsilon}(x_0)$, provided $x\in D_r(x_0)$.
Then we turn to the inequality $\eqref{pri:5.4}$.
On the one hand,
\begin{equation}\label{}
\begin{aligned}
\dashint_{D_{2r}(x_0)}
  \dashint_{B_\varepsilon(x)\cap\Omega}|f|
&= \frac{1}{|D_{2r}||B_\varepsilon|}
\int_{\mathbb{R}^d}|f(y)|
\int_{\mathbb{R}^d} I_{D_{2r}(x_0)}(x)
I_{\{(x,y)\in\Omega\times\Omega:|x-y|<\varepsilon\}}(x,y)dx dy\\
&\gtrsim \dashint_{D_{r}(x_0)}|f(y)|dy.
\end{aligned}
\end{equation}
On the other hand,
\begin{equation}\label{}
\begin{aligned}
\dashint_{D_{2r}(x_0)}
  \dashint_{B_\varepsilon(x)\cap\Omega}|f|
&= \frac{1}{|D_{2r}||B_\varepsilon|}
\int_{\mathbb{R}^d}\underbrace{|f(y)|
\int_{\mathbb{R}^d} I_{D_{2r}(x_0)}(x)
I_{\{(x,y)\in\Omega\times\Omega:|x-y|<\varepsilon\}}(x,y)dx}_{
:=\tilde{f}(y)} dy\\
&\lesssim \dashint_{D_{6r}(x_0)}|f(y)|dy,
\end{aligned}
\end{equation}
since the support of $\tilde{f}$ is included in $D_{6r}(x_0)$.
Then we proceed to show the estimate $\eqref{pri:5.9}$, and
\begin{equation*}
\begin{aligned}
\int_{\tilde{\Omega}}
  \dashint_{B_r(x)\cap\tilde{\Omega}}|f(y)|dy \rho(x) dx
&\leq
\int_{\Omega}|f(y)|
  \dashint_{B_r(y)} \rho(x)dx dy\\
&\leq \int_{\Omega}|f(y)|
  \mathcal{M}(\rho)(y) dy
\lesssim \int_{\Omega}|f(y)|\rho(y) dy,
\end{aligned}
\end{equation*}
where we use the Fubini's theorem in the first inequality
and the definition of $A_1$
(see for example \cite[pp.134]{D}) in the last one.
We have completed the proof.
\end{proof}

\noindent \textbf{The proof of Theorem \ref{thm:1.2}}.
The main idea is based upon Prof. Felix Otto's unpublished
lectures given in CIMI (Toulouse).
We first consider $p\geq 2$, while the case $1<p<2$ would be done
by a duality argument. The core ingredient of the proof
is Shen's real methods \cite{S3}.

Let $B := B_{r}(x_0)$ be any ball with $0<r<r_0/10$, such that
$x_0\in\partial\Omega$ or $4B\subset\Omega$.
To
achieve the target, we impose the following quantities:
\begin{equation*}
\begin{aligned}
&U(x):=
\Big(\dashint_{B_\varepsilon(x)\cap\Omega_\varepsilon}
|\nabla u_\varepsilon|^2\Big)^{\frac{1}{2}},
\quad F(x):=
\Big(\dashint_{B_\varepsilon(x)\cap\Omega_\varepsilon}
|f|^2\Big)^{\frac{1}{2}}, \\
&W_B(x): = \Big(\dashint_{B_\varepsilon(x)\cap\Omega_\varepsilon}
|\nabla w_\varepsilon|^2\Big)^{\frac{1}{2}},
\quad V_B(x): = \Big(\dashint_{B_\varepsilon(x)\cap\Omega_\varepsilon}
|\nabla v_\varepsilon|^2\Big)^{\frac{1}{2}}
\end{aligned}
\end{equation*}
for any $x\in\Omega$, and later on one may prefer
$W_B$ and $V_B$ according to $B$, such that
$|U|\leq |W_B|+|V_B|$. For the ease of the statement,
$\tilde{u}_\varepsilon$, $\tilde{v}_\varepsilon$ and
$\tilde{w}_\varepsilon$ are
corresponding extension function
(in the way of Lemma $\ref{extensiontheory}$).

Case 1. If $0<r<\varepsilon/4$, then it is fine to fix
$W_B = U$ and $V_B = 0$. For any $x\in B\cap\Omega$,
one may show that
\begin{equation}\label{f:5.5}
\begin{aligned}
W_B^2(x) = \dashint_{B_\varepsilon(x)\cap\Omega_\varepsilon}
|\nabla u_\varepsilon|^2
&= \dashint_{B_\varepsilon(x)\cap\Omega}
l_\varepsilon^{+}|\nabla \tilde{u}_\varepsilon|^2\\
&
\lesssim \dashint_{4B\cap\Omega}
\dashint_{B_\varepsilon(x)\cap\Omega}l_\varepsilon^{+}
|\nabla \tilde{u}_\varepsilon|^2
= \dashint_{4B\cap\Omega} U^2,
\end{aligned}
\end{equation}
where we employ the estimate $\eqref{pri:5.3}$ in the
inequality.
Thus, we have
\begin{equation}\label{f:5.8}
\Big(\dashint_{B\cap\Omega} W_B^p\Big)^{\frac{1}{p}}
\leq \sup_{x\in B\cap\Omega}|W_B(x)|
\lesssim \Big(\dashint_{4B\cap\Omega} U^2\Big)^{\frac{1}{2}}
+ \Big(\dashint_{B\cap\Omega} F^2\Big)^{\frac{1}{2}}
\end{equation}
and trivially,
\begin{equation}\label{f:5.9}
\Big(\dashint_{B\cap\Omega} V_B^p\Big)^{\frac{1}{p}}
\lesssim
\Big(\dashint_{B\cap\Omega} F^2\Big)^{\frac{1}{2}}
\end{equation}
in such the case.

Case 2. For $r\geq (\varepsilon/4)$, let $u_\varepsilon
= v_\varepsilon+w_\varepsilon$, and
$v_\varepsilon$ with $w_\varepsilon$
satisfies the following equations:
\begin{equation}\label{pde:5.2}
(\text{i})~\left\{\begin{aligned}
\mathcal{L}_\varepsilon(v_\varepsilon) &= \nabla\cdot (I_Bf)
&\quad&\text{in}~\Omega_{\varepsilon},\\
\sigma_\varepsilon(v_\varepsilon) &= n\cdot (I_Bf)
&\quad&\text{on}~ S_\varepsilon,\\
v_\varepsilon &= 0
&\quad&\text{on}~ \Gamma_\varepsilon,\\
\end{aligned}\right.
\qquad
(\text{ii})~\left\{\begin{aligned}
\mathcal{L}_\varepsilon(w_\varepsilon) &= 0
&\quad&\text{in}~D_r^\varepsilon(x_0),\\
\sigma_\varepsilon(w_\varepsilon) &= 0
&\quad&\text{on}~ D_r(x_0)\cap S_\varepsilon,\\
w_\varepsilon &= 0
&\quad&\text{on}~ \Delta_r(x_0)\cap\Gamma_\varepsilon,\\
\end{aligned}\right.
\end{equation}
respectively. For any $x\in B\cap\Omega$, in terms of (ii) above,
it follows from
boundary (and interior) Lipschitz estimates $\eqref{pri:5.5}$
(and \cite[Theorem 1.1]{BR}) that
\begin{equation}\label{f:5.16}
\begin{aligned}
\dashint_{B_\varepsilon(x)\cap\Omega_\varepsilon}
|\nabla w_\varepsilon|^2
\lesssim
\dashint_{D_{5r}^\varepsilon(x_0)}
|\nabla w_\varepsilon|^2
&=\dashint_{D_{5r}(x_0)}
l_{\varepsilon}^{+}|\nabla \tilde{w}_\varepsilon|^2\\
&\lesssim
\dashint_{D_{10r}(x_0)}
\dashint_{B_\varepsilon(x)\cap\Omega}
l_{\varepsilon}^{+}|\nabla \tilde{w}_\varepsilon|^2
\lesssim \dashint_{B_{10r}(x_0)\cap\Omega} W_B^2,
\end{aligned}
\end{equation}
where we have use the estimate $\eqref{pri:5.4}$ in the
second inequality, and this implies
\begin{equation}\label{f:5.6}
\begin{aligned}
\sup_{x\in B\cap\Omega}|W_B(x)|^2
\lesssim \dashint_{B_{10r}(x_0)\cap\Omega} \big(
U^2+V_B^2\big).
\end{aligned}
\end{equation}
Then in view of
(i) in $\eqref{pde:5.2}$, we have
\begin{equation*}
\begin{aligned}
\dashint_{D_{10r}(x_0)} V_B^2
\lesssim \dashint_{D_{14r}(x_0)}
l_\varepsilon^{+}|\nabla \tilde{v}_\varepsilon|^2
&=\dashint_{D_{14r}^\varepsilon(x_0)}
|\nabla v_\varepsilon|^2\\
&\lesssim \dashint_{D_{r}^\varepsilon(x_0)}
|f|^2 =  \dashint_{D_{r}(x_0)}l_\varepsilon^{+}
|f|^2 \lesssim
\dashint_{D_{2r}(x_0)} F^2,
\end{aligned}
\end{equation*}
where we employ the estimate $\eqref{pri:5.4}$ in
the first and last inequalities and
energy estimates in the second one. This together with
$\eqref{f:5.6}$ implies that
\begin{equation}\label{f:5.7}
\begin{aligned}
&\Big(\dashint_{B\cap\Omega} W_B^p\Big)^{\frac{1}{p}}
\leq
\sup_{x\in B\cap\Omega}|W_B(x)|
\lesssim \Big(\dashint_{10B\cap\Omega} U^2\Big)^{\frac{1}{2}}
+ \Big(\dashint_{2B\cap\Omega} F^2\Big)^{\frac{1}{2}},\\
&\Big(\dashint_{B\cap\Omega} V_B^2\Big)^{\frac{1}{2}}
\lesssim \Big(\dashint_{2B\cap\Omega} F^2\Big)^{\frac{1}{2}}.
\end{aligned}
\end{equation}

Hence, the condition $\eqref{pri:5.6}$ has been verified by
the estimates $\eqref{f:5.7}$, $\eqref{f:5.8}$
and $\eqref{f:5.9}$, and so one may obtain
\begin{equation*}
\begin{aligned}
\Big(\int_{\Omega} U^p\Big)^{1/p}
\lesssim \Big(\int_{\Omega} U^2\Big)^{1/2}
+ \Big(\int_{\Omega} F^p\Big)^{1/p}
\end{aligned}
\end{equation*}
for any $p\geq 2$,
and this together with
\begin{equation*}
\begin{aligned}
\int_{\Omega} U^2
=
\int_{\Omega}
\dashint_{B_\varepsilon(x)\cap\Omega_\varepsilon}
|\nabla u_\varepsilon|^2 dx
&\lesssim \int_{\Omega_0}
|\nabla \tilde{u}_\varepsilon|^2 dx\\
&\lesssim \int_{\Omega_\varepsilon}
|\nabla u_\varepsilon|^2 dx
\lesssim  \int_{\Omega_\varepsilon}
|f|^2 dx \lesssim
\int_{\Omega}
F^{2} dx
\lesssim
\Big(\int_{\Omega}
F^{p} dx\Big)^{2/p}
\end{aligned}
\end{equation*}
implies the desired estimate $\eqref{pri:5.1}$,
where we employ Lemma $\ref{extensiontheory}$ in the
second inequality, and also use the zero extension
of $f$ and $\eqref{pri:5.4}$ in the fourth one.
We have completed the whole proof.
\qed


\vspace{0.2cm}

\noindent \textbf{The proof of Theorem \ref{thm:1.4}.}
For the ease of the statement, we still use
the notation imposed in the proof of Theorem $\ref{thm:1.2}$.
The main idea is to reuse the estimate $\eqref{pri:1.2}$ in the
case of $1<p<2$ and maximal function arguments. Roughly speaking,
we use the arguments developed in Lemma $\ref{lemma:5.4}$ twice.
To achieve the goal, we define the localized Hardy-Littlewood maximal function
as
\begin{equation}\label{def:1}
\mathcal{M}_{\Omega}(U)(x) = \sup_{Q\subset\Omega\atop
Q\ni x} \dashint_{Q} |U(y)|dy,
\end{equation}
where $Q$ is cubes in $\mathbb{R}^d$. Let $\mathcal{M}$ be
the Hardy-Littlewood maximal function, and it is not hard
to see $\mathcal{M}_{\Omega}(U)(x)\leq \mathcal{M}(I_{\Omega}U)(x)$
for any $x\in\Omega$.
Due to the relationship between
the Hardy-Littlewood maximal functions and $A_p$ weights,
there holds
\begin{equation}\label{f:5.12}
\begin{aligned}
\Big(\int_{\Omega} \big[\mathcal{M}_{\Omega}(F^s)\big]^{\frac{p}{s}}\rho
dx\Big)^{1/p}
\leq
\Big(\int_{\mathbb{R}^d}
\big[\mathcal{M}(I_{\Omega}F^s)\big]^{\frac{p}{s}}\rho
dx\Big)^{1/p}
\lesssim
\Big(\int_{\Omega} |F|^p\rho dx\Big)^{1/p},
\end{aligned}
\end{equation}
where
we use \cite[Theorem 2.5]{D} in the second inequality, and
where we ask for $1<s<p<\infty$ and
$\rho\in A_{p/s}$.
For $\rho$ satisfies the reverse H\"older property, i.e.,
\begin{equation}\label{f:8}
 \Big(\dashint_{Q} \rho^{1+\epsilon}
 \Big)^{\frac{1}{1+\epsilon}}
 \lesssim \dashint_Q \rho,
\end{equation}
(In fact, the above estimate holds for any $\rho\in A_p$
with $1\leq p<\infty$.) and
this property implies that if
$\rho\in A_{p/s}$ then it will belong to $A_p$ whenever
$s$ is close to 1, which plays an important role in the
whole arguments.
To show the estimate $\eqref{pri:1.5}$, it will be accomplished
through
\begin{equation}\label{f:1}
\int_{\Omega} \big[\mathcal{M}_{\Omega}(U^s)(x)
\big]^{\frac{p}{s}}\rho(x) dx
\lesssim
\rho(\Omega)
\bigg(\dashint_{\Omega} U^s\bigg)^{\frac{p}{s}}
+
\int_{\Omega}
\big[\mathcal{M}_{\Omega}(F^s)(x)\big]^{\frac{p}{s}}\rho(x)
dx.
\end{equation}
On the one hand, from the fact that
$[U(x)]^p\leq
[\mathcal{M}_{\Omega}(U^s)(x)]^{\frac{p}{s}}$
for a.e. $x\in\Omega$, there holds
\begin{equation}\label{f:5.14}
\int_{\Omega} [U(x)]^p\rho(x)dx
\leq \int_{\Omega}
\big[\mathcal{M}_{\Omega}(U^s)(x)\big]^{\frac{p}{s}}\rho(x) dx.
\end{equation}
On the other hand, it follows from the quenched
Calder\'on-Zygmund estimate
$\eqref{pri:1.2}$ that
\begin{equation}\label{f:5.15}
\begin{aligned}
\rho(\Omega)
\bigg(\dashint_{\Omega} U^s\bigg)^{\frac{p}{s}}
&\lesssim
\rho(\Omega)
\bigg(\dashint_{\Omega} F^s\bigg)^{\frac{p}{s}}\\
&\lesssim\rho(\Omega)\Big(\dashint_{\Omega}
F^p\rho dx\Big)\Big(\dashint_{\Omega}\rho^{-\frac{s}{p-s}}dx
\Big)^{\frac{p-s}{s}}
\lesssim \int_{\Omega} F^p\rho dx
\end{aligned}
\end{equation}
in which we use
the fact that there exists an
universal constant $C$ such that
\begin{equation*}
 \Big(\dashint_{\Omega}\rho dx\Big)
 \Big(\dashint_{\Omega}\rho^{-\frac{s}{p-s}}dx
\Big)^{\frac{p-s}{s}}
\leq C,
\end{equation*}
since $\rho$ belongs to $A_{p/s}$ classes.
Hence combining the estimates $\eqref{f:1}$, $\eqref{f:5.12}$,
$\eqref{f:5.14}$ and $\eqref{f:5.15}$ we arrive at the
desired estimate $\eqref{pri:1.5}$.

Thus, our task is reduced to establish the estimate
$\eqref{f:1}$. By a real method developed by Z. Shen,
it suffices to verify
the following two conditions:
for any ball $B:=B_r(x_0)\subset\mathbb{R}^d$ with
$x_0\in\partial\Omega$ or $4B\subset\Omega$, there exists
$W_B$ and $V_B$ satisfying
\begin{equation}\label{f:2}
        \Big(\dashint_{B\cap\Omega}|V_B|^s\Big)^{1/s}
        \lesssim \Big(\dashint_{5B\cap\Omega}|F|^s)^{1/s}
\end{equation}
and
\begin{equation}\label{f:3}
        \sup_{B\cap\Omega}
        W_B \lesssim \Big(\dashint_{3B\cap\Omega}|W_B|^s\Big)^{1/s}
\end{equation}
where $1<s<\infty$. Obviously, we are
interested in the case
$1<s<2$. As we have shown in the proof of Theorem $\ref{thm:1.2}$,
we take $W_B = U$ for $0<r<(\varepsilon/4)$, while we set
$u_\varepsilon=v_\varepsilon + w_\varepsilon$ in the case of
$r\geq (\varepsilon/4)$, where $v_\varepsilon$ and $w_\varepsilon$
are the solutions of $\eqref{pde:5.2}$, respectively.
Therefore, $U\leq V_B + W_B$.

We first to show the estimate $\eqref{f:3}$.
For any $r>0$, it follows from the estimates $\eqref{f:5.5}$
and $\eqref{f:5.16}$ that
\begin{equation}\label{f:5}
\sup_{B\cap\Omega} W_B \lesssim \Big(
\dashint_{2B\cap\Omega} W_B^2\Big)^{1/2}
\end{equation}
This in fact implies that for any $0<s<2$ we have
\begin{equation*}
\Big(\dashint_{\frac{1}{2}B\cap\Omega} W_B^2\Big)^{1/2}
\lesssim
\Big(\dashint_{B\cap\Omega} W_B^s\Big)^{1/s}
\end{equation*}
where we employ an convexity argument (see \cite[pp.173]{FS}),
and this together with $\eqref{f:5}$ leads to
the stated estimate $\eqref{f:3}$.
Now, we proceed to verify the estimate $\eqref{f:2}$.
The case $0<r<(\varepsilon/4)$ is trivial since one may
prefer $V_B =0$, while in the case of $r>(\varepsilon/4)$,
it follows from the quenched Calder\'on-Zygmund estimate
$\eqref{pri:1.2}$ that
\begin{equation}\label{f:7}
\begin{aligned}
\dashint_{B} |V_B|^s dx
\lesssim \frac{1}{|B|}\int_{\Omega}
|V_B|^s dx
\lesssim
\frac{1}{|B|}\int_{\Omega}
\Big(\dashint_{B_\varepsilon(x)\cap\Omega_\varepsilon} |fI_B|^2
\Big)^{\frac{s}{2}} dx
\lesssim
\dashint_{5B\cap\Omega}
F^s dx
\end{aligned}
\end{equation}
and its last inequality follows from
the observation that
$\text{supp}\big(I_B* I_{B_\varepsilon}\big)\subset
\text{supp}(I_B)\cup\text{supp}(I_{B_\varepsilon})$.
Moreover, we have $|x-x_0|\leq |x-y|+|y-x_0|\leq \varepsilon
+ r \leq 5r$ for any $y\in B$ and some $x\in\Omega$. This
yields the estimate $\eqref{f:2}$.

Hence, the remainder of the proof is devoted to show the estimate $\eqref{f:1}$ under
the conditions $\eqref{f:2}$ and $\eqref{f:3}$.
Let $E(\lambda) = \{x\in\Omega:\mathcal{M}_{\Omega}(U^s)(x)>\lambda\}$, and
$K(\lambda) = \{x\in\Omega:\mathcal{M}_{\Omega}(F^s)(x)>\lambda\}$.
Let $\epsilon$ be given in $\eqref{f:8}$, it suffices
to prove the following good-$\lambda$ inequality
\begin{equation}\label{f:9}
\rho(E(T\lambda)) \leq \delta^\sigma \rho(E(\lambda))
+ \rho(K(\theta\lambda))
\end{equation}
for any $\lambda\geq\lambda_0$, where $T = (2\delta^\sigma)^{-\frac{s}{p}}$,
$\sigma = \epsilon/(1+\epsilon)$, and
$\delta,\theta\in(0,1)$ will be determined later. There exists
$C_0=C_d/\delta$, such that
\begin{equation*}
  \lambda_0 = C_0\dashint_{\Omega} U^{s}(y)dy
  \quad
  \text{and}
  \quad |E(\lambda)|\leq \delta|\Omega|.
\end{equation*}
Observe that
\begin{equation*}
E(T\lambda)\subset \big\{ E(T\lambda)\cap K^{c}(\theta\lambda)\big\}
\cup K(\theta\lambda),
\end{equation*}
where the notation $K^c(\theta\lambda)$ represents the complementary set of
$K(\theta\lambda)$. Due to the property
\begin{equation}\label{f:11}
\rho(S)\leq \Big(\frac{|S|}{|Q|}\Big)^\sigma\rho(Q)
\end{equation}
for any $S\subset Q\subset\mathbb{R}^d$,
(which is true for any $\rho\in A_p$ with
$1\leq p<\infty$,) it suffices to show
$\big|E(T\lambda)\cap K^{c}(\theta\lambda)
\big|\leq \delta |E(\lambda)|$ as $\lambda\geq \lambda_0$.
It will be done by  the Calder\'on-Zygmund
decomposition, that means the previous
estimate is reduced to show
\begin{equation}\label{f:10}
  \big|E(T\lambda)\cap K^{c}(\theta\lambda)\cap Q_j
  \big| \leq \delta|Q_j|,
\end{equation}
where $\{Q_j\}$ is a family of dyadic cubes, satisfying $E(\lambda) = \cup_j Q_j$
and $Q_j \cap Q_i = \emptyset$ if $i\not=j$,
and
\begin{equation}\label{f:5.13}
\lambda<\dashint_{Q_j}U^s
 \leq 2^d \lambda.
\end{equation}
Since $T\lambda < \mathcal{M}_{\Omega}(U^s)(x)
\leq \mathcal{M}(U^{s}I_{2Q_j})(x)$ if $x\in Q_j$,
we have
\begin{equation*}
Q_j\cap E(T\lambda)\subset
\{x\in Q_j:\mathcal{M}(U^s I_{2Q_j})>T\lambda\},
\end{equation*}
and therefore,
\begin{equation*}
\begin{aligned}
&\big\{x\in Q_j: \mathcal{M}(U^s I_{2Q_j})>T\lambda\big\} \\
&\subset \big\{x\in Q_j: \mathcal{M}((U-W_{B_j})^sI_{2Q_j})>\frac{T\lambda}{2^s}\big\}
\cup \big\{x\in Q_j: \mathcal{M}((W_{B_j})^s I_{2Q_j})>\frac{T\lambda}{2^s}\big\}
=:A \cup B.
\end{aligned}
\end{equation*}
Then it follows from the estimate $\eqref{f:2}$ and
weak (1,1) property of the Hardy-Littlewood maximal function
(see for example \cite[Theorem 2.5]{D}) that
\begin{equation}
\begin{aligned}
|A| \leq \frac{C_d}{T\lambda}\int_{2B_j\cap\Omega}|U-W_{B_j}|^s dy
\leq \frac{C_d|Q_j|}{T\lambda}\dashint_{6B_j\cap\Omega} F^s dy
\leq C_d|Q_j|T^{-1}\theta,
\end{aligned}
\end{equation}
where the last inequality follows from the condition
$Q_j\cap K^{c}(\theta\lambda)\not=\emptyset$.
In view of the estimates $\eqref{f:2}$ and $\eqref{f:3}$,
we also acquire
\begin{equation}
\begin{aligned}
|B| &\leq \frac{C_d}{(T\lambda)^{\frac{q}{s}}}
\int_{2B_j\cap\Omega}|W_{B_j}|^q dy \\
&\leq \frac{C_d|Q_j|}{(T\lambda)^{\frac{q}{s}}}
\bigg\{\Big(\dashint_{4B_j\cap\Omega} U^s dy\Big)^{\frac{1}{s}}
+  \Big(\dashint_{4B_j\cap\Omega} F^s dy\Big)^{\frac{1}{s}}\bigg\}^q \\
&\leq C_d|Q_j|(T\lambda)^{-\frac{q}{s}}
\Big\{\lambda^{\frac{q}{s}} + (\theta\lambda)^{\frac{q}{s}}\Big\} \\
&\leq C_d|Q_j|T^{-\frac{q}{s}},
\end{aligned}
\end{equation}
where $q>\frac{p}{\sigma}$, and we use the conditions
$\eqref{f:5.13}$ and $Q_j\cap K^{c}(\theta\lambda)
\not=\emptyset$ in the third
inequality.
Now, it is not hard to see
\begin{equation*}
\big|E(T\lambda)\cap K^{c}(\theta\lambda)\cap Q_j\big|
\leq \delta|Q_j|\Big\{C_d \delta^{\sigma(s/p)-1}\theta
+ C_d\delta^{\sigma(q/p)-1}\Big\}
\end{equation*}
by recaling $T=(2\delta^\sigma)^{-\frac{s}{p}}$.
Therefore one may choose
$\delta\in(0,1)$ small enough such that
$C_d\delta^{\sigma(q/p)-1}\leq 1/2$, and
then by choosing $\theta\in(0,1)$ we can also have
$C_d \delta^{\sigma(s/p)-1}\theta\leq 1/2$. Thus,
we have proved the desired estimate $\eqref{f:10}$ and
this together with the estimate
$\eqref{f:11}$ leads to
\begin{equation*}
   \rho(E(T\lambda)\cap K^{c}(\theta\lambda))
\leq \delta^\sigma \rho(E(\lambda))
\end{equation*}
and consequently yields $\eqref{f:9}$.

In the last step, we show the estimate $\eqref{f:1}$
by a routine computation. On account of $\eqref{f:9}$, it is not hard to see
\begin{equation*}
\int_{\lambda_0}^{N_0}\lambda^{\frac{p}{s}-1}
\rho(E(T\lambda)) d\lambda
\leq \delta^{\sigma}\int_{\lambda_0}^{N_0}\lambda^{\frac{p}{s}-1}
\rho(E(\lambda)) d\lambda
+ \int_{0}^{N_0}\lambda^{\frac{p}{s}-1}
\rho(K(\theta\lambda)) d\lambda
\end{equation*}
holds for any $N_0>T\lambda_0$. Moreover, by changing variable and
merger of similar items,
\begin{equation*}
\big(T^{-\frac{p}{s}}-\delta^\sigma\big)
\int_{T\lambda_0}^{N_0}
\lambda^{\frac{p}{s}-1}\rho(E(\lambda)) d\lambda
\leq C\rho(\Omega)\lambda_0^{\frac{p}{s}}
+C\int_{\Omega} \big[\mathcal{M}_{\Omega}(F^s)(x)
\big]^{\frac{p}{s}} \rho dx.
\end{equation*}
Noting that $T^{-\frac{p}{s}}-\delta^\sigma = \delta^\sigma$,
we obtain
\begin{equation*}
\int_{0}^{N_0}
\lambda^{\frac{p}{s}-1}\rho(E(\lambda)) d\lambda
\leq C\rho(\Omega)\lambda_0^{\frac{p}{s}}
+C\int_{\Omega} \big[\mathcal{M}_{\Omega}(F^s)(x)\big]^{\frac{p}{s}} \rho dx
\end{equation*}
and then
let $N_0\to\infty$. This consequently yields the stated estimate
$\eqref{f:1}$, and we have completed the whole argument.
\qed


\section{Boundary Estimates at Large-scales}\label{sec:7}

\begin{lemma}[boundary Caccioppoli's inequality]
Suppose that the coefficient $A$ satisfies $\eqref{eq:1.4}$. Let
$u_\varepsilon\in H^1(D_4^\varepsilon;\mathbb{R}^d)$
be the solution to
$\mathcal{L}_\varepsilon(u_\varepsilon)=0$ in
$D_4^{\varepsilon}$ and $\sigma_\varepsilon(u_\varepsilon) = 0$
on $D_4\cap\partial(\varepsilon\omega)$ and
$u_\varepsilon = 0$ on $\Delta_4\cap \varepsilon\omega$. Then,
for any $0<r\leq 1$,
one may have
\begin{equation}\label{pri:5.2}
\bigg(\dashint_{D_r^{\varepsilon}} |\nabla u_\varepsilon|^2
\bigg)^{1/2}
\lesssim \frac{1}{r}\bigg(\dashint_{D_{2r}^{\varepsilon}}
|u_\varepsilon|^2\bigg)^{1/2},
\end{equation}
where the constant depends only on $\mu_0$, $\mu_1$ and $d$.
\end{lemma}

\begin{proof}
The proof is standard, and we just let the test function
be $\varphi^2 u_\varepsilon$ and then use the elasticity
conditions $\eqref{eq:1.4}$
(here we also need to use Korn's inequality).
\end{proof}

\begin{lemma}[quenched $L^2$-error estimates]\label{lemma:5.2}
Let $\varepsilon\leq r\leq 1$.
Suppose that $\mathcal{L}_{\varepsilon}$
and $\omega$ satisfy the hypothesises
\emph{(H1)} and \emph{(H2)}. Let $u_\varepsilon$ be
a weak solution to $\mathcal{L}_\varepsilon(u_\varepsilon) = 0$
in $D_{4r}^{\varepsilon}$,
$\sigma_\varepsilon(u_\varepsilon)=0$ on $D_{4r}
\cap\partial(\varepsilon\omega)$ and $u_\varepsilon = 0$ on
$\Delta_{4r}
\cap\varepsilon\omega$.
Then there exist a weak solution $v\in H^1(D_r;\mathbb{R}^d)$
such that $\mathcal{L}_0(v) = 0$ in $D_{r}$ and
$v=0$
on $\Delta_{r}$, and
\begin{equation}\label{pri:5.1}
 \bigg(\dashint_{D_r^\varepsilon}
 |u_\varepsilon - v|^2\bigg)^{1/2}
 \lesssim \Big(\frac{\varepsilon}{r}\Big)^{1/2}
 \bigg(\dashint_{D_{3r}^\varepsilon}|u_\varepsilon|^2\bigg)^{1/2}.
\end{equation}
\end{lemma}


\begin{proof}
By rescaling one may assume $r=1$.
First, we should extend $u_\varepsilon$
from $H^{1}(D_2^{\varepsilon})$
to $H^1(D_3)$, denoted by $\tilde{u}_\varepsilon$. To do so,
let $\varphi$ be a smooth cut-off function satisfying
$\varphi=1$ on $D_2$ and $\varphi=0$ outside $D_{5/2}$ with
$|\nabla\varphi|\lesssim 1$, and $\varphi u_\varepsilon = 0$
on $\partial D_{5/2}^{\varepsilon}$. Then one may apply
the extension result of Lemma $\ref{extensiontheory}$ to
$\varphi u_\varepsilon$ and we have
$\|\tilde{u}_\varepsilon\|_{H^1(D_3)}\lesssim
\|\varphi u_\varepsilon\|_{H^1(D_{5/2}^\varepsilon)}$.
This together with boundary Caccioppoli's
inequality $\eqref{pri:5.2}$ gives
\begin{equation}\label{f:5.2}
\|\tilde{u}_\varepsilon\|_{H^1(D_3)}
\lesssim
\|\varphi u_\varepsilon\|_{H^1(D_{5/2}^\varepsilon)}
\lesssim \|u_\varepsilon\|_{L^2(D_3^\varepsilon)}.
\end{equation}
Thus, there exists $t\in [5/4,3/2]$ such that
\begin{equation}\label{f:5.1}
\|\tilde{u}_\varepsilon\|_{H^{1}(\partial D_t
\setminus\Delta_t)}
\lesssim \|u_\varepsilon\|_{L^2(D_3)}
\quad
\text{and}
\quad \tilde{u}_\varepsilon = u_\varepsilon
\quad \text{in}~D_t^\varepsilon.
\end{equation}
Moreover, in terms of the radius $t$, we construct
$v_h\in H^1(D_t;\mathbb{R}^d)$ such that
$\mathcal{L}_0(v_h) = 0$ in $D_t$ and
$v_h=\tilde{u}_\varepsilon$ on $\partial D_t$.
It follows from the estimates $\eqref{pri:1.4}$ and
$\eqref{f:5.1}$ that
\begin{equation}\label{f:5.3}
\begin{aligned}
\|u_\varepsilon - v_h\|_{L^2(D_1^\varepsilon)}
\leq \|u_\varepsilon - v_h\|_{L^2(D_t^\varepsilon)}
\lesssim \varepsilon^{1/2}\|\tilde{u}_\varepsilon\|_{
H^1(\partial D_t\setminus\Delta_t)}
\lesssim \varepsilon^{1/2}\|u_\varepsilon\|_{
L^2(D_3)}.
\end{aligned}
\end{equation}
Then, we further construct the equation
$\mathcal{L}_\varepsilon(v) = 0$ in $D_t$ and
$v = \widetilde{u}_\varepsilon^0$ on $\partial D_t$,
where $\widetilde{u}_\varepsilon^0 = 0$ on $\Delta_t$ and
$\widetilde{u}_\varepsilon^0 = \tilde{u}_\varepsilon$ on
$\partial D_t\setminus\Delta_t$. Let $w=v_h-v$, and there
holds $\mathcal{L}_0(w) = 0$ in $D_t$ and
$w=\tilde{u}_\varepsilon-\widetilde{u}_\varepsilon^0$ on
$\partial D_t$.
Note that $\Delta_t\setminus\Gamma_\varepsilon$ represents
the holes intersected with the boundary set $\Delta_t$, which
are disconnected. For each connected component, one may appeal
to the trace theorem and its estimated constant would
rely on the diameter of the component
which is around the $\varepsilon$-scale. Thus, we obtain
the following computation:
\begin{equation}\label{f:5.10}
\begin{aligned}
\int_{\Delta_t\setminus\Gamma_\varepsilon}
|\tilde{u}_\varepsilon|^2dS
&\lesssim \frac{1}{\varepsilon}
\int_{O_\varepsilon\cap D_t}
|\tilde{u}_\varepsilon|^2dx
+\varepsilon
\int_{O_\varepsilon\cap D_t}
|\nabla \tilde{u}_\varepsilon|^2dx\\
&\lesssim \varepsilon
\int_{O_\varepsilon\cap D_t}
|\nabla \tilde{u}_\varepsilon|^2dx
\lesssim \varepsilon\int_{D_{3/2}}
|\nabla \tilde{u}_\varepsilon|^2dx
\lesssim \varepsilon\int_{D_3^\varepsilon}|u_\varepsilon|^2dx
\end{aligned}
\end{equation}
where we use Poincar\'e's inequality in the second step.
The last one
follows from the estimate $\eqref{f:5.2}$.
Then, on account of nontangential maximal function
estimates (see \cite[Theorem 3.6]{DKV}) that
\begin{equation}\label{f:5.4}
\|w\|_{L^2(D_t)}
\lesssim \|(w)^*\|_{L^2(\partial D_t)}
\lesssim \|\tilde{u}_\varepsilon\|_{L^2(\Delta_t)}
\lesssim \varepsilon^{1/2}\|u_\varepsilon\|_{
L^2(D_3)},
\end{equation}
where we refer the reader to $\eqref{def:5}$
for the notation $(w)^*$,
and we employ the estimate $\eqref{f:5.10}$ in the last
inequality.
Combining the estimates $\eqref{f:5.3}$ and $\eqref{f:5.4}$
we have
\begin{equation*}
\|u_\varepsilon - v\|_{L^2(D_1)}
\leq \|u_\varepsilon - v_h\|_{L^2(D_1)}
+ \|w\|_{L^2(D_1)}
\lesssim \varepsilon^{1/2}\|u_\varepsilon\|_{
L^2(D_3)},
\end{equation*}
and rescaling back leads to the desired result
$\eqref{pri:5.1}$. We have completed the proof.
\end{proof}

For the ease of the statement, we impose the following
notation:
\begin{equation}\label{eq:5.1}
\begin{aligned}
G(r,u) &:= \frac{1}{r}\inf_{M\in\mathbb{R}^{d\times d}}
\bigg\{
\Big(\dashint_{D_r} |u-Mx|^2 dx\Big)^{\frac{1}{2}}
+ r\|T(M)\|_{L^\infty(\Delta_r)}+\|Mx\|_{L^\infty(\Delta_r)}\bigg\},\\
G_\varepsilon(r,u) &:=
\frac{1}{r}\inf_{M\in\mathbb{R}^{d\times d}}\bigg\{
\Big(\dashint_{D_r^{\varepsilon}} |u-Mx|^2 dx\Big)^{\frac{1}{2}}
+ r\|T(M)\|_{L^\infty(\Delta_r)}
+\|Mx\|_{L^\infty(\Delta_r)}\bigg\},
\end{aligned}
\end{equation}
where we denote $T(M) := \nabla_{\text{tan}}(Mx)$,
and its component may be
written as $(n_jM_{ik}-n_iM_{jk})$ and $n=(n_1,\cdots,n_d)$
is the outward unit normal to $\partial\Omega$.
The notation $D_r, \Delta_r$ and $D_r^\varepsilon$ are
introduced in Subsection $\ref{notation}$.

\begin{lemma}[comparing at large-scales]\label{lemma:5.1}
Let $\varepsilon\leq r\leq 1$.
Suppose that $\omega$ satisfies the hypothesis \emph{(H2)}.
Assume that $v\in
H^{1}(D_{4r};\mathbb{R}^d)$ is a solution to
$\mathcal{L}_0(v) = 0$ in $D_{4r}$ with
$v = 0$ on $\Delta_{4r}$.
Then one may derive that
\begin{equation}\label{pri:5.7}
G(r,v) \lesssim G_{\varepsilon}(2r,v),
\end{equation}
where the up to constant depends on $\mu_0$, $\mu_1$, $d$,
$\mathfrak{g}^{\omega}$ and
the character of $\omega$.
\end{lemma}

\begin{proof}
Let $\tilde{v} = v-Mx$, and we have $\mathcal{L}_0(\tilde{v}) = 0$ in
$D_{4r}$ with $\tilde{v} = -Mx$ on $\Delta_{4r}$.
The proof is reduced to show there exists a constant,
in dependent of $\varepsilon$ and $r$, such that
\begin{equation}\label{f:5.11}
\dashint_{D_r}|\tilde{v}|^2
\leq C\dashint_{D_{2r}^\varepsilon}|\tilde{v}|^2
+ r^2\|T(M)\|_{L^\infty(\Delta_{2r})}^2 + \|Mx\|_{L^\infty(\Delta_{2r})}^2.
\end{equation}
If so, for any $M\in\mathbb{R}^{d\times d}$
one may have
\begin{equation*}
\begin{aligned}
\Big(\dashint_{D_r}|v-Mx|^2\Big)^{1/2}
&+ r\|T(M)\|_{L^\infty(\Delta_{r})} + \|Mx\|_{L^\infty(\Delta_{r})}\\
&\lesssim \Big(\dashint_{D_{2r}^\varepsilon}|v-Mx|^2\Big)^{1/2}
+r\|T(M)\|_{L^\infty(\Delta_{2r})} + \|Mx\|_{L^\infty(\Delta_{2r})},
\end{aligned}
\end{equation*}
and this implies
the desired estimate $\eqref{pri:5.7}$.
The reminder of the proof is
devoted to the estimate $\eqref{f:5.11}$.
Recalling that $\mathbb{R}^d\setminus \varepsilon\omega
=\cup_{k=1}^\infty H_k^\varepsilon$,
we impose two parameters $0<c_1<c_2<
(\mathfrak{g}^{\omega}/10)$.
In fact,
let $\varphi_1^{\varepsilon} = 1$ on $\mathbb{R}^d
\setminus(\varepsilon\omega)$ and $\varphi_1^{\varepsilon}(x) = 0$
whenever $\text{dist}(x,H_k^\varepsilon)>c_1\varepsilon\}$
for each $k$ and
$|\nabla\varphi_1^\varepsilon|\lesssim 1/\varepsilon$,
while we set $\varphi_2^{\varepsilon}(x) = 1$ if
$\text{dist}(x,H_k^\varepsilon)>c_1\varepsilon\}$ for each $k$,
and $\varphi_2^{\varepsilon}(x) = 0$
if $\text{dist}(x,H_k^\varepsilon)>c_2\varepsilon$ and
$|\nabla\varphi_2^\varepsilon|\lesssim 1/\varepsilon$.
These two are cut-off functions
whose supports are around the holes $\{H_k^\varepsilon\}$. Since
\begin{equation*}
 \dashint_{D_{r}}|\tilde{v}|^2 \leq
 \dashint_{D_{r}^\varepsilon}|\tilde{v}|^2
 +\dashint_{D_{r}\setminus\varepsilon\omega}|\tilde{v}|^2,
\end{equation*}
we just study the second term of the right-hand side above, and
\begin{equation*}
\begin{aligned}
\dashint_{D_{r}\setminus\varepsilon\omega}|\tilde{v}|^2
\leq \dashint_{D_r\cap\text{supp}(\varphi_1^\varepsilon)}
|\varphi_1^\varepsilon \tilde{v}|^2
&\lesssim \varepsilon^2
\dashint_{D_{r}\cap\text{supp}(\varphi_1^\varepsilon)}
|\nabla(\varphi_1^\varepsilon \tilde{v})|^2 \\
&\lesssim
\dashint_{D_{2r}^\varepsilon}
|\tilde{v}|^2
+ \varepsilon^2\dashint_{\tilde{D}_{r}
\cap\text{supp}(\varphi_2^\varepsilon)}
|\varphi_2^\varepsilon\nabla \tilde{v}|^2\\
&\lesssim
\dashint_{D_{2r}^\varepsilon}
|\tilde{v}|^2 + \|Mx\|_{L^\infty(\Delta_{2r})}^2
+\varepsilon^2\|T(M)\|_{L^\infty(\Delta_{2r})}^2,
\end{aligned}
\end{equation*}
where one may prefer the region $\tilde{D}_r$ such that
$D_r\subset\tilde{D}_r\subset D_{2r}$ and
$\varphi_2^\varepsilon v = 0$
on $\partial\tilde{D}_r\setminus\Delta_{2r}$, and the last inequality above
actually comes from the boundary Caccioppoli's inequality
$\eqref{pri:8.1}$.
This proved the stated estimate $\eqref{f:5.11}$ (by noting
that $\varepsilon\leq r\leq 1$),
and we have completed the whole proof.
\end{proof}

\begin{lemma}[decay estimates for homogenized equations]
\label{lemma:5.3}
Let $\varepsilon\leq r<1$ and $\Omega$ be
a bounded $C^{1,\eta}$ domain.
Let $v$ be the solution to $\mathcal{L}_0(v) = 0$ in
$D_{4r}$ with  $v=0$ on $\Delta_{4r}$.
Then for any $\theta\in(0,1)$ there holds
\begin{equation}\label{pri:5.8}
G_\varepsilon(\theta r,v)
\lesssim \theta^{\kappa}
G_\varepsilon(r,v)
\end{equation}
where $\kappa\in(0,1)$,
and the up to constant depends $\mu_0,\mu_1,d$,
$\mathfrak{g}^{\omega}$ and
the characters of $\omega$ and $\Omega$.
\end{lemma}

\begin{proof}
The desired result follows from Lemma $\ref{lemma:5.1}$ and
classical boundary Schauder estimates. The argument is the same
to that given in \cite[Lemma 7.1]{AS} and we provide a proof for
the sake of the completeness.
One may prefer $M=\nabla v(x)$ for any $x\in\Delta_{4r}$.
Then $T(M)=\nabla_{\text{tan}}v = 0$ (since $v=0$ on $\Delta_{4r}$),
and we have
\begin{equation*}
\begin{aligned}
G_\varepsilon(\theta r,v)
\leq G(\theta r,v)
&\leq C(\theta r)^\kappa
\big[\nabla v\big]_{C^{0,\kappa}(D_{\theta r})}
+\|v-\nabla v \cdot x\|_{L^\infty(D_{\theta r})}
\lesssim
(\theta r)^\kappa\big[\nabla v\big]_{C^{0,\kappa}(D_{r})}.
\end{aligned}
\end{equation*}
By noting that $[\nabla v]_{C^{0,\kappa}(D_{r})}
= [\nabla \tilde{v}]_{C^{0,\kappa}(D_{r})}$
and $\tilde{v} := v-Mx$ for any $M\in\mathbb{R}^{d\times d}$,
one may have
\begin{equation*}
r^\kappa\big[\nabla v\big]_{C^{0,\kappa}(D_{r})}
\lesssim \frac{1}{r}\bigg\{
\Big(\dashint_{D_{2r}}|\tilde{v}|^2dx\Big)^{1/2}
+\|Mx\|_{L^\infty(\Delta_{2r})}
\bigg\} + \|\nabla_{\text{tan}}(Mx)\|_{L^\infty(\Delta_{2r})},
\end{equation*}
since $\mathcal{L}_0(\tilde{v})=0$ in $D_{4r}$ and
$\tilde{v} = -Mx$ on $\Delta_{4r}$.
For $M\in\mathbb{R}^{d\times d}$
is arbitrary, the desired estimate $\eqref{pri:5.8}$
consequently follows
from the estimate $\eqref{pri:5.7}$,
and we have completed the proof.
\end{proof}

\begin{lemma}[iteration's inequality]
Suppose that $\mathcal{L}_{\varepsilon}$
and $\omega$ satisfy the hypothesises
\emph{(H1)} and \emph{(H2)}. Let
$u_\varepsilon$ be a weak solution to the equation
$\mathcal{L}_\varepsilon(u_\varepsilon) = 0$
in $D_{4}^\varepsilon$,
$\sigma_\varepsilon(u_\varepsilon) =0$ on $D_{4}
\cap\partial(\varepsilon\omega)$ and $u_\varepsilon = 0$
on $\Delta_{4}\cap\varepsilon\omega$.
Define the quantity
\begin{equation*}
\Phi(r) :=
\frac{1}{r}
\Big(\dashint_{D_r^{\varepsilon}}|u_\varepsilon|^2 dx
\Big)^{\frac{1}{2}}.
\end{equation*}
Then, there exists $\theta\in(0,1/4)$, such that
\begin{equation}\label{pri:5.10}
G_\varepsilon(\theta r,u_\varepsilon)
\leq \frac{1}{2}G_\varepsilon(r,u_\varepsilon)
+ C\Big(\frac{\varepsilon}{r}\Big)^{1/2}\Phi(r)
\end{equation}
holds for any $\varepsilon\leq r<1$.
\end{lemma}

\begin{proof}
The proof directly follows from Lemmas
$\ref{lemma:5.2}$ and $\ref{lemma:5.3}$ and we omit the details.
\end{proof}

\noindent \textbf{The proof of Theorem \ref{thm:1.3}.}
The desired estimate $\eqref{pri:5.5}$ mainly follows from
$\eqref{pri:5.10}$ coupled with \cite[Lemma 8.5]{S}
and Caccioppoli's inequality $\eqref{pri:5.2}$.
For the scheme on interior Lipschitz estimates
had already been well stated
in the proof \cite[Theorem 1.1]{BR} in details,
we left these computations to the reader.
\qed

\section{Appendix}\label{sec:8}
\begin{lemma}[Caccioppoli's inequality with nonvanishing
boundary data]
Let $\Omega$ be a Lipschitz domain and
$0<r\leq 1$. Suppose $\mathcal{L}$ is an elliptic operator with
a constant coefficient and satisfies the condition $\eqref{eq:1.4}$.
Let $v$ be the solution to $\mathcal{L}(v) = 0$ in $D_{2r}$
and $v = g$ on $\Delta_{2r}$ with
$g\in H^1(\Delta_{2r};\mathbb{R}^d)$. Then we have
\begin{equation}\label{pri:8.1}
\Big(\dashint_{D_r}|\nabla v|^2dx\Big)^{1/2}
\lesssim \frac{1}{r}\bigg\{\Big(\dashint_{D_r}|v|^2dx\Big)^{1/2}
+ \Big(\dashint_{\Delta_{2r}}|g|^2dx\Big)^{1/2}
\bigg\} + \Big(\dashint_{\Delta_{2r}}
|\nabla_{\emph{tan}}g|^2dx\Big)^{1/2},
\end{equation}
where the constant depends on $\mu_0,\mu_1,d$ and the character of
$\Omega$.
\end{lemma}

\begin{proof}
By rescaling argument it is fine to assume $r=1$.
Let $\tilde{g}$ be the extension of $g$ such that
$\tilde{g} = g$ on $\Delta_{2}$ with $\tilde{g}\in H_0^1(\Delta_3)$,
and $\|\tilde{g}\|_{H^1(\partial D_3)}\lesssim
\|g\|_{H^1(\Delta_{2})}$. Then we construct an auxiliary function
$G$ satisfying
\begin{equation*}
\mathcal{L}(G) = 0 \quad\text{in} \quad D_3
\qquad  G=\tilde{g} \quad\text{on}\quad \partial D_3.
\end{equation*}
Then there holds
\begin{equation}\label{f:8.1}
\begin{aligned}
&\|G\|_{L^2(D_3)}
\lesssim\|\nabla G\|_{L^2(D_3)}
 \lesssim \|\tilde{g}\|_{H^{1/2}(\partial D_3)}
 \lesssim \|\tilde{g}\|_{H^1(\partial D_3)}
 \lesssim \|g\|_{H^1(\Delta_2)},
\end{aligned}
\end{equation}
where we use the fact that $G = 0$ on
$\partial D_3\setminus\Delta_3$ in the first inequality.
Let $w = v-G$, and then $\mathcal{L}(w) = 0$ in $D_2$
and $w=0$ on $\Delta_2$, it follows from the estimate
similar to $\eqref{pri:5.2}$ that
\begin{equation*}
 \|\nabla w\|_{L^2(D_1)}
 \lesssim \|w\|_{L^2(D_2)}
\end{equation*}
and therefore we have
\begin{equation*}
\begin{aligned}
\|\nabla v\|_{L^2(D_1)}
&\leq \|\nabla w\|_{L^2(D_1)} + \|\nabla G\|_{L^2(D_3)} \\
&\lesssim \|w\|_{L^2(D_1)} + \|\nabla G\|_{L^2(D_3)}
\lesssim \|v\|_{L^2(D_2)} + \|g\|_{H^1(\Delta_2)},
\end{aligned}
\end{equation*}
where we use the estimate $\eqref{f:8.1}$ in the last step.
Then rescaling back there holds the desired estimate
$\eqref{pri:8.1}$ and we have completed the proof.
\end{proof}

\begin{theorem}[weighted $W^{1,p}$ estimates
for mixed boundary conditions] \label{app:thm:1}
Let $|\frac{1}{p}-\frac{1}{2}|<\frac{1}{2d}+\epsilon$ with
$0<\epsilon\ll 1$, and $\rho\in A_p$. Assume that
$\Omega$ is a bounded Lipschitz domain, and
$\partial\Omega =(\partial\Omega)_D\cup(\partial\Omega)_N$
with $(\partial\Omega)_D\cap(\partial\Omega)_N=\emptyset$ satisfy
the same geometry conditions
as in \cite[Theorem 1.5]{B}.
Suppose that $\bar{a}$ is a constant coefficient
satisfying the ellipticity and symmetry conditions $\eqref{eq:1.4}$.
Let $u$ be associated with $f$ by
\begin{equation}\label{app:pde:1}
\left\{\begin{aligned}
\nabla\cdot\bar{a}\nabla u &= \nabla\cdot f
&\quad&\emph{in}~~\Omega,\\
\partial u/\partial\nu &= -n\cdot f &\quad&\emph{on}~~
(\partial\Omega)_N,\\
u &= 0 &\quad&\emph{on}~~
(\partial\Omega)_D,
\end{aligned}\right.
\end{equation}
in which
$\partial u/\partial\nu :=n\cdot\bar{a}\nabla u$ is
the conormal derivative of $u$.
Then, there holds the weighted estimate
\begin{equation}\label{app:pri:1}
\Big(\int_{\Omega}|\nabla u|^p\rho dx\Big)^{1/p}
\lesssim \Big(\int_{\Omega}|f|^p\rho dx\Big)^{1/p},
\end{equation}
where the up to constant depends on $\mu_0, \mu_1, p, d$
and the character of $\Omega$.
\end{theorem}

\begin{proof}

For the ease of the statement, we
impose the following notation:
$\Delta_{4r}^N:=(\partial\Omega)_N\cap\partial D_{4r}$;
$\Delta_{4r}^D:=(\partial\Omega)_D
\cap\partial D_{4r}$. The main idea is
actually parallel to that given for Theorem $\ref{thm:1.4}$,
and we provide a proof for the reader's convenience.
The proof is divided into four steps.

Step 1. We first show the reverse H\"older's inequality
 \begin{equation}\label{app:f:2}
\Big(\dashint_{D_r}|\nabla u|^{\bar{p}}\Big)^{1/\bar{p}}
\lesssim \Big(\dashint_{D_{2r}}|\nabla u|^{2}\Big)^{1/2}
\end{equation}
holds for $\bar{p}=\frac{2d}{d-1}$, provided $u$ satisfies
$\nabla\cdot \bar{a}\nabla u = 0$ in $D_{4r}$, and
$\partial u/\partial\nu = 0$ on
$\Delta_{4r}^N$ with $u = 0$ on $\Delta_{4r}^D$.
Based upon the Sobolev embedding theorem and duality arguments
(see for example \cite[Remark 9.3]{SZW4}),
for $t\in(1,2)$,
one may derive that
\begin{equation}\label{app:f:1}
\begin{aligned}
\Big(\int_{D_{tr}}|\nabla u|^{\bar{p}}dx\Big)^{1/\bar{p}}
& \lesssim \Big(\int_{\partial D_{tr}}
|(\nabla u)^{*}|^{2}dx\Big)^{1/2}\\
&\lesssim \Big(\int_{\partial D_{tr}\setminus\Delta_{4r}}
|\nabla u|^{2}dS\Big)^{1/2}
+\Big(\int_{\Delta_{tr}^D}
|\nabla_{\text{tan}} u|^{2}dS\Big)^{1/2}
+\Big(\int_{\Delta_{tr}^N}
|\frac{\partial u}{\partial\nu}|^2 dS\Big)^{1/2}\\
&\lesssim \Big(\int_{\partial D_{tr}\setminus\Delta_{4r}}
|\nabla u|^{2}dS\Big)^{1/2},
\end{aligned}
\end{equation}
where we employ the Rellich's estimate \cite[Theorem 1.5]{B}
in the second inequality
(we remark that the Rellich's estimate comes from
Rellich's identity, which is still valid for
elliptic systems with a symmetry coefficient),
and the last inequality is due to the assumption on
the boundary data. Then, squaring and integrating both sides of
$\eqref{app:f:1}$ with respect to $t\in(1,2)$, we obtain
\begin{equation}\label{app:f:10}
\Big(\int_{D_{r}}|\nabla u|^{\bar{p}}dx\Big)^{2/\bar{p}}
\lesssim \frac{1}{r}\int_{D_{2r}}
|\nabla u|^{2}dx,
\end{equation}
and this implies the desired inequality $\eqref{app:f:2}$.
Moreover, it follows from a self-improvement property that
there exits a small parameter $\epsilon>0$, depending on
$\mu_0,\mu_1,d$ and the character of $\Omega$, such that
the estimate $\eqref{app:f:2}$ is still true for
the new index $\bar{p}^+:=\frac{2d}{d-1}+\epsilon$.

Step 2. The real arguments lead to
$W^{1,p}$ estimates for $2\leq p<\bar{p}^{+}$,
i.e.,
\begin{equation}\label{app:f:5}
\Big(\int_{\Omega}|\nabla u|^p dx\Big)^{1/p}
\lesssim \Big(\int_{\Omega}|f|^pdx\Big)^{1/p}.
\end{equation}
To do so, we decompose the equation $\eqref{app:pde:1}$
as follows:
\begin{equation}\label{app:pde:2}
(\text{i})\left\{\begin{aligned}
\nabla\cdot\bar{a}\nabla v &= \nabla\cdot (I_Bf)
&\quad&\text{in}~~\Omega,\\
\partial v/\partial\nu &= -n\cdot (I_Bf) &\quad&\text{on}~~
(\partial\Omega)_N,\\
v &= 0 &\quad&\text{on}~~
(\partial\Omega)_D,
\end{aligned}\right.
\qquad
(\text{ii})\left\{\begin{aligned}
\nabla\cdot\bar{a}\nabla w &= \nabla\cdot g
&\quad&\text{in}~~\Omega,\\
\partial w/\partial\nu &= -n\cdot g &\quad&\text{on}~~
(\partial\Omega)_N,\\
w &= 0 &\quad&\text{on}~~
(\partial\Omega)_D,
\end{aligned}\right.
\end{equation}
where $g:=(1-I_B)f$, and $B:=B(x,r)$ with
$r>0$ and $x\in\bar{\Omega}$ are arbitrary.
Hence, it is not hard to see that $u=v+w$.
On the one hand, from the equation $(\text{i})$ one may
get
\begin{equation}\label{app:f:3}
\Big(\dashint_{\frac{1}{4}B\cap\Omega}|\nabla v|^2\Big)^{1/2}
\lesssim
\Big(\dashint_{B\cap\Omega}|f|^2\Big)^{1/2},
\end{equation}
where we merely employ the energy estimate for the solution of
$(\text{i})$. One the other hand,
 it is known by Step 1 that $w$ in $(\text{ii})$
 satisfies the estimate
\begin{equation}\label{app:f:4}
\begin{aligned}
 \Big(\dashint_{\frac{1}{4}B\cap\Omega}|\nabla w|^{\bar{p}^+}
 \Big)^{1/\bar{p}^+}
&\lesssim
  \Big(\dashint_{\frac{1}{2}B\cap\Omega}|\nabla w|^{2}
 \Big)^{1/2}\\
&\lesssim
  \Big(\dashint_{\frac{1}{2}B\cap\Omega}|\nabla u|^{2}
 \Big)^{1/2}
 + \Big(\dashint_{\frac{1}{2}B\cap\Omega}|\nabla v|^{2}
 \Big)^{1/2}\\
&\lesssim^{\eqref{app:f:3}}
  \Big(\dashint_{\frac{1}{2}B\cap\Omega}|\nabla u|^{2}
 \Big)^{1/2}
 + \Big(\dashint_{B\cap\Omega}|f|^2\Big)^{1/2}.
\end{aligned}
\end{equation}
Thus, for any $2\leq p<\bar{p}^+$,
the estimates $\eqref{app:f:3}$ and $\eqref{app:f:4}$
together with Lemma $\ref{lemma:5.4}$ leads to
\begin{equation*}
\Big(\int_{\Omega}|\nabla u|^{p}dx\Big)^{1/p}
\lesssim \Big(\int_{\Omega}|\nabla u|^2\Big)^{1/2}
+ \Big(\int_{\Omega}|f|^p\Big)^{1/p}
\lesssim \Big(\int_{\Omega}|f|^p\Big)^{1/p},
\end{equation*}
and the last step follows from energy estimate
and H\"older's inequality. This completes the proof
of $\eqref{app:f:5}$. Consequently, the duality argument
implies that the estimate
$\eqref{app:f:5}$ holds for any $p$ satisfying
$|\frac{1}{p}-\frac{1}{2}|<\frac{1}{2d}+\epsilon$.

Step 3. Let $U=|\nabla u|$ and $F=|f|$.
For $\frac{2d}{d+1}-\epsilon<s<p<\frac{2d}{d-1}+\epsilon$,
we claim that the next estimate
\begin{equation}\label{app:f:6}
\int_{\Omega} \big[\mathcal{M}_{\Omega}(U^s)(x)
\big]^{\frac{p}{s}}\rho(x) dx
\lesssim
\rho(\Omega)
\bigg(\dashint_{\Omega} U^s\bigg)^{\frac{p}{s}}
+
\int_{\Omega}
\big[\mathcal{M}_{\Omega}(F^s)(x)\big]^{\frac{p}{s}}\rho(x)
dx
\end{equation}
implies the desired estimate $\eqref{app:pri:1}$, where
the localized Hardy-Littlewood maximal operator
$\mathcal{M}_{\Omega}$ is defined by $\eqref{def:1}$, and
$\rho\in A_{p/s}$. On the one hand,
recalling the estimate $\eqref{f:5.12}$ we have
\begin{equation}\label{app:f:7}
\begin{aligned}
\Big(\int_{\Omega} \big[\mathcal{M}_{\Omega}(F^s)\big]^{\frac{p}{s}}\rho
dx\Big)^{1/p}
\lesssim
\Big(\int_{\Omega} |f|^p\rho dx\Big)^{1/p}.
\end{aligned}
\end{equation}
Moreover, it follows from Step 2 and
the computation given for $\eqref{f:5.15}$  that
\begin{equation}\label{app:f:8}
\begin{aligned}
\rho(\Omega)
\bigg(\dashint_{\Omega} |\nabla u|^s\bigg)^{\frac{p}{s}}
&\lesssim
\rho(\Omega)
\bigg(\dashint_{\Omega} |f|^s\bigg)^{\frac{p}{s}}\\
&\lesssim\rho(\Omega)\Big(\dashint_{\Omega}
|f|^p\rho dx\Big)\Big(\dashint_{\Omega}\rho^{-\frac{s}{p-s}}dx
\Big)^{\frac{p-s}{s}}
\lesssim \int_{\Omega} |f|^p\rho dx.
\end{aligned}
\end{equation}
One the other hand, in terms of
the Lebesgue's differentiation theorem, it is known that for
a.e. $x\in\Omega$ there holds
$|\nabla u(x)|^s = \lim_{r\to0}\dashint_{B(x,r)}|\nabla u|^s$,
and so $|\nabla u(x)|^s\lesssim
\mathcal{M}_\Omega(|\nabla u|^s)(x)$. This implies
\begin{equation}\label{app:f:9}
  \int_{\Omega}|\nabla u(x)|^p\rho(x) dx\lesssim
  \int_{\Omega}\big[\mathcal{M}_\Omega(|\nabla u|^s)(x)
  \big]^{\frac{p}{s}}\rho(x) dx,
\end{equation}
since $\rho\geq 0$. Plugging the estimates
$\eqref{app:f:7}, \eqref{app:f:8}$ and $\eqref{app:f:9}$
back into $\eqref{app:f:6}$ we have completed the arguments
for the claim and therefore obtained the stated estimate
$\eqref{app:pri:1}$. Hence, the remainder of the proof is
devoted to $\eqref{app:f:6}$.

Step 4. Show the estimate $\eqref{app:f:6}$. In fact, it may
be reduced to the so-called good-$\lambda$
inequality $\eqref{f:9}$, which finally relies on
the decomposition $u=v+w$ as in $\eqref{app:pde:2}$, as well as,
the following estimates
\begin{equation}\label{app:f:11}
        \Big(\dashint_{\frac{1}{4}B\cap\Omega}|\nabla v|^s\Big)^{1/s}
        \lesssim \Big(\dashint_{B\cap\Omega}|f|^s\Big)^{1/s}
\end{equation}
and
\begin{equation}\label{app:f:12}
\Big(\dashint_{\frac{1}{4}B\cap\Omega}
|\nabla w|^{\bar{p}^+}\Big)^{1/\bar{p}^+}
 \lesssim \Big(\dashint_{B\cap\Omega}|\nabla w|^s\Big)^{1/s}.
\end{equation}
The estimate $\eqref{app:f:11}$ follows from Step 2 directly,
while the estimate $\eqref{app:f:12}$ involves a little more
(in the case of $s<2$) compared with the estimate $\eqref{app:f:2}$.
Recalling the estimate $\eqref{app:f:10}$, a covering argument
leads to
\begin{equation*}
\Big(\int_{B(0,s)\cap\Omega}|\nabla w|^{\bar{p}}dx\Big)^{1/\bar{p}}
\leq \frac{C_0}{s^{\frac{
1}{2}}(t-s)^{\frac{1}{2}}}
\Big(\int_{B(0,t)\cap\Omega}
|\nabla w|^{2}dx\Big)^{1/2}
\end{equation*}
for any $0<s<t<1$. Then, applying the convexity argument
\cite[pp.173]{FS} to the above estimate we arrive at
\begin{equation}\label{app:f:13}
\Big(\int_{B(0,r_0)}|\nabla w|^2dx\Big)^{1/2}
\lesssim \Big(\int_{B(0,1)}|\nabla w|^{\underline{s}}dx
\Big)^{1/\underline{s}}
\end{equation}
for any $0<\underline{s}<2$, where $r_0\in(0,1)$ and the up to
constant depends $\underline{s},\bar{p}$ and $C_0$.
Thus, combining the estimates $\eqref{app:f:2}$ and
$\eqref{app:f:13}$ gives the
desired estimate $\eqref{app:f:12}$. Finally, repeating
the same arguments given for the good-$\lambda$
inequality $\eqref{f:9}$, we refer the reader to
the proof of Theorem $\ref{thm:1.4}$ and therefore omit these
details here. We have completed the whole proof.
\end{proof}

\begin{theorem}[layer $\&$ co-layer type estimates]
\label{app:thm:2}
Let $\Omega\subset\mathbb{R}^d$ be a bounded Lipschitz domain
and $0<\varepsilon\ll 1$.
Given $F\in L^2(\Omega;\mathbb{R}^d)$ and $g\in H^1(\partial\Omega;
\mathbb{R}^d)$, let $u_0$ be associated with $F$ and $g$ by
the equation $\eqref{pde:1.3}$, where
the coefficient $\widehat{A}$ satisfies the elasticity condition
$\eqref{a:2}$. Then there hold the following estimates.
\begin{enumerate}
  \item $L^p$ estimates:
  \begin{equation}\label{pri:9.3}
  \|\nabla u_0\|_{L^{\frac{2d}{d-1}}(\Omega)}
  \lesssim \Big\{\|F\|_{L^{\frac{2d}{d+1}}(\Omega)}
  +\|g\|_{H^1(\partial\Omega)}\Big\}.
  \end{equation}
  \item Radial maximal function estimates
  \begin{equation}\label{pri:9.4}
  \|\mathrm{M}_{\emph{r}}(\nabla u_0)\|_{L^{2}(\partial\Omega)}
  \lesssim
  \Big\{\|F\|_{L^{2}(\Omega)}
  +\|g\|_{H^1(\partial\Omega)}\Big\},
  \end{equation}
  where the operator $\mathrm{M}_{\emph{r}}$
  is defined by $\eqref{def:4}$.
  \item Layer type estimates
  \begin{equation}\label{pri:9.2}
\begin{aligned}
\|\nabla u_{0}\|_{L^{2}(O_{4\varepsilon})}
&\lesssim \varepsilon^{\frac{1}{2}}\Big\{\|g\|_{H^{1}(\partial\Omega)}
  +\|F\|_{L^{2}(\Omega)}\Big\};\\
\Big(\int_{O_{4\varepsilon}}
|\nabla u_{0}|^{2} \delta dx\Big)^{\frac{1}{2}}
&\lesssim \varepsilon
\Big\{\|g\|_{H^{1}(\partial\Omega)}
+\|F\|_{L^{2}(\Omega)}\Big\}.
\end{aligned}
\end{equation}
  \item Co-layer type estimates
   \begin{equation}\label{pri:9.1}
 \begin{aligned}
\|\nabla^2 u_{0}\|_{L^{2}(\Omega\backslash O_{4\varepsilon})}
&\lesssim \varepsilon^{-\frac{1}{2}}
\Big\{\|g\|_{H^{1}(\partial\Omega)}+
\|F\|_{L^{2}(\Omega)}\Big\};\\
\Big(\int_{\Omega\backslash O_{4\varepsilon}}
|\nabla^{2} u_{0}|^{2} \delta dx\Big)^{\frac{1}{2}}
&\lesssim \ln^{\frac{1}{2}}(1/\varepsilon)
\Big\{\|g\|_{H^{1}(\partial\Omega)}
+\|F\|_{L^{2}(\Omega)}\Big\};\\
\Big(\int_{\Omega\backslash O_{4\varepsilon}}
|\nabla u_{0}|^{2}\delta^{-1}dx\Big)^{\frac{1}{2}}
&\lesssim \ln^{\frac{1}{2}}(1/\varepsilon)
\Big\{\|g\|_{H^{1}(\partial\Omega)}
+\|F\|^2_{L^{2}(\Omega)}
\Big\}.
\end{aligned}
\end{equation}
\end{enumerate}
The up to constant depends only on
$\widehat{\mu}_{0}$, $\widehat{\mu}_{1}$, $d$ and
the character of $\Omega$.
\end{theorem}

\begin{remark}\label{app:remark:1}
To the authors' best knowledge,
the first line of the estimates $\eqref{pri:9.2}$
and $\eqref{pri:9.1}$ were originally investigated by Z. Shen
in \cite[Theorems 2.6]{S}, while the second author
generalized his arguments to the weighted type
estimates (see \cite[Lemma 4.5]{Q1}).
\end{remark}

\begin{proof}
Here we merely show the proofs of $\eqref{pri:9.3}$
and $\eqref{pri:9.4}$ by
using a similar idea given for the estimates $\eqref{pri:9.2}$
and $\eqref{pri:9.1}$, which is inspired by
\cite[Theorems 2.6]{S} as we have mentioned in Remark
$\ref{app:remark:1}$.

Firstly, we decompose $u_{0}$ into two parts: $v$ and $w$, which satisfy the following systems:
\begin{equation}\label{PDE1}
  -\nabla\cdot(\widehat{A}\nabla v)=\tilde{F} \qquad\text{in}
 \quad\mathbb{R}^d,
\end{equation}
and
\begin{equation}\label{PDE2}
\left\{\begin{aligned}
  -\nabla\cdot(\widehat{A}\nabla w)&=0 &\qquad&\text{in~}
  \Omega,\\
  w&=g-v&\qquad&\text{on~}\partial\Omega,
 \end{aligned} \right.
\end{equation}
respectively. Let $\tilde{F}$ be a zero-extension of
$F$, satisfying $\tilde{F}=F$ in $\Omega$ and $\tilde{F}=0$ in $\mathbb{R}^d\backslash \Omega$.
In terms of $\eqref{PDE1}$, by the well-known fractional integral estimates and singular integral estimates for $v$ (see for example \cite{Stein}), we have that
\begin{equation}\label{eq:2.15}
\begin{aligned}
\|\nabla v\|_{L^{p'}(\mathbb{R}^d)}+
\|\nabla^2 v\|_{L^p(\mathbb{R}^d)}\leq C\|\tilde{F}\|_{L^p(\mathbb{R}^d)} \qquad\text{for~}
1<p<d,
\quad\frac{1}{p'}=\frac{1}{p}-\frac{1}{d}.
\end{aligned}
\end{equation}
It follows from the divergence theorem that
\begin{equation}\label{eq:2.17}
\begin{aligned}
  \int_{\partial\Omega}|\nabla v|^2dS
  &\leq C\bigg[ \int_{\Omega}|\nabla v|^2dx+\int_{\Omega}|\nabla v||\nabla^2 v|dx\bigg]\\
  &\lesssim \|\nabla v\|^{2}_{L^{\frac{2d}{d-1}}(\Omega)}
  +\|\nabla^2 v\|^{2}_{L^{\frac{2d}{d+1}}(\Omega)},
  \end{aligned}
\end{equation}
where we use H\"{o}lder's inequality.
Then we turn to the equation $\eqref{PDE2}$.
It follows from the
nontangential maximal function estimates for
$L^2$-regular problem in Lipschitz domains (see \cite{DKV}) that
\begin{equation*}
\begin{aligned}
  \|(\nabla w)^{*}\|_{L^{2}(\partial\Omega)}&\leq C \bigg(\|\nabla_{\text{tan}}g\|_{L^{2}(\partial\Omega)}
  +\|\nabla_{\text{tan}}v\|_{L^{2}(\partial\Omega)}\bigg)\\
  &\lesssim\|g\|_{H^{1}(\partial\Omega)}
  +\|\nabla v\|_{L^{2}(\partial\Omega)},
  \end{aligned}
\end{equation*}
where $(\nabla w)^{*}$ denotes the nontangential maximal
function of $\nabla w$ (see $\eqref{def:5}$ for its definition).
On account of \eqref{eq:2.17}, we have
\begin{equation}\label{eq:2.19}
\begin{aligned}
\|(\nabla w)^{*}\|_{L^{2}(\partial\Omega)}&\lesssim
\|g\|_{H^{1}(\partial\Omega)}+\|\nabla v\|_{L^{\frac{2d}{d-1}}(\Omega)}
  +\|\nabla^2 v\|_{L^{\frac{2d}{d+1}}(\Omega)}\\
  &\lesssim^{\eqref{eq:2.15}}\|g\|_{H^{1}(\partial\Omega)}
  +\|F\|_{L^{\frac{2d}{d+1}}(\Omega)}.
\end{aligned}
\end{equation}

Secondly, for any $(u)^{*}\in L^2(\partial\Omega)$,
the fractional integral coupled
with a duality argument gives the following inequality
\begin{equation}\label{f:9.1}
\|u\|_{L^{\frac{2d}{d-1}}(\Omega)}
\lesssim \|(u)^{*}\|_{L^2(\partial\Omega)}
\end{equation}
(see for example \cite[Remark 9.3]{SZW4}). In terms of
the radial maximal function estimate, it follows from
\cite[Lemma 2.24]{Q2} that
\begin{equation}\label{f:9.2}
 \|\mathrm{M}_{\text{r}}(u)\|_{L^2(\partial\Omega)}
 \lesssim \|u\|_{H^1(\Omega)}.
\end{equation}

Consequently, based upon the decomposition $u_0 = v+w$ above,
there hold
\begin{equation*}
\begin{aligned}
\|\nabla u_0\|_{L^{\frac{2d}{d-1}}(\Omega)}
&\leq \|\nabla v\|_{L^{\frac{2d}{d-1}}(\Omega)}
+ \|\nabla w\|_{L^{\frac{2d}{d-1}}(\Omega)}\\
&\lesssim^{\eqref{eq:2.15},\eqref{f:9.1}}
\|F\|_{L^{\frac{2d}{d+1}}(\Omega)}
+ \|(\nabla w)^*\|_{L^{2}(\partial\Omega)}\\
&\lesssim^{\eqref{eq:2.19}}
\|F\|_{L^{\frac{2d}{d+1}}(\Omega)} +
\|g\|_{H^1(\partial\Omega)},
\end{aligned}
\end{equation*}
and
\begin{equation*}
\begin{aligned}
\|\mathrm{M}_{\text{r}}(\nabla u_0)\|_{L^{2}(\partial\Omega)}
&\leq \|\mathrm{M}_{\text{r}}(\nabla v)\|_{L^{2}(\partial\Omega)}
+
\|\mathrm{M}_{\text{r}}(\nabla w)\|_{L^{2}(\partial\Omega)}\\
&\lesssim^{\eqref{f:9.2}}
\|\nabla v\|_{H^{1}(\Omega)}
+ \|(\nabla w)^*\|_{L^{2}(\partial\Omega)}\\
&\lesssim^{\eqref{eq:2.15},\eqref{eq:2.19}}
\|F\|_{L^{2}(\Omega)} +
\|g\|_{H^1(\partial\Omega)},
\end{aligned}
\end{equation*}
where we note the fact that
$\mathrm{M}_{\text{r}}(\nabla w)(Q)
\leq (\nabla w)^*(Q)$ for a.e. $Q\in\partial\Omega$.
This ends the proof.
\end{proof}

\begin{remark}
If $\partial\Omega\in C^1$, then, for any $1<q<\infty$,
the solution $w$ to $\eqref{PDE2}$
owns the nontangential maximal function estimates
$\|(\nabla w)^*\|_{L^q(\partial\Omega)}
\lesssim \|\nabla w\|_{L^q(\partial\Omega)}$. Thus,
for $2\leq p<\frac{2d}{d-2}$, there holds
\begin{equation}\label{pri:9.5}
  \|\nabla u_0\|_{L^{p}(\Omega)}
  \lesssim \Big\{\|F\|_{L^{2}(\Omega)}
  +\|g\|_{W^{1,\bar{p}}(\partial\Omega)}\Big\},
\end{equation}
where $\bar{p}=p-p/d$, and the proof is as the same as
that given for $\eqref{pri:9.3}$, and so is omitted.
\end{remark}

\begin{center}
\textbf{Acknowledgements}
\end{center}

Some referees had read the previous version
(arXiv:2001.06874v1) and kindly pointed out
a fatal mistake therein. The authors were grateful for
this, since it started a real ``improvement''.
The first author acknowledged the hospitality when
she visited the Max Planck Institute for Mathematics
in the Sciences in the winter of 2019. The second author
deeply appreciated Prof. Felix Otto and his lectures,
as well as financial support in CIMI (Toulouse),
for the source of the idea of Theorem $\ref{thm:1.2}$. 
Also, 
the research was supported
by the Young Scientists Fund of the National Natural Science
Foundation of China (Grant No. 11901262), and
supported by the Fundamental Research
Funds for the Central Universities (Grant No.
lzujbky-2019-21).


\noindent Li Wang\\
School of Mathematics and Statistics, Lanzhou University,
Lanzhou, 710000, China.\\
E-mail:lwang10@lzu.edu.cn\\

\noindent Qiang Xu\\
Max Planck Institute for Mathematics in the Sciences,
Inselstrasse 22, 04103 Leipzig, Germany\\
E-mail:qiangxu@mis.mpg.de\\

\noindent Peihao Zhao\\
School of Mathematics and Statistics, Lanzhou University,
Lanzhou, 710000, China.\\
E-mail:zhaoph@lzu.edu.cn\\


\begin{thebibliography}{000}

\bibitem{APMP}
E.Acerbi, V. Piat, G. Maso, D. Percivale,
An extension theorem from connected sets, and homogenization in general periodic domains. Nonlinear Anal. 18(1992), no.5, 481-496.

\bibitem{GA} G. Allaire, Homogenization and two-scale
convergence. SIAM J. Math. Anal. 23(1992), no.6, 1482-1518.

\bibitem{AD}
S. Armstrong, J.-P. Daniel,
Calderon-Zygmund estimates for stochastic homogenization,
J. Funct. Anal.  270 (2016),  no.1, 312-329.

\bibitem{AKM}
S. Armstrong, T. Kuusi, J.-C. Mourrat,
The additive structure of elliptic homogenization,
Invent. Math.  208 (2017),  no.3, 999-1154.

\bibitem{AM}
S. Armstrong, J.-C. Mourrat, Lipschitz regularity for elliptic equations with random coefficients,
Arch. Ration. Mech. Anal.  219 (2016),  no.1, 255-348.

\bibitem{AS}
S. Armstrong, Z. Shen, Lipschitz estimates in
almost-periodic homogenization,
Comm. Pure Appl. Math. 69(2016),  no.10, 1882-1923.

\bibitem{AL} M. Avellaneda, F. Lin,
Compactness methods in the theory of homogenization,
Comm. Pure Appl. Math. 40, (1987) 803-847.

\bibitem{BHCD}
H. Bahouri, Hajer; J.-Y. Chemin, R. Danchin,
Fourier Analysis and Nonlinear Partial Differential Equations,
Grundlehren der Mathematischen Wissenschaften
[Fundamental Principles of Mathematical Sciences],
343. Springer, Heidelberg, 2011.

\bibitem{BF}
F. Boyer,  P. Fabrie,
Mathematical tools for the study of the incompressible
Navier-Stokes equations and related models,
Applied Mathematical Sciences,
183. Springer, New York, 2013.

\bibitem{B}
R. Brown, The mixed problem for Laplace's equation
in a class of Lipschitz domains,
Comm. Partial Differential Equations
19 (1994), no.7-8, 1217-1233.

\bibitem{BW}
S.-S. Byun, L. Wang,
Elliptic equations with BMO coefficients in Reifenberg domains,
Comm. Pure Appl. Math.  57 (2004),  no.10, 1283-1310.

\bibitem{CP} L. Caffarelli, I. Peral, On $W^{1,p}$ estimates for elliptic equations in divergence form,
Comm. Pure Appl. Math. 51, 1-21 (1998).

\bibitem{C}
S.-K. Chua, Extension theorems on weighted Sobolev spaces,
Indiana Univ. Math. J.  41(1992), no.4, 1027-1076.

\bibitem{CDG} D. Cioranescu, A. Damlamian, G. Griso, The periodic unfolding method. Theory and applications to partial differential problems. Series in Contemporary Mathematics,3. Springer, Singapore, 2018.

\bibitem{CP} D. Cioranescu, J. Paulin,
Homogenization in open sets with holes,
J. Math. Anal. Appl. 71 (1979), no.2, 590-607.

\bibitem{DKV} B. Dahlberg, C. Kenig, G. Verchota,
Boundary value problems for the systems of elastostatics
in Lipschitz domains. Duke Math. J. 57 (1988), no.3, 795-818.

\bibitem{D}
J. Duoandikoetxea,  Fourier Analysis,
translated and revised from the 1995 Spanish original by David Cruz-Uribe, Graduate Studies in Mathematics, 29.
American Mathematical Society, Providence, RI, 2001.

\bibitem{DO}
M. Duerinckx, F. Otto,
Higher-order pathwise theory of flucuations in
stochastic homogenization,
Stoch PDE: Anal Comp (2019)
(https://doi.org/10.1007/s40072-019-00156-4).

\bibitem{DST}
R.-G. Duran, M. Sanmartino,  M. Toschi,
Weighted a priori estimates for the Poisson equation,
Indiana Univ. Math. J. 57 (2008), no.7, 3463-3478.

\bibitem{LCE1} L. Evans, R. Gariepy,
Measure Theory and Fine Properties of Functions (Rewised Edition),
CRC Press, Boca Raton, 2015.

\bibitem{FS}
C. Fefferman, E. Stein, $H^p$ spaces of several variables,
Acta Math. 129(1972),  no.3-4, 137-193.

\bibitem{Gloria_Neukamm_Otto_2015}
A. Gloria, S. Neukamm, F. Otto,
\newblock A regularity theory for random elliptic operators,
\newblock {\em ArXiv e-print arXiv:1409.2678v3}, 2015.

\bibitem{DPV}
E. Di Nezza, G. Palatucci, E. Valdinoci,
Hitchhiker's guide to the fractional Sobolev spaces,
Bull. Sci. Math. 136 (2012), no. 5, 521-573.


\bibitem{GSS}
J. Geng, Z. Shen, L. Song, Boundary Korn inequality and Neumann problems in homogenization of systems of elasticity. Arch. Ration. Mech. Anal. 224 (2017), no.3, 1205-1236.

\bibitem{GSS1}
J. Geng, Z. Shen, L. Song, Uniform $W^{1,p}$ estimates
for systems of linear elasticity in a periodic medium,
J. Funct. Anal.  262(2012),  no.4, 1742-1758.

\bibitem{JKO} V. Jikov, S. Kozlov, O. Oleinik,
Homogenization of differential operators and integral functionals. (English summary)
Translated from the Russian by G. A. Yosifian
[G.A. Iosif'yan]. Springer-Verlag, Berlin, 1994.

\bibitem{JO}
M. Josien, F. Otto,
The annealed Calderon-Zygmund estimates
as convenient tool in quantitative schochastic
homogenization, arXiv:2005.08811v1.

\bibitem{KLS} C. Kenig, F. Lin, Z. Shen,
Convergence rates in $L^{2}$ for elliptic homogenization
problems, Arch. Ration. Mech. Anal. 203(2012), no. 3, 1009-1036.

\bibitem{SZW4} C. Kenig, F. Lin, Z. Shen,
Homogenization of elliptic systems with Neumann boundary conditions,
J. Amer. Math. Soc. 26(2013), no.4, 901-937.

\bibitem{SZW24} C. Kenig, Z. Shen, Layer potential methods for elliptic homogenization problems,
Comm. Pure Appl. Math. 64(2011), no.1, 1-44.

\bibitem{Le}
J. Lehrb\"ack, Weighted Hardy inequalities beyond Lipschitz domains,
Proc. Amer. Math. Soc. 142 (2014), no.5, 1705-1715.

\bibitem{LV}
J. Lehrb\"ack, A. V\"ah\"akangas, In between the inequalities of Sobolev and Hardy,
J. Funct. Anal. 271 (2016), no.2, 330-364.

\bibitem{L}
A. Lunardi,  Interpolation Theory,
Third edition,
Appunti. Scuola Normale Superiore di Pisa (Nuova Serie),
16. Edizioni della Normale, Pisa, 2018.

\bibitem{N}
J. Necas, Sur une methode pour resoudre
les equations aux dérivées partielles du type elliptique, voisine de la variationnelle,
Ann. Scuola Norm. Sup. Pisa Cl. Sci. (3) 16 (1962), 305-326
(French).

\bibitem{OSY} O. Oleinik, A. Shamaev,  G. Yosifian,
Mathematical problems in elasticity and homogenization.
Studies in Mathematics and its Applications, 26. North-Holland Publishing Co., Amsterdam, 1992.

\bibitem{BR}
B. Russell, Homogenization in perforated domains and
interior Lipschitz estimates,
J. Differential Equations 263 (2017), no.6, 3396-3418.

\bibitem{S0} Z. Shen,
Periodic homogenization of elliptic systems.
Operator Theory: Advances and Applications,
269. Advances in Partial Differential Equations (Basel).
Birkhäuser/Springer, Cham, 2018.

\bibitem{S} Z. Shen, Boundary estimates in elliptic homogenization,
Anal. PDE 10, 653-694 (2017).

\bibitem{S4} Z. Shen,
Weighted $L^2$ estimates for elliptic
homogenization in Lipschitz domains,
arXiv:2004.03087v1.

\bibitem{S3} Z. Shen, Bounds of Riesz transforms on
$L^p$ spaces for second order elliptic operators,
 Ann. Inst. Fourier (Grenoble) 55, 173-197 (2005).

\bibitem{SZ1}
Z. Shen, J. Zhuge, Boundary
layers in periodic homogenization of Neumann problems,
Comm. Pure Appl. Math.  71  (2018),  no. 11, 2163–2219.

\bibitem{SZ}
Z. Shen, J. Zhuge, Convergence rates in periodic homogenization
of systems of elasticity. Proc. Amer. Math. Soc. 145 (2017),
no.3, 1187-1202.

\bibitem{Stein} E. Stein,
Singular Integrals and Differentiability Properties
of Functions, Princeton Mathematical Series,
No.30, Princeton University Press, 1970.

\bibitem{TS} T. Suslina, Homogenization
of the Neumann problem for elliptic systems
with periodic coefficients.(English summary)
SIAM J. Math. Anal. 45(2013), no.6, 3453-3493.

\bibitem{TS1}
T. Suslina, Homogenization of the Dirichlet problem for
elliptic systems: $L^2$-operator error estimates,
Mathematika 59 (2013), no.2, 463-476.

\bibitem{WXZ}
L. Wang, Q. Xu, P. Zhao,
Quantitative estimates for homogenization of nonlinear elliptic operators in perforated domains, Preprinted in arXiv (2020).

\bibitem{Q1} Q. Xu, Convergence rates for general
elliptic homogenization problems in Lipschitz domains,
SIAM J. Math. Anal. 48(2016), no.6, 3742-3788.

\bibitem{Q2} Q. Xu,
Uniform regularity estimates in homogenization theory of elliptic
systems with lower order terms on the Neumann boundary problem,
J. Differential Equations 261(2016), no.8, 4368-4423.

\bibitem{ZR}
V. Zhikov, M. Rychago, Homogenization of nonlinear elliptic equations of the second
order in perforated domains, Izv. Ross. Akad. Nauk, Ser. Mat 61, 69-89 (1997).

\end{thebibliography}
\end{document}